\documentclass[12pt]{amsart}
\usepackage{amsmath}
\usepackage{amsxtra}
\usepackage{amscd}
\usepackage{amsthm}
\usepackage{amsfonts}
\usepackage{amssymb}
\usepackage {pstricks}
\usepackage{pstricks,pst-node}
\usepackage[all]{xy}

\newtheorem{theorem}{Theorem}[section]

\newtheorem{lemma}[theorem]{Lemma}

\newtheorem{corollary}[theorem]{Corollary}
\newtheorem{proposition}[theorem]{Proposition}

\theoremstyle{definition}
\newtheorem{definition}[theorem]{Definition}

\newtheorem{scholium}[theorem]{Scholium}

\theoremstyle{remark}
\newtheorem{remark}[theorem]{Remark}

\numberwithin{equation}{section}

\newcommand{\inv}{{^{-1}}}
\newcommand{\be}{\begin{equation}}
\newcommand{\ee}{\end{equation}}
\newcommand{\nc}{\newcommand}
\nc{\mc}{\mathcal}
\nc{\mf}{\mathfrak}

\nc{\on}{\operatorname}

\nc{\C}{{\mathbb C}}
\nc{\R}{{\mathbb R}}
\nc{\Z}{{\mathbb Z}}
\nc{\N}{{\mathbb N}}
\nc{\F}{{\mathbb F}}
\nc{\bbA}{{\mathbb A}}
\nc{\bbU}{{\mathbb U}}

\nc{\CA}{{\mathcal A}}
\nc{\CE}{{\mathcal E}}
\nc{\CP}{{\mathcal P}}
\nc{\CO}{{\mathcal O}}
\nc{\CK}{{\mathcal K}}
\nc{\CT}{{\mathcal T}}

\nc{\Ann}{{\rm{Ann}}}
\nc{\Rad}{{\rm{Rad}}}
\nc{\Res}{{\rm{Res}}}
\nc{\Ind}{{\rm{Ind}}}
\nc{\Ker}{{\rm{Ker}}}
\nc{\id}{{\rm{id}}}

\nc{\fg}{\mf g}
\nc{\fh}{\mf h}
\nc{\fb}{\mf b}
\nc{\fn}{\mf n}
\nc{\fk}{{\mathfrak k}}
\nc{\fl}{{\mathfrak l}}
\nc{\fp}{{\mathfrak p}}
\nc{\fu}{{\mathfrak u}}

\nc{\cB}{\mc B}
\nc{\cC}{\mc C}
\nc{\cD}{\mc D}
\nc{\cE}{\mc E}
\nc{\cH}{\mc H}
\nc{\cI}{\mc I}
\nc{\cR}{\mc R}
\nc{\cS}{\mc S}
\nc{\cT}{\mc T}
\nc{\cV}{\mc V}

\nc{\half}{\frac{1}{2}}

\nc{\dt}{\mathord{\hbox{${\frac{d}{d t}}$}}}

\newcommand{\End}{{\rm{End}}}
\newcommand{\Hom}{{\rm{Hom}}}

\newcommand{\Uq}{{{\rm U}_q}}
\newcommand{\GL}{{\rm{GL}}}
\newcommand{\Sym}{{\rm{Sym}}}

\newcommand{\Tr}{{\rm{Tr}}}

\newcommand{\rank}{{\rm{rank}}}

\newcommand{\cp}{{\rm{char(K)}}}

\advance\headheight by 2pt

\nc{\fsl}{{\mathfrak {sl}}}
\nc{\fsp}{{\mathfrak {sp}}}
\nc{\fso}{{\mathfrak {so}}}
\nc{\fgl}{{\mathfrak {gl}}}
\nc{\fo}{{\mathfrak {o}}}

\nc{\Sp}{{\rm Sp}}
\nc{\Or}{{\rm O}}
\nc{\BMW}{{\rm BMW}}
\nc{\ot}{\otimes}
\nc{\ve}{\varepsilon}
\newcommand{\ep}{{\epsilon}}
\nc{\lr}{\longrightarrow}

\nc{\Sdim}{{\rm Sdim}}

\begin{document}

\normalfont

\title[Brauer category and invariant theory]{The Brauer category
and invariant theory.}

\author{G.I. Lehrer and R.B. Zhang}
\thanks{This research was supported by the Australian Research Council}
\address{School of Mathematics and Statistics,
University of Sydney, N.S.W. 2006, Australia}
\email{gustav.lehrer@sydney.edu.au, ruibin.zhang@sydney.edu.au}
\date {24 July, 2012}

\begin{abstract}
A category of Brauer diagrams, analogous to Turaev\rq{}s tangle category,
is introduced, and a presentation of the category
is given; specifically, we prove that seven relations among its four generating
homomorphisms suffice to deduce all equations among the
morphisms. Full tensor functors are constructed from this category to
the category of tensor representations of the orthogonal group $\Or(V)$ or the symplectic group $\Sp(V)$
over any field of characteristic zero. The first and second fundamental theorems of invariant theory
for these classical groups are generalised to the category theoretic setting.
The major outcome is that we obtain new presentations for the endomorphism algebras of the module
$V^{\otimes r}$. These are obtained by appending
to the standard presentation of the Brauer algebra of degree $r$ one additional relation.
This relation stipulates the vanishing of an element of the Brauer algebra which is quasi-idempotent,
and which we describe explicitly both in terms of diagrams and algebraically.
In the symplectic case, if $\dim V=2n$,
the element is precisely the central idempotent in the Brauer
subalgebra of degree $n+1$, which corresponds to its trivial representation. Since this is
the Brauer algebra of highest degree which is semisimple, our generator is an exact
analogue for the Brauer algebra of the Jones idempotent of the Temperley-Lieb algebra.
In the orthogonal case the additional relation is also a quasi-idempotent in the integral Brauer algebra.
Both integral and quantum analogues of these results are given, the latter of which involve the BMW algebras.
\end{abstract}
\maketitle

\tableofcontents

\section{Introduction}
The fundamental theorems \cite{W} of classical invariant theory are concerned with generators
and relations for invariants of classical group actions, and can be formulated in different ways \cite{GW}.
A linear formulation \cite{W} of the first and second fundamental theorems
describes a spanning set of the vector space of invariant linear functionals on tensor modules,
and all the linear relations among the elements of this set.
There is also a commutative algebraic formulation which describes the invariants of
classical group actions on the coordinate ring of an appropriate module \cite{GW}.
The fundamental theorems in this case give a presentation of the algebra of invariant functions
as a commutative algebra.
The two formulations are equivalent.

Another formulation, which is more frequently encountered in representation theory,
is in terms of the (non-commutative) endomorphism algebras of tensor modules.
The first fundamental theorem (FFT) in this formulation \cite{GW}
describes the endomorphism algebra as the homomorphic image of
some known algebra, which is the group algebra of the symmetric group in
the case of the general linear group following Schur,
and the Brauer algebra \cite{Br} with appropriate parameters in the case of the orthogonal or symplectic group
by work of Brauer.
However, except in type $A$ ($\GL_n$) there does not seem to exist a standard form of the
second fundamental theorem (SFT) in this formulation.
A reasonable expectation is that the SFT should provide convenient presentations for
these endomorphism algebras,
which cannot be deduced from the other two formulations of the SFT in any easy way,
except in the case of the general linear group. Since there is a large (non-commutative)
algebra of endomorphisms, one might expect that there should be
only a small number of relations necessary to generate the ideal of all relations, other
than the ``Brauer relations''. This does indeed turn out to be the case, with a single
explicitly described idempotent generating all additional relations as an ideal of the Brauer algebra.

These results permit an analysis of
the generic quantum case, which we present, and should lead to results for
quantum groups at roots of unity, and for
 the case where the base field has positive characteristic.
 This is because our generating elements in both cases are sums of
 diagrams with coefficients $\pm 1$.

In \cite{LZ2, LZ3}, the orthogonal group $\Or(V)$ over $\C$ with $\dim V=3$ was investigated
(together with its quantum analogue at generic $q$).
We obtained a single idempotent $E$ in the Brauer algebra of degree $r\ge 4$, which generates
a two-sided ideal that is equal to the kernel of the algebra homomorphism from the Brauer algebra
to the endomorphism algebra $\End_{\Or(V)}(V^{\ot r})$ (the kernel is trivial if $r<4$). Thus
$\End_{\Or(V)}(V^{\ot r})$ can be presented in terms of the standard generators and relations
of the Brauer algebra with the single additional relation $E=0$.

Remarkably, the situation has turned out to be the same
for all the orthogonal groups \cite{LZ4} and symplectic groups \cite{HX}
over any field $K$ of characteristic zero.
The methods used in the papers \cite{LZ2, LZ3, LZ4} and \cite{HX} are quite different.
In \cite{LZ2, LZ3}, we analysed the radical of the Brauer algebra to prove our result,
making extensive use of the theory of cellular algebras \cite{GL96, GL03, GL04}.
The paper \cite{HX} relied on results on the detailed structure and representations
 \cite{DHW, HW, RS, X} of the Brauer algebra and BMW algebra \cite{BW}.
In particular, it made essential use of a series of earlier papers of Hu and collaborators.
In contrast, invariant theory featured much more prominently in \cite{LZ4}.

In the present paper we give a unified treatment of the SFTs for all the orthogonal
and symplectic groups in the endomorphism algebra formulation,  following
a categorical approach inspired by works on quantum invariants of links \cite{J, T1, R, RT,
ZGB}.

Recall that a key algebraic result in quantum topology is that the category of
tangles is a strict monoidal category with braiding \cite{FY1, FY2, T1} (also see \cite{RT, T2})
in the sense of Joyal and Street \cite{JS}.
The set of objects of this category is $\N=\{0, 1, 2, \dots\}$, and the vector spaces
of morphisms have bases consisting of non-isotopic tangle diagrams.
We define a similar, but much simpler category $\cB(\delta)$,
the {\em category of Brauer diagrams}
with parameter $\delta \in K$.  The space of morphisms of $\cB(\delta)$ is spanned by
{\em Brauer diagrams} (see Definition \ref{def:brauer-diag}),
which include the usual Brauer diagrams \cite{Br} as a special case, as endomorphisms
of an object of the category.

Let $G$ be either the orthogonal group $\Or(V)$ or the symplectic group $\Sp(V)$,
and denote by $\cT_G(V)$ the full subcategory of the category
of finite dimensional $G$-representations
with objects $V^{\ot r}$ ($r\in \N$).
There exists an additive functor
$F: \cB(\epsilon m)\longrightarrow \cT_G(V)$, which is given by Theorem \ref{thm:functor}.
Here $\epsilon m=\epsilon(G)\dim V$
with $\epsilon(G)=1$ for $\Or(V)$ and $-1$ for $\Sp(V)$.
The functor $F$ is shown to be full in Theorem \ref{thm:fft-sft}(1).
This significantly generalises the FFTs for the orthogonal and symplectic groups.
Both the linear and endomorphism algebra versions of FFT are now special
cases of Theorem \ref{thm:fft-sft}(1),
and their equivalence becomes entirely transparent.

For each pair of objects $r, s$ in the category $\cB(\epsilon m)$,
the functor $F$ induces a linear map
$F_r^s: \Hom_{\cB(\epsilon m)}(r, s)\longrightarrow\Hom_{\cT_G(V)}(V^{\ot r}, V^{\ot s})$.
A simple description of the subspace $\Ker{F}_r^s$ is obtained in Theorem \ref{thm:fft-sft}(2),
which contains the linear version of SFT as a special case.

When $s=r$, the domain of $F_r^r$ is the Brauer algebra of degree $r$,  the range
is the endomorphism algebra $\End_{\cT_G(V)}(V^{\ot r})$,
and the map  is an algebra homomorphism. In this case,
we want to understand the algebraic structure
of the kernel of the map $F_r^r$.

We explicitly construct an element
in the Brauer algebra which generates $\Ker{F}_r^r$ as a two-sided ideal
($\Ker{F}_r^r\ne 0$ only when $r>d$, see Theorem \ref{thm:sft}).
The result for the symplectic group is given in Theorem \ref{thm:sp-main},
and that for the orthogonal group in Theorem \ref{thm:o-main}.
This leads to a presentation of $\End_{\cT_G(V)}(V^{\ot r})$
upon imposing the condition that this element vanishes.
In the case of $\Or(V)$, the generating element we obtain is shown to be equal to
that obtained in \cite{LZ4}.
In the symplectic case, the element of \cite{HX} is a scalar multiple of the one obtained here
(Remark \ref{rem:uniqueness-Phi}). However our approach yields an explicit formula for the
element, both in terms of generators and relations, and in terms of diagrams; moreover we show that
the element is (a multiple of) the central idempotent corresponding to the trivial
representation of the Brauer algebra on $n+1$ strings, if $r=2n$. We note that $B_r(-2n)$ is semisimple
if and only if $r\leq n+1$ (see \S \ref{s:ss} below). Thus our generating element is an exact analogue
of Jones\rq{} `augmentation\rq{} idempotent \cite{J, GL98}.

We remark that notwithstanding the fact that convenient formulae for our generating elements
involve rational numbers with large denominators, the elements are actually sums of diagrams with coefficients
$\pm 1$. This permits reduction modulo primes, and an approach to the case of positive characteristic
(\S \ref{s:ss}).

The category of Brauer tangle diagrams
provides an appropriate framework for uniformly treating
the SFTs of the orthogonal and symplectic groups
in the endomorphism algebra formulation because to move between
the linear, commutative algebraic and endomorphism algebra formulations,
we need to consider arbitrary Brauer diagrams, not only those in the Brauer
algebras.

The categorical framework is also the most natural
setting for studying the invariant theory of quantum groups \cite{D, L}.
In Section \ref{sect:quantum} we present some generalisations of our results
to the quantum case, where we show that similar results hold, with the
Brauer algebras replaced by the Birman-Murakami-Wenzl algebras.


\section{The category of Brauer diagrams}
We begin with a discussion on Brauer diagrams, which could be thought of as a highly
simplified version of the tangle diagrams of \cite{FY1, FY2, T1} (also see \cite{RT, T2}).
Tangles in this paper are neither oriented nor framed. In fact we shall find it easier to
work with the (equivalent) category of Brauer diagrams, with no reference to tangles.

\subsection{The category of Brauer diagrams}\label{subsect:tangles}

Let $\N=\{0, 1, 2, \dots\}$.
\begin{definition}\label{def:brauer-diag}
For any pair $k, \ell \in\N$,
a $(k, \ell)$ (Brauer) diagram, (or Brauer diagram from $k$ to $\ell$)
is a partitioning of the set $\{1,2,\dots,k+\ell\}$ as a disjoint union of pairs.
\end{definition}
This is thought of as a diagram where $k+\ell$ points (the nodes, or vertices) are placed on
two parallel horizontal lines,
$k$ on the lower line and $\ell$ on the upper, with
arcs drawn to join points which are paired. We shall speak of the lower and upper nodes or vertices
of a diagram. The pairs will be known as {\em arcs}. If $k=\ell=0$, there is by convention just one
Brauer $(0,0)$-diagram.

Figure \ref{fig1} below is a $(6, 4)$ Brauer diagram.

\begin{figure}[h]
\begin{center}
\begin{picture}(160, 60)(0,0)

\qbezier(0, 60)(30, 0)(60, 60)
\qbezier(0, 0)(60, 60)(90, 0)

\qbezier(30, 60)(90, 40)(150, 0)
\qbezier(90, 60)(60, 30)(30, 0)

\qbezier(60, 0)(90, 40)(120, 0)
\end{picture}
\end{center}
\caption{ }
\label{fig1}
\end{figure}

\begin{remark}\label{rem:tangles}
Such a diagram may be thought of as the image of a tangle diagram (i.e. ambient isotopy class of
$(k,\ell)$ tangles) under projection to a plane. It is straightforward that if overcrossings
and undercrossings are identified in a tangle projection, the only invariants of a tangle
are the number of free loops and the set of pairs of boundary points, each of which is
the boundary of a connected component of the tangle. Hence the identification with Brauer
diagrams. We shall therefore not use tangles explicitly.
\end{remark}

There are two operations on Brauer diagrams: {\em composition} defined using
concatenation of diagrams and {\em tensor product} defined using juxtaposition
(see below).

\begin{definition}\label{def:bkl} Let $K$ be a commutative ring with identity,
and fix $\delta\in K$.
Denote by ${B}_k^\ell(\delta)$ the free
$K$-module with a basis consisting of
$(k, \ell)$ Brauer diagrams. Note that
${B}_k^\ell(\delta)\neq 0$ if and only if $k+\ell$ is even,
since the free $K$-module with basis the empty set is zero.
By convention there is one diagram in $B_0^0(\delta)$, viz. the empty diagram.
Thus $B_0^0(\delta)=K$.
\end{definition}

There are two $K$-bilinear operations on diagrams.
\begin{eqnarray}\label{eq:products}
\begin{aligned}
&&\text{composition} & \quad \circ:  & & B_\ell^p(\delta)
\times B_k^\ell(\delta)\longrightarrow B_k^p(\delta), and  \\
&&\text{tensor product} & \quad  \otimes: & & B_p^q(\delta)
\times B_k^\ell(\delta)\longrightarrow B_{k+p}^{q+\ell}(\delta)
\end{aligned}
\end{eqnarray}

These operations are defined as follows.
\begin{enumerate}
\item The composite $D_1\circ D_2$ of the Brauer diagrams $D_1\in B_\ell^p(\delta)$ and
$D_2\in B_k^\ell(\delta)$ is defined as follows.
First, the concatenation $D_1\#D_2$ is obtained by placing $D_1$ above $D_2$,
and identifying the $\ell$ lower nodes of $D_1$ with the corresponding upper nodes of $D_2$.
Then $D_1\#D_2$ is the union of a Brauer $(k,p)$ diagram $D$ with a certain number, $f(D_1,D_2)$ say,
of free loops. The composite $D_1\circ D_2$ is the element $\delta^{f(D_1,D_2)}D\in B_k^p(\delta)$.

\item The tensor product $D\otimes D'$ of
any two Brauer diagrams $D\in B_p^q(\delta)$
and $D'\in B_k^l(\delta)$ is the $(p+k, q+l)$
diagram obtained by juxtaposition, that is,
placing $D'$ on the right of $D$ without overlapping.
\end{enumerate}
Both operations are clearly associative.

\begin{definition}
The {\em category  of Brauer diagrams}, denoted by $\cB(\delta)$,
is the following pre-additive small category equipped with a bi-functor $\otimes$
(which will be called the tensor product):
\begin{enumerate}
\item the set of objects is $\N=\{0, 1, 2, \dots\}$,  and for any pair of objects $k, l$,
$\Hom_{\cB(\delta)}(k, l)$ is the $K$-module $B_k^l(\delta)$; the composition of morphisms is
given by the composition of Brauer diagrams defined by \eqref{eq:products};
\item the tensor product $k\otimes l$ of objects  $k, l$ is  $k+l$ in $\N$, and the
tensor product of morphisms is given by the tensor product of Brauer diagrams of \eqref{eq:products}.
\end{enumerate}
\end{definition}
It follows from the associativity of composition of Brauer diagrams that
$\cB(\delta)$ is indeed a pre-additive category.

\begin{remark}\label{rem:quotient-cat}
The operations in $\cB(\delta)$ mirror the operations in the tangle category
considered in \cite{FY1, FY2, T1, RT, T2} and
the {\em category of Brauer diagrams} is a quotient category
of the category of tangles in the sense of \cite[\S II.8]{MacL}.
\end{remark}

\subsection{Involutions}\label{ssec:involutions}
The category $\cB(\delta)$ has a {\em duality functor} $^*:\cB(\delta)\to \cB(\delta)^{\text{op}}$,
which takes each object to itself, and takes each diagram to its reflection in a horizontal line.
More formally, for any $(k,\ell)$ diagram $D$, $D^*$ is the $(\ell,k)$ diagram with precisely the same
pairs identified as $D$. Further, there is an involution $^\sharp:\cB(\delta)\to \cB(\delta)$
which also takes objects to themselves, but takes a diagram $D$ to its reflection in a vertical line.
Formally, if the upper nodes of the diagram $D$ are labelled $1,2.\dots,\ell$ and the lower
nodes are labelled $1',2',\dots,k'$, we apply the permutation $i\mapsto \ell+1-i,j'\mapsto k+1-j'$
to the nodes to get the arcs of $D^\sharp$. We shall meet the contravariant functor
$D\mapsto *D:=D^{*\circ\sharp}$ later.

It is easily checked that $(D_1\circ D_2)^*=D_2^*\circ D_1^*$, $(D_1\ot D_2)^*=D_1^*\ot D_2^*$,
and that
$(D_1\circ D_2)^\sharp=D_1^\sharp\circ D_2^\sharp$ and $(D_1\ot D_2)^\sharp=D_2^\sharp\ot D_1^\sharp$.

\subsection{Generators and relations}
Generators and relations for tangle diagrams were described in
\cite{FY1, FY2, T1, RT, T2}. The corresponding result for Brauer diagrams
is the main result of this section.

\begin{theorem}\label{thm:presentation}
\begin{enumerate}
\item\label{generators}
The four Brauer diagrams
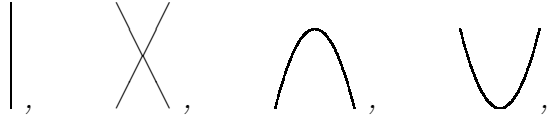
\begin{center}
\begin{picture}(205, 40)(-5,0)
\put(0, 0){\line(0, 1){40}}
\put(5, 0){,}

\put(40, 0){\line(1, 2){20}}
\put(60, 0){\line(-1, 2){20}}
\put(65, 0){,}

\qbezier(100, 0)(115, 60)(130, 0)
\put(135, 0){,}

\qbezier(170, 30)(185, -30)(200, 30)
\put(200, 0){,}
\end{picture}
\end{center}
generate all Brauer diagrams by composition and tensor product (i.e., juxtaposition).
We shall refer to these generators as the elementary Brauer diagrams,
and denote them by $I$, $X$, $A$ and $U$
respectively. Note that
these diagrams are all fixed by $^\sharp$, and that $^*$ fixes $I$ and $X$, while
$A^*=U$ and $U^*=A$.
\item \label{relations} A complete set of relations among
these four generators is given by the following, and their transforms under $^*$ and $^\sharp$. This means that any
equation relating two words in these four generators can be deduced from the given relations.
\begin{eqnarray}
&I\circ I= I,\:(I\ot I)\circ X= X, \;(I\ot I)\circ A=A,\;(I\ot I)\circ U=U, \label{eq:identity}\\
&X\circ X= I,\label{eq:XX} \\
&(X\ot I)\circ(I\ot X)\circ(X\ot I) = (I\ot X)\circ(X\ot I)\circ(I\ot X),
\label{eq:braid}\\
&A\circ X = A,\label{eq:AX}\\
&A\circ U = \delta, \label{eq:AU}\\
&(A\ot I)\circ(I\ot X)=(I\ot A)\circ(X\ot I)\label{eq:slide}\\
&(A\ot I)\circ(I\ot U)=I.\label{eq:straight}
\end{eqnarray}
\end{enumerate}
The relations \eqref{eq:XX}-\eqref{eq:straight} are depicted diagrammatically in Figures \ref{fig:braid},
\ref{fig:loop} and \ref{fig:slide}.
\end{theorem}

\begin{figure}[h]
\begin{center}
\begin{picture}(100, 70)(90,-10)
\qbezier(0, 60)(50, 30)(0, 0)
\qbezier(30, 60)(-10, 30)(30, 0)

\put(40, 30){$=$}

\qbezier(60, 60)(60, 30)(60, 0)
\qbezier(90, 60)(90, 30)(90, 0)
\put(95, 0){;}

\put(10, -20){Double crossing}

\qbezier(150, 60)(130, 30)(150, 0)
\qbezier(120, 60)(170, 40)(180, 0)
\qbezier(120, 0)(170, 20)(180, 60)

\put(190, 30){$=$}

\qbezier(240, 60)(260, 30)(240, 0)
\qbezier(270, 0)(220, 20)(210, 60)
\qbezier(270, 60)(220, 40)(210, 0)

\put(180, -20){Braid relation}
\end{picture}
\end{center}
\caption{Relations \eqref{eq:XX} and \eqref{eq:braid}}
\label{fig:braid}
\end{figure}


\begin{figure}[h]
\begin{center}
\begin{picture}(100, 80)(40,-10)
\qbezier(0, 40)(15, 90)(30, 40)
\qbezier(0, 40)(3, 15)(30, 0)
\qbezier(30, 40)(27, 15)(0, 0)
\put(40, 30){$=$}
\qbezier(60, 10)(75, 90)(90, 10)
\put(95, 0){;}

\put(30, -20){De-looping}
\qbezier(120, 33)(135, 90)(150, 33)
\qbezier(120, 33)(135, -20)(150, 33)
\put(160, 30){$=\delta$}

\put(115, -20){Loop Removal}

\end{picture}
\end{center}
\caption{Relations \eqref{eq:AX} and \eqref{eq:AU}}
\label{fig:loop}
\end{figure}

\begin{figure}[h]
\begin{center}
\begin{picture}(100, 80)(50,-10)
\qbezier(0, 0)(5, 120)(40,0)
\qbezier(20, 0)(30, 30)(40,60)
\put(45, 30){$=$}
\qbezier(60, 0)(95, 120)(100,0)
\qbezier(60, 60)(70, 30)(80,0)

\put(105, 0){;}

\put(30, -20){Sliding}
\qbezier(130, 0)(135, 90)(150, 30)
\qbezier(150, 30)(165, -30)(170, 60)
\put(180, 30){$=$}
\qbezier(200, 0)(200, 30)(200, 60)

\put(130, -20){Straightening}

\end{picture}
\end{center}
\caption{Relations \eqref{eq:slide} and \eqref{eq:straight}}
\label{fig:slide}
\end{figure}

\begin{proof} We first prove (1). The fact that the elementary Brauer diagrams $I,X$, $A$ and $U$
generate all Brauer diagrams under the operations of $\circ$ and $\ot$ may be seen as follows.
Fix the nodes of an arbitrary diagram $D$ from $k$ to $\ell$, and draw all the arcs as
piecewise smooth curves, in such a way
that there are at most two arcs through any point, and that no two crossings
or turning points have the same vertical coordinate. We may now draw a set of horizontal lines
(possibly after a small perturbation of the diagram) such that

(i) each line is not tangent to any of the arcs

(ii) between successive lines there is precisely one crossing or turning point.

Then the part of the diagram between successive lines may be thought of as the $\ot$-product
of the four generators, all except one being equal to $I$.
Thus we have exhibited $D$ as a word in the generators, of the form
$D=D_1\circ D_2\circ\dots\circ D_n$, where each $D_i$ is of the form
\be\label{eq:elem}
D_i=I^{\ot r}\ot Y\ot I^{\ot s},
\ee
with $Y$ being one of $A,U$ or $X$. Such an expression will be called
a {\em regular expression}, and the
factors $D_i$ {\em elementary diagrams}.
A product of elementary diagrams
in which $Y=X$ for each factor will be called a {\em permutation diagram}.
An example of a particular regular expressions is given in Figure \ref{fig:Ts}.

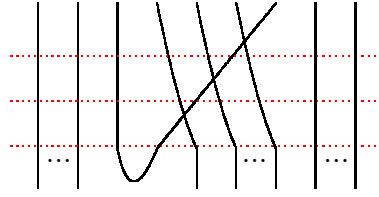
\begin{figure}[h]
\begin{center}
\begin{picture}(130, 70)(-10,0)
\qbezier(0, 70)(0, 30)(0, 0)
\put(3, 10){$...$}
\qbezier(15,70)(15, 30)(15, 0)

{\color{red}\qbezier[60](-10, 16)(56, 16)(130, 16)}

\qbezier(30, 70)(30, 30)(30, 15)
\qbezier(30, 15)(35, -10)(45, 15)
\qbezier(45, 15)(53, 25)(90, 70)

{\color{red}\qbezier[60](-10, 33)(56, 33)(130, 33)}


\qbezier(45, 70)(53, 30)(60, 15)
\qbezier(60, 0)(60, 10)(60, 15)

{\color{red}\qbezier[60](-10, 50)(56, 50)(130, 50)}

\qbezier(60, 70)(68, 30)(75, 15)
\qbezier(75, 0)(75, 10)(75, 15)


\qbezier(75, 70)(83, 30)(90, 15)
\qbezier(90, 0)(90, 10)(90, 15)

\put(77, 10){$...$}

\qbezier(105, 70)(105, 30)(105, 0)
\put(108, 10){$...$}
\qbezier(120, 70)(120, 30)(120, 0)
\end{picture}
\end{center}
\caption{Regular expression}
\label{fig:Ts}
\end{figure}

This completes the proof of (1).

We now turn to the proof that the stated relations form a complete set.
Observe first that any expression for a diagram $D$ as a word in the generators
provides a regular expression for $D$ by repeated use of the relation
\eqref{eq:identity} and its dual. Accordingly we say that two regular expressions
$\mf{D},\mf{D}'$ are {\em equivalent}, and write $\mf D\sim \mf D'$ if
one can be obtained from the other by a sequence of applications of the relations in part (2)
of the Theorem. This is clearly an equivalence relation on regular expressions.

However, a word in the generators does not in general yield a Brauer diagram, but
rather a diagram multiplied by $\delta^k$ for some nonnegative integer $k$,
where $k$ is the number of deleted loops.
For any Brauer diagram $D$ and any $N\in\Z_+$, the above argument shows that
we can always represent $\delta^N D$
as a word in the generators, and hence also as a regular expression. We therefore
need to work with morphisms of the form
$\delta^N D$, where $D$ is a diagram. We refer to such a morphism
as  a {\em scaled Brauer diagram}, or simply a {\em scaled diagram}.
Every Brauer diagram is clearly a scaled diagram.

The discussion above shows that to prove the Theorem, it will suffice to show that
\be\label{eq:task}
\text{Any two regular expressions for a scaled diagram are equivalent.}
\ee

We shall extend the notion of equivalence to any expression of the form $D_1\circ\dots\circ D_n$,
where the $D_i$ are diagrams.

\begin{definition}\label{def:equiv}
The two compositions $D_1\circ\dots\circ D_n$ and $D_1'\circ\dots\circ D_m'$ are said to be equivalent
if one can be obtained from the other using only the relations
in Theorem \ref{thm:presentation} (2), and the properties of $\circ$ and $\ot$.
\end{definition}

To prove \eqref{eq:task} we require some analysis of regular expressions and equivalence.
We shall return to the proof after carrying this out.
\end{proof}

\begin{definition}\label{def:valency}
\begin{enumerate}
\item The {\em valency} of scaled diagram $D\in B_k^l$ is the pair $(k,l)$.
\item If $D=I^{\ot r}\ot Y\ot I^{\ot s}$ is elementary, the {\em abscissa}
$a(D)$ of $D$ is $r+1$, while the {\em type} $t(D)=Y$ $(=A,U$ or $X)$.
\item The {\em length} of a regular expression $E_1\circ\dots\circ E_n$ is $n$.
\end{enumerate}
\end{definition}

We shall repeatedly apply the following elementary observation, which we refer to
as the ``commutation principle''.
\begin{remark}\label{rem:commute}
\begin{enumerate}
\item Let $E_1,E_2$ be elementary diagrams such that $E_1\circ E_2$ makes sense.
If $|a(E_1)-a(E_2)|>1$ then $E_1\circ E_2\sim E_1'\circ E_2'$,
where $t(E_1')=t(E_2)$ and $t(E_2')=t(E_1)$.
\item If $D,D'$ are scaled diagrams of valency $(k,l)$ and $(k',l')$ respectively,
then $D\ot D'=(I^{\ot l}\ot D')\circ (D\ot I^{\ot k'})=(D\ot I^{\ot l'})\circ (I^{\ot k}\ot D')$.
\end{enumerate}
\end{remark}
Part (2) of the Remark states the obvious relations among diagrams depicted in Figure \ref{fig:commute}.
They follow from the fact that $(A\ot B)\circ (A'\ot B')\sim (A\circ A')\ot(B\circ B')$
for $A,A',B,B'$ of appropriate valency,
and the relation \eqref{eq:identity}.

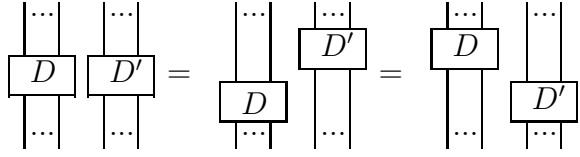
\begin{figure}[h]
\begin{center}
\begin{picture}(100, 60)(70,0)

\qbezier(16, 40)(16, 50)(16, 60)
\qbezier(29, 40)(29, 50)(29, 60)
\put(18, 55){$...$}

\qbezier(10, 40)(20, 40)(35, 40)
\qbezier(10, 25)(20, 25)(35, 25)
\put(18, 30){$D$}
\qbezier(10, 40)(10, 15)(10, 30)
\qbezier(35, 40)(35, 15)(35, 30)

\qbezier(16, 25)(16, 15)(16, 5)
\qbezier(29, 25)(29, 15)(29, 5)
\put(18, 10){$...$}


\qbezier(46, 40)(46, 50)(46, 60)
\qbezier(59, 40)(59, 50)(59, 60)
\put(48, 55){$...$}

\qbezier(40, 40)(50, 40)(65, 40)
\qbezier(40, 25)(60, 25)(65, 25)
\put(48, 30){$D'$}
\qbezier(40, 40)(40, 15)(40, 30)
\qbezier(65, 40)(65, 15)(65, 30)

\qbezier(46, 25)(46, 15)(46, 5)
\qbezier(59, 25)(59, 15)(59, 5)
\put(48, 10){$...$}

\put(70, 30){$=$}

\qbezier(96, 30)(96, 50)(96, 60)
\qbezier(109, 30)(109, 50)(109, 60)
\put(98, 55){$...$}

\qbezier(90, 30)(110, 30)(115, 30)
\qbezier(90, 15)(110, 15)(115, 15)
\put(98, 18){$D$}
\qbezier(90, 30)(90, 15)(90, 15)
\qbezier(115, 30)(115, 15)(115, 15)

\qbezier(96, 15)(96, 10)(96, 5)
\qbezier(109, 15)(109, 10)(109, 5)
\put(98, 10){$...$}


\qbezier(126, 50)(126, 50)(126, 60)
\qbezier(139, 50)(139, 50)(139, 60)
\put(128, 55){$...$}

\qbezier(120, 50)(130, 50)(145, 50)
\qbezier(120, 35)(140, 35)(145, 35)
\put(128, 40){$D'$}
\qbezier(120, 50)(120, 45)(120, 35)
\qbezier(145, 50)(145, 45)(145, 35)

\qbezier(126, 35)(126, 15)(126, 5)
\qbezier(139, 35)(139, 15)(139, 5)
\put(128, 10){$...$}

\put(150, 30){$=$}

\qbezier(176, 50)(176, 55)(176, 60)
\qbezier(189, 50)(189, 55)(189, 60)
\put(178, 55){$...$}

\qbezier(170, 50)(180, 50)(195, 50)
\qbezier(170, 35)(180, 35)(195, 35)
\put(178, 40){$D$}
\qbezier(170, 50)(170, 45)(170, 35)
\qbezier(195, 50)(195, 45)(195, 35)

\qbezier(176, 35)(176, 15)(176, 5)
\qbezier(189, 35)(189, 15)(189, 5)
\put(178, 10){$...$}


\qbezier(206, 30)(206, 40)(206, 60)
\qbezier(219, 30)(219, 40)(219, 60)
\put(208, 55){$...$}

\qbezier(200, 30)(210, 30)(225, 30)
\qbezier(200, 15)(210, 15)(225, 15)
\put(208, 20){$D'$}
\qbezier(200, 30)(200, 20)(200, 15)
\qbezier(225, 30)(225, 20)(225, 15)

\qbezier(206, 15)(206, 10)(206, 5)
\qbezier(219, 15)(219, 10)(219, 5)
\put(208, 10){$...$}
\end{picture}
\end{center}
\caption{Commutativity}
\label{fig:commute}
\end{figure}

The next two results will be used in the reduction of the proof of Theorem \ref{thm:presentation}
(2) to a single case.

\begin{lemma}\label{lem:red-perm} Let $P,Q$ be permutation diagrams of valency
$(l,l)$ and $(k,k)$ respectively and let $D\in B_k^l$ be a scaled diagram.
If any two regular expressions for ${P}\circ{D}\circ{Q}$
are equivalent, then so are any two regular expressions for ${D}$.
\end{lemma}
\begin{proof} Let $\mf D$, $\mf D'$ be two regular expressions for $D$,
and suppose for the moment that $P$ is an elementary permutation diagram.
Then $P\circ\mathfrak{D}$ and $P\circ\mathfrak{D}'$ are regular expressions
for $P\circ D$, and hence are equivalent by hypothesis. Now
$P\circ P\circ\mathfrak{D}$ is a regular expression, and
it is evident that $P\circ P\circ\mathfrak{D}$ is equivalent
to $P\circ P\circ\mathfrak{D}'$. But from \eqref{eq:XX}, $P\circ P\circ\mathfrak{D}\sim \mf{D}$
and $P\circ P\circ\mathfrak{D}'\sim\mathfrak{D}'$, whence
$\mf D$ and $\mf D'$ are equivalent. This proves the Lemma for elementary $P$ and $Q=\id$.

Applying the above statement repeatedly, we see that for any permutation diagram $P$,
if any two regular expressions for $P\circ D$ are equivalent, the same is true for
$D$. A similar argument applies to prove the corresponding statement for $D\circ Q$,
for any permutation diagram $Q$.
\end{proof}

It follows that in proving \eqref{eq:task}, we may pre- and post-multiply
$D$ by arbitrary permutation diagrams, and replace $D$ by the resulting scaled diagram.

For the second reduction, we require the following definitions.
\begin{definition}\label{def:raising}
\begin{enumerate}
\item Define $R:B_k^l\to B_{k-1}^{l+1}$ (for $k\geq 1$) (the {\em raising operator}) by
$R(D)=(D\ot I)\circ (I^{\ot (k-1)}\ot U)$, and (the {\em lowering operator})
$L:B_k^l\to B_{k+1}^{l-1}$ by
$L(D)=(I^{\ot (l-1)}\ot A)\circ(D\ot I)$.
\item If $\mf D=D_1\circ D_2\circ\dots\circ D_n$ is a regular expression for the
scaled diagram $D\in B_k^l$, define the regular expression $R(\mf D)$ for $R(D)$ by
$R(\mf D)=(D_1\ot I)\circ (D_2\ot I)\circ\dots\circ (D_n\ot I)\circ (I^{\ot k-1}\ot U)$,
and similarly define the regular expression $L(\mf D)$ for $L(D)$.
Note that if $E$ is elementary, then so is $E\ot I$, so that the above definition makes sense.
\end{enumerate}
\end{definition}

\begin{lemma}\label{lem:red-raise}
\begin{enumerate}
\item For any regular expression $\mf D$ for a scaled diagram $D\in B_k^l$,
we have $R\circ L(\mf D)\sim \mf D$ and $L\circ R(\mf D)\sim\mf D$.
\item Suppose $D$ is a scaled diagram of valence $(k,l)$ with $k\geq 1$.
The regular expressions $\mf D,\mf D'$ for $D$ are equivalent if and only if
$L(\mf D)$ and $L(\mf D')$ (or $R(\mf D)$ and $R(\mf D')$) are equivalent.
\end{enumerate}
\end{lemma}
\begin{proof}
To prove (1), let $\mf D=E_1\circ\dots\circ E_n$ be a regular expression for $D\in B_k^l$. Then
$$
\begin{aligned}
R\circ L&(\mf D)=R((I^{\ot(l-1)}\ot A)\circ(E_1\ot I)\dots\circ (E_n\ot I)\\
&=(I^{\ot(l-1)}\ot A\ot I)\circ(E_1\ot I\ot I)\dots\circ(E_n\ot I)\circ I^{\ot k}\ot U\\
&\sim(I^{\ot(l-1)}\ot A\ot I)\circ(I^{\ot l}\ot U)\circ E_1\circ\dots\circ E_n
\text{ by several applications of }\ref{rem:commute}\\
&\sim I^{\ot l}\circ E_1\circ\dots\circ E_n\text {by }\eqref{eq:straight}\\
&\sim E_1\circ\dots\circ E_n\text{ by }\eqref{eq:identity}\\
&=\mf D.\\
\end{aligned}
$$

This shows that $R\circ L(\mf D)\sim\mf D$, and the proof that $L\circ R(\mf D)\sim\mf D$ is similar.

Now to prove (2), suppose first that $\mf D, \mf D'$ are equivalent regular expressions
for $D$. Then the same sequence of moves using the relations in Theorem
\ref{thm:presentation} (2) which convert $\mf D$ into $\mf D'$ may be applied to
$L(\mf D)$ to convert it into $L(\mf D')$. This shows that if $\mf D,\mf D'$ are equivalent
regular expressions for $D$, then $L(\mf D),L(\mf D')$ are equivalent regular expressions
for $L(D)$. A similar argument proves the corresponding statement for $R(D)$.

To prove the converse,
suppose that any two regular expressions for $R(D)$ are equivalent,
and that $\mf D_1$ and $\mf D_2$ are two regular expressions for $D$. Then $R(\mf D_1)$ and
$R(\mf D_2)$ are two regular expressions for $R(D)$, and hence by hypothesis are
equivalent. Hence by the above, $L\circ R(\mf D_1)$ and $L\circ R(\mf D_2)$ are two
equivalent regular expressions for $L\circ R(D)$, which is equal to $D$ by (1).
But by (1), $L\circ R(\mf D_1)\sim \mf D_1$ and $L\circ R(\mf D_2)\sim\mf D_2$, whence
$\mf D_1\sim\mf D_2$.
\end{proof}

The following lemma is the key computation involving the relations in
Theorem \ref{thm:presentation} (2).

\begin{lemma}\label{lem:stack}
Let $\mf T_s:=E_s\circ E_{s-1}\circ\dots\circ E_0$ be a regular expression, where
$t(E_0)=U$, $a(E_0)=a$, $t(E_i)=X$ and $a(E_i)=a+i$ for $i\geq 1$. The diagram $\mf T_s$
is shown in Figure \ref{fig:Ts}.
Let $E$ be an elementary diagram of type $A$ or $X$ which does not `commute with'
$E_s\circ E_{s-1}\circ\dots\circ E_0$, i.e. such that $a-1\leq a(E)\leq a+s+1$.
Then
\begin{enumerate}
\item If $t(E)=A$, then $E\circ\mf T_s$ is equivalent to a shorter regular expression unless
$s=0$ and $a(E)=a(E_0)$. In the latter case, $E\circ\mf T_s$ is the identity multiplied by $\delta$.
\item Suppose $t(E)=X$; then

(i) if $a+1\leq a(E)\leq a+s-1$, then $E\circ\mf T_s\sim \mf T_s\circ E'$ for
an elementary diagram $E'$ of type $X$. (Thus $E$ may be `moved through' $E\circ\mf T_s$).

(ii) if $a(E)=a$ or $a+s$, then $E\circ\mf T_s$ is equivalent to
a shorter regular expression.

(iii) if $a(E)=a-1$ or $a+s+1$ then $E\circ \mf T_s\sim \mf T_{s+1}$.

\item Let $\mf T_s$ be as above and let $E$ be elementary of type $A$ or $X$.
Then $E\circ \mf T_s$ is equivalent to a shorter regular expression
(possibly multiplied by $\delta$) or to $\mf T_s\circ E'$
for some elementary $E'$, or to $\mf T_{s+1}$.
\end{enumerate}
\end{lemma}

\begin{proof} Consider first the case where $t(E)=A$.

If $s=0$ and $a(E)=a(E_0)$, the claim follows from the loop removal relation \eqref{eq:AU}.

If $a(E)=a+s+1$, then
applying \eqref{eq:slide}, $E\circ E_s\sim E'\circ E_s'$, where $t(E')=t(E)=A$,
$t(E_s')=t(E_s)=X$, $a(E')=a+s$ and $a(E_s')=a+s+1$. It now follows by repeated application
of Remark \ref{rem:commute} about commutation, that $E\circ\mf T_s\sim E''\circ \mf T_{s-1}\circ E'''$,
where $t(E'')=A$ and $a(E'')=a+s$. Repeating this argument $s$ times, we see that
$E\circ \mf T_s$ is equivalent to a regular expression
of length $s+1$ which includes $F\circ E_0$ as a subexpression, where $t(F)=A$
and $a(F)=a+1$. Applying \eqref{eq:straight}, we see that $F\circ E\sim I^{\ot k}$ for some $k$,
and hence $E\circ\mf T_s$ is equivalent to a regular expression of length $s-1$.

If $a(E)=a+s$, then by \eqref{eq:AX}, $E\circ E_s\sim E$, and we have again shortened $E\circ\mf T_s$.

If $a\leq a(E)\leq a+s-1$, then by commutation, $E\circ\mf T_s$ is equivalent to a regular
expression with a subexpression of the form $E\circ E_i\circ E_{i-1}$, where $t(E_i)=X$
and $a(E)=a(E_i)-1$. Applying \eqref{eq:straight}, this is equivalent
to an expression $E'\circ E_i'\circ E_{i-1}$, where $a(E_i')=a(E_{i-1})$, and $t(E_i')=X$.
Using either \eqref{eq:XX} (if $i>1$) or the $^*$ of \eqref{eq:AX}, we again reduce the
length to show that $E\circ\mf T_s$ is equivalent to a shorter regular expression.

Finally if $a(E)=a-1$, we use commutation to show that $E\circ\mf T_s$ is equivalent
to a regular expression of length $s+1$ with a subexpression of the form
$E'\circ E_0$, where $t(E')=A$ and $a(E')=a-1=a(E_0)-1$. Applying \eqref{eq:straight},
we see that $E'\circ E_0\sim I^{\ot k}$ for some $k$, and this completes the proof of (1).

Now consider the case where $t(E)=X$.

If $a+1\leq a(E)\leq a+s-1$, then after applying the commutation rule, $E\circ\mf T_s$
is equivalent to a regular expression of length $s+1$ which has a subexpression of
the form $E\circ E_{a(E)+1}\circ E_{a(E)}$. But using the braid relation \eqref{eq:braid},
this is equivalent to $E'\circ E_{a(E)}\circ E_{a(E)+1}$, where $E'=E_{a(E)+1}$.
Again using commutation, we may now move the last factor below $E_0$ (since $a(E)+1\geq a+2$).
It follows that $E\circ\mf T_s\sim \mf T_s\circ E'$, where $t(E')=X$. This proves (i).

If $a(E)=a+s+1$ then evidently $E\circ\mf T_s= \mf T_{s+1}$.
If $a(E)=a+s$, the relation $X\circ X=I\ot I$ \eqref{eq:XX} shows that $E\circ E_s\sim I^{\ot r}$
for some $r$, and hence $E\circ\mf T_s$ is equivalent to a shorter regular expression.
If $a(E)=a-1$, then we may use commutation to see that
$E\circ\mf T_s\sim E_s\circ\dots\circ E_1\circ E\circ E_0$. Using the relation \eqref{eq:slide}
we see that this is equivalent to $E_s\circ\dots\circ E_1\circ E_1\circ E_0'$, where $t(E_0')=U$.
Applying \eqref{eq:XX}, we see that $E\circ\mf T_s$ is equivalent to a shorter regular expression.
Finally, if $a(E)=a$, we again use commutation to see that $E\circ\mf T_s$ is equivalent to
$E_s\circ E_{s-1}\circ\dots\circ E\circ E_1\circ E_0$. Again applying \eqref{eq:slide},
we obtain a factor $E\circ E$, and applying \eqref{eq:XX}, we again shorten the regular
expression $E\circ\mf T_s$. This completes the proof of (2).

The statement (3) is a summary of the previous two statements.
\end{proof}

\begin{proof}[Completion of the proof of Theorem \ref{thm:presentation} (2)]
It remains to prove \eqref{eq:task}.
It follows from Lemmas \ref{lem:red-raise} and \ref{lem:red-perm} that to complete the
proof of the theorem, it suffices to prove \eqref{eq:task} for any scaled diagram which can be
obtained from $D$ by raising or lowering, or multiplication by a permutation diagram.
It follows that we may take $D$ to be the scaled diagram
$D=\delta^N U^{\ot r}$ $(N\in\Z_+)$. Hence we shall be done if we prove
the following result.
\be\label{eq:main}
\text {Any two regular expressions for $D=\delta^N U^{\ot r}$ are equivalent.}
\ee

We shall prove \eqref{eq:main} by induction on $r$, starting with $r=0$.
For convenience, we adopt the following local convention:
\begin{enumerate}
\item scaled diagrams will
be simply called ``diagrams";
\item a regular expression $\mf D$ is said to be ``$\delta$-equivalent" to another regular expression
$\mf D'$ if it can be changed to $\delta^k \mf D'$ for some $k\in\Z_+$ by the relations in
Theorem \ref{thm:presentation} (2).
\end{enumerate}

Let $r=0$ and suppose $\mf D:=D_1\circ\dots\circ D_n$ is a regular expression for the empty
scaled diagram $\delta^N$ in $B_0^0$. We need to show that $\mf D$ is $\delta$-equivalent to the empty regular
expression; we do this by showing that every non-empty regular expression for the
empty scaled diagram is $\delta$-equivalent to one of shorter length.

Now by valency considerations, we must have $D_1=A$ and $D_n=U$. Let $i$ be the least
integer such that $t(D_i)=U$; then for all $j<i$, $t(D_j)=A$ or $X$. Applying
Lemma \ref{lem:stack} repeatedly, we see that since at least one of the $D_j$
for $j<i$ is of type $A$, $\mf D$ is $\delta$-equivalent to a shorter regular expression.
This proves the result for $r=0$

Now take $r>0$ and let $\mf D=D_1\circ\dots\circ D_n$ be a regular
expression for $D$. Then since at least $r$ of the $D_i$ must have type $U$, we
have $n\geq r$. Moreover if $n=r$, which happen only if $N=0$,
then the $D_i$ are all of type $U$, and have odd
abscissa, and any such regular expression represents $D$.
Any two such regular expressions (which will be called minimal) are equivalent
by the commutation rule (see Remark \ref{rem:commute}).

It therefore suffices to show that if $n>r$, then $\mf D$ is $\delta$-equivalent to a shorter
regular expression.

Clearly we have $t(D_n)=U$; if $t(D_1)=U$ then $\mf D':=D_2\circ\dots\circ D_n$ is a
regular expression for $U^{\ot (r-1)}$, and we conclude by induction on $r$ that $\mf D'$
is $\delta$-equivalent to a shorter regular expression. Thus we are finished.
Let $p=p(\mf D)$ be the least index such that $D_p$ is of type $U$. We have seen that if $p=1$
then we are finished by induction. It will therefore suffice to show that
$\mf D$ is either equivalent to a regular expression $\mf D'$ with $p(\mf D')<p(\mf D)$,
or is $\delta$-equivalent to a shorter regular expression $\mf D'$.

Thus we take $p>1$; then $t(D_p)=U$, and $t(D_i)=A$ or $X$ for $i<p$.
We now apply Lemma \ref{lem:stack} to conclude that either
we may commute one of the $D_i$ ($i<p$) past $D_p$, or $D_1\circ\dots\circ D_p
\sim \mf T_{p-1}$ or at least one of the $D_i$ ($i<p$) is of type $A$.
In the first case, we obtain a regular expression with small $p$-value;
in the second case, in the diagram $D_1\circ\dots\circ D_n$ if the nodes are
numbered $1,2,\dots,2r$ from left to right, node $a(D_p)$ would be joined to
node $a(D_p)+p$. Hence $p=1$, which has been excluded.

In the third case, suppose $i$ is the largest index such that $1\leq i\leq p-1$
and $D_i$ is of type $A$. Then either some $D_j$ ($i\leq j\leq p-1$)
can be commuted past $D_p$ by application of Remark (\ref{rem:commute}), or else
we are in the situation of Lemma \ref{lem:stack} (1). In the former case, we have reduced
$p$; in the latter, by {\it loc. cit.} $D_i\circ\dots\circ D_p$ is $\delta$-equivalent to a shorter
regular expression.

We have now shown that either $\mf D$ is $\delta$-equivalent to a shorter regular expression,
or equivalent to a regular expression which has the same length as $\mf D$ but a smaller $p$ value.

This completes the proof of \eqref{eq:main}, and hence of Theorem \ref{thm:presentation}.
\end{proof}

\begin{remark}
We note that to prove part (2) of the theorem, we could
have proceeded by regarding $\cB(\delta)$ as a quotient category of
the category of (unoriented) tangles (see Remark \ref{rem:quotient-cat}) and deduce
the relations among the generators of Brauer diagrams
from a complete set of relations among the generators of tangles
given in \cite[\S 3.2]{T1} (suppressing information about orientation).
This way we obtain all relations except the one which
enforces the removal of free loops and multiplication by powers of $\delta$,
i.e., \eqref{eq:AU}.
\end{remark}

%
%
\subsection{Some useful diagrams}
We shall find the following diagrams useful in later sections of this work.
Let $A_q=A\circ (I\ot A\ot I)\dots (I^{\ot (q-1)}\ot A\ot I^{\ot (q-1)})$,
$U_q=(I^{\ot (q-1)}\ot U\ot I^{\ot (q-1)})\circ\dots\circ (I\ot U\ot I)\circ U$
and $I_q=I^{\ot q}$. These are depicted as diagrams in Figure \ref{Omega-U},
\begin{figure}[ht]
\begin{center}
\begin{picture}(340, 40)(-5,0)
\put(-5, 10){$A_q=$}
\qbezier(25, 0)(55, 60)(85, 0)
\put(33, 10){...}
\put(30, 0){\tiny$q$}
\qbezier(40, 0)(55, 40)(70, 0)
\put(90, 0){,}

\put(140, 10){$U_q=$}
\qbezier(165, 30)(195, -35)(225, 30)
\put(172, 25){\tiny$q$}
\put(172, 20){...}
\qbezier(180, 30)(195, -15)(210, 30)
\put(215, 0){,}

\put(270, 10){$I_q=$}
\put(300, 0){\line(0, 1){40}}
\put(310, 20){...}
\put(314, 0){\tiny$q$}
\put(330, 0){\line(0, 1){40}}
\put(335, 0){.}
\end{picture}
\end{center}
\caption{ }
\label{Omega-U}
\end{figure}

We shall also need $X_{s, t}$, the $(s+t, s+t)$ Brauer diagram shown in Figure \ref{X}.
\begin{figure}[h]
\begin{center}
\begin{picture}(50, 60)(0,0)

\qbezier(0, 60)(15, 30)(30, 0)
\qbezier(15, 60)(30, 30)(45, 0)

\qbezier(0, 0)(15, 30)(30, 60)
\qbezier(15, 0)(30, 30)(45, 60)

\put(5, 5){...}
\put(5, -5){\tiny$s$}

\put(30, 5){...}
\put(35, -5){\tiny$t$}
\end{picture}
\end{center}
\caption{}
\label{X}
\end{figure}

The following result is easy.

\begin{lemma} \label{isomorphism-pri}
\begin{enumerate}
\item For any Brauer diagrams $D_1\in B_k^r(\delta)$ and
$D_2\in B_r^q(\delta)$, we have $I_r\circ D_1 =D_1$ and $D_2\circ I_r =D_2$.
That is, $I_r=\id_r$ for any object $r$ of $\cB(\delta)$.
\item  The following relation holds.
\[(I_q\otimes A_q)\circ(U_q\otimes I_q)=(U_q\otimes I_q)\circ(I_q\otimes A_q)=I_q.\]
\end{enumerate}
\end{lemma}

\begin{corollary}\label{isomorphism}
For all $p, q$ and $r$, define the linear maps
\[
\begin{aligned}
{\mathbb U}_p^q=(-\otimes I_q)\circ(I_p\otimes U_q):
B_{p+q}^r(\delta)\longrightarrow B_p^{r+q}(\delta)\\
{\mathbb A}^r_q=(I_{r+q}\otimes A_q)\circ(- \otimes I_q):
B_p^{r+q}(\delta)\longrightarrow B_{p+q}^r(\delta)
\end{aligned}
\]
Then ${\mathbb U}_p^q=R^q$ and ${\mathbb A}^r_q=L^q$ (see Definition \ref{def:raising}).
These are mutually inverse.
\end{corollary}
This is clear since by Lemma \ref{lem:red-raise}, $L$ and $R$ are mutually inverse.

Let $\ast: B_p^q(\delta) \longrightarrow B_q^p(\delta)$ be the linear map defined for any
$D \in B_p^q(\delta)$ by 
$
*D=(I_p\otimes A_q)\circ(I_p\otimes D\otimes I_q)\circ(U_p\otimes I_q).
$
The diagram $*D$ is depicted in Figure \ref{D-star}.
\begin{figure}[h]
\begin{center}
\begin{picture}(100, 60)(-20,0)

\put(-5, 20){\line(1, 0){35}}
\put(-5, 20){\line(0, 1){20}}
\put(30, 20){\line(0, 1){20}}
\put(-5, 40){\line(1, 0){35}}
\put(8, 25){D}

\qbezier(5, 40)(40, 95)(60, 0)
\qbezier(20, 40)(35, 70)(45, 0)
\put(45, 10){...}

\qbezier(5, 20)(-15, -10)(-20, 60)
\qbezier(20, 20)(-20, -35)(-35, 60)
\put(-30, 50){...}

\end{picture}
\end{center}
\caption{$*D$}
\label{D-star}
\end{figure}

\begin{lemma}\label{anti-invol}
The map $\ast$ coincides with the anti-involution 
$D\mapsto D^{*\circ\sharp}$ discussed in \S\ref{ssec:involutions}.
That is, $*D=D^{*\circ\sharp}$ for any diagram $D$.
\end{lemma}
This is easily seen in terms of diagrams.

\subsection{The Brauer algebra}\label{sect:Brauer-algebra}

For any object $r$ in $\cB(\delta)$,  the set of morphisms
$B_r^r(\delta)$ from $r$ to itself form
a unital associative $K$-algebra under composition
of Brauer diagrams. This is the Brauer algebra of degree $r$ with parameter $\delta$,
which we will denote by $B_r(\delta)$.
The first two results of the following lemma are well known.

\begin{lemma}\label{lem:brprops}
\begin{enumerate}
\item For $i=1,\dots,r-1$, let $s_i$ and $e_i$ respectively be the $(r, r)$ Brauer diagrams shown in
Figure \ref{s-e} below.
\begin{figure}[ht]
\begin{center}
\begin{picture}(350, 60)(0,0)
\put(0, 0){\line(0, 1){60}}
\put(40, 0){\line(0, 1){60}}
\put(18, 30){...}
\put(10, 0){$i-1$}

\qbezier(60, 0)(70, 30)(80, 60)
\qbezier(60, 60)(70, 30)(80, 0)

\put(100, 0){\line(0, 1){60}}
\put(140, 0){\line(0, 1){60}}
\put(118, 30){...}
\put(150, 0){, }

\put(200, 0){\line(0, 1){60}}
\put(240, 0){\line(0, 1){60}}
\put(218, 30){...}
\put(210, 0){$i-1$}

\qbezier(260, 0)(270, 45)(280, 0)
\qbezier(260, 60)(270, 15)(280, 60)

\put(300, 0){\line(0, 1){60}}
\put(340, 0){\line(0, 1){60}}
\put(318, 30){...}
\put(350, 0){. }
\end{picture}
\end{center}
\caption{ }
\label{s-e}
\end{figure}
Then $B_r(\delta)$ has the following presentation as $K$-algebra with anti-involution $*$.
The generators are $\{s_i, e_i \mid i=1, 2, \dots, r-1\}$, and relations
\[
\begin{aligned}
&s_is_j=s_js_i,\;s_ie_j=e_js_i,\;e_ie_j=e_je_i, \quad\text{if $|i-j|\geq 2$},\\
&s_i^2=1,\; s_is_{i+1}s_i=s_{i+1}s_is_{i+1},\\
&s_ie_i=e_is_i=e_i,\\
&e_i^2=\delta e_i,\\
&e_ie_{i\pm 1}e_i=e_i,\\
&s_ie_{i+1}e_i=s_{i+1}e_i,
\end{aligned}
\]
where the last five relations being valid for all applicable $i$.

\item The elements $s_1,\dots,s_{r-1}$ generate a subalgebra of $B_r(\delta)$, isomorphic to
the group algebra $K\Sym_r$ of the symmetric group $\Sym_r$.

\item  The map $\ast$ of Lemma \ref{anti-invol} restricts to an anti-involution of the Brauer
algebra.
\end{enumerate}
\end{lemma}
Parts (1) and (2) follow from Theorem \ref{thm:presentation}, noting that any
regular expression for a diagram in $B_r(\delta)$ contains an equal number of factors of
type $A$ and $U$. The stated relations are precise analogues of the relations in Theorem
\ref{thm:presentation} (2).
Part (3) is easy to prove. However we note that
$*s_i= s_{r+1-i}$ and $*e_i= e_{r+1-i}$. This is different from
the standard cellular anti-involution $^*$ of the Brauer algebra.

We remark that multiplying the last relation above by $e_i$ on the left and using
two of the earlier relations, we obtain
\[
e_is_{i+ 1}e_i=e_i,
\]
a relation which we shall often use, together with its transform under $*$:
$e_is_{i-1}e_i=e_i$.

\medskip
Now we prove some technical lemmas for later use.
\begin{lemma} \label{lem:Sigma-1} Let
$\Sigma_\epsilon(r)=\sum_{\sigma\in \Sym_r} (-\epsilon)^{|\sigma|} \sigma\in B_r(\delta)$,
where $\epsilon= \pm 1$ and $|\sigma|$ is the length of $\sigma$.
Represent $\Sigma_\epsilon(r)$  pictorially by Figure \ref{Sigma-r}.

\begin{figure}[h]
\begin{center}
\begin{picture}(100, 60)(-5,0)
\put(20, 40){\line(0, 1){20}}
\put(35, 50){...}
\put(60, 40){\line(0, 1){20}}

\put(0, 20){\line(1, 0){80}}
\put(0, 20){\line(0, 1){20}}
\put(80, 20){\line(0, 1){20}}
\put(0, 40){\line(1, 0){80}}
\put(40, 28){$r$}

\put(20, 20){\line(0, -1){20}}
\put(35, 10){...}
\put(60, 20){\line(0, -1){20}}

\put(85, 0){.}
\end{picture}
\end{center}
\caption{}
\label{Sigma-r}
\end{figure}

Then the following relations hold for all $r$.
\begin{enumerate}
\item
\[
\begin{picture}(80, 60)(0, -30)
\put(0, 10){\line(1, 0){60}}
\put(0, -10){\line(1, 0){60}}
\put(0, 10){\line(0, -1){20}}
\put(60, 10){\line(0, -1){20}}
\put(25, -3){$r$}

\put(10, 10){\line(0, 1){15}}
\put(23, 15){$\cdots$}
\put(50, 10){\line(0, 1){15}}

\put(10, -10){\line(0, -1){15}}
\put(23, -20){$\cdots$}
\put(50, -10){\line(0, -1){15}}

\put(70, -2){$=$}
\end{picture}
\begin{picture}(80, 60)(-10, -30)
\put(0, 10){\line(1, 0){50}}
\put(0, -10){\line(1, 0){50}}
\put(0, 10){\line(0, -1){20}}
\put(50, 10){\line(0, -1){20}}
\put(15, -3){$r-1$}

\put(10, 10){\line(0, 1){15}}
\put(18, 15){$\cdots$}
\put(40, 10){\line(0, 1){15}}

\put(10, -10){\line(0, -1){15}}
\put(18, -20){$\cdots$}
\put(40, -10){\line(0, -1){15}}

\put(60, -25){\line(0, 1){50}}

\put(70, -3){$-$}

\end{picture}
\begin{picture}(80, 60)(-70, -30)

\put(-60, -3){$\epsilon (r-2)!^{-1}$}

\put(0, 10){\line(1, 0){40}}
\put(0, 25){\line(1, 0){40}}
\put(0, 10){\line(0, 1){15}}
\put(40, 10){\line(0, 1){15}}
\put(12, 15){\tiny$r-1$}

\put(5, 10){\line(0, -1){20}}
\put(25, 10){\line(0, -1){20}}
\put(8, -3){$\cdots$}

\qbezier(35,10)(45, 5)(50, -35)
\qbezier(35,-10)(45, -5)(50, 35)

\put(0, -10){\line(1, 0){40}}
\put(0, -25){\line(1, 0){40}}
\put(0, -10){\line(0, -1){15}}
\put(40, -10){\line(0, -1){15}}
\put(12, -20){\tiny$r-1$}

\put(5, 25){\line(0, 1){10}}
\put(35, 25){\line(0, 1){10}}
\put(14, 28){$\cdots$}

\put(5, -25){\line(0, -1){10}}
\put(35, -25){\line(0, -1){10}}
\put(14, -33){$\cdots$}

\put(55, -35){.}
\end{picture}
\]

\item
\[
\begin{picture}(80, 60)(0, -30)
\put(0, 10){\line(1, 0){60}}
\put(0, -10){\line(1, 0){60}}
\put(0, 10){\line(0, -1){20}}
\put(60, 10){\line(0, -1){20}}
\put(25, -3){$r$}

\put(10, 10){\line(0, 1){15}}
\put(20, 15){$\cdots$}
\put(40, 10){\line(0, 1){15}}

\qbezier(50, -10)(60, -40)(65, 0)
\qbezier(50, 10)(60, 40)(65, 0)

\put(10, -10){\line(0, -1){15}}
\put(20, -20){$\cdots$}
\put(40, -10){\line(0, -1){15}}

\put(75, -3){$=$}
\end{picture}
\begin{picture}(80, 60)(-90, -30)
\put(-80, -5){$-\epsilon(r-1-\epsilon\delta)$}
\put(0, 10){\line(1, 0){60}}
\put(0, -10){\line(1, 0){60}}
\put(0, 10){\line(0, -1){20}}
\put(60, 10){\line(0, -1){20}}
\put(20, -3){$r-1$}

\put(10, 10){\line(0, 1){15}}
\put(23, 15){$\cdots$}
\put(50, 10){\line(0, 1){15}}

\put(10, -10){\line(0, -1){15}}
\put(23, -20){$\cdots$}
\put(50, -10){\line(0, -1){15}}

\put(60, -25){.}
\end{picture}
\]

\item
\[
\begin{picture}(80, 60)(0, -30)
\put(0, 10){\line(1, 0){60}}
\put(0, -10){\line(1, 0){60}}
\put(0, 10){\line(0, -1){20}}
\put(60, 10){\line(0, -1){20}}
\put(25, -3){$r$}

\put(10, 10){\line(0, 1){15}}
\put(20, 15){$\cdots$}
\put(40, 10){\line(0, 1){15}}

\put(50, -10){\line(0, -1){15}}
\qbezier(50, 10)(65, 40)(70, -25)

\put(10, -10){\line(0, -1){15}}
\put(20, -20){$\cdots$}
\put(40, -10){\line(0, -1){15}}

\put(75, -3){$=$}
\end{picture}
\begin{picture}(80, 60)(-65, -30)
\put(-55, -5){$\sum_{i=0}^{r-1} (-\epsilon)^i$}
\put(0, 10){\line(1, 0){60}}
\put(0, -10){\line(1, 0){60}}
\put(0, 10){\line(0, -1){20}}
\put(60, 10){\line(0, -1){20}}
\put(20, -3){$r-1$}

\put(10, 10){\line(0, 1){15}}
\put(23, 15){$\cdots$}
\put(50, 10){\line(0, 1){15}}

\put(10, -10){\line(0, -1){15}}
\put(11, -20){\tiny{...}}
\put(20, -10){\line(0, -1){15}}
\put(40, -10){\line(0, -1){15}}
\put(41, -20){\tiny{...}}
\put(42, -28){\tiny$i$}
\put(50, -10){\line(0, -1){15}}

\qbezier(30, -25)(45, 0)(60, -25)

\put(65, -25){.}
\end{picture}
\]
\end{enumerate}
\end{lemma}
\begin{proof}
Part (1) generalises \cite[Lemma 5.1 (i)]{LZ4} and is a simple consequence of the double coset decomposition of
$\Sym_r$ into $\Sym_{r-1}\amalg \Sym_{r-1}s_{r-1}\Sym_{r-1}$. Part (2) immediately follows from
(1). Statement (3) can be obtained from (1) by induction on $r$.
\end{proof}

\begin{remark}
Symmetry considerations easily show that the second diagram on the right hand side
of Lemma \ref{lem:Sigma-1} (1) is a $(r-2)!$-multiple of a $\Z$-linear combination of Brauer diagrams;
thus the second term is still defined over $\Z$ despite having the coefficient $\frac{1}{(r-2)!}$.
The same remark applies to similar terms appearing in Lemma \ref{lem:Sigma} and its proof.
\end{remark}

\begin{lemma} \label{lem:Sigma}
Set $\epsilon=-1$. Then for all $k\ge 0$,
\begin{eqnarray}\label{eq:Sigma}
\begin{aligned}
\begin{picture}(80, 50)(50, -25)
\put(0, 10){\line(1, 0){60}}
\put(0, -10){\line(1, 0){60}}
\put(0, 10){\line(0, -1){20}}
\put(60, 10){\line(0, -1){20}}
\put(25, -3){$r$}

\put(5, 10){\line(0, 1){15}}
\put(15, 15){$\cdots$}
\put(35, 10){\line(0, 1){15}}
\qbezier(45, 10)(50, 35)(55, 10)

\put(5, -10){\line(0, -1){15}}
\put(7, -20){...}
\put(20, -10){\line(0, -1){15}}
\qbezier(45, -10)(50, -35)(55, -10)
\qbezier(25, -10)(30, -35)(35, -10)
\put(35, -20){...}
\put(38, -30){\tiny $k$}

\put(65, -3){$=$}
\end{picture}
\begin{picture}(80, 50)(-45, -25)
\put(-100, -5){$4k(r+\frac{\delta}{2}-k-1)$}
\put(-10, 10){\line(1, 0){60}}
\put(-10, -10){\line(1, 0){60}}
\put(-10, 10){\line(0, -1){20}}
\put(50, 10){\line(0, -1){20}}
\put(10, -3){$r-2$}

\put(-5, 10){\line(0, 1){15}}
\put(13, 15){$\cdots$}
\put(45, 10){\line(0, 1){15}}

\put(-5, -10){\line(0, -1){15}}
\put(-2, -20){...}
\put(10, -10){\line(0, -1){15}}
\qbezier(35, -10)(40, -35)(45, -10)
\qbezier(15, -10)(20, -35)(25, -10)
\put(25, -20){...}
\put(22, -30){\tiny $k-1$}

\put(55, -3){$+$}
\end{picture}
\begin{picture}(80, 60)(-105, -25)
\put(-75, -3){$(r-2-2k)!^{-1}$}

\put(0, 10){\line(1, 0){55}}
\put(0, 25){\line(1, 0){55}}
\put(0, 10){\line(0, 1){15}}
\put(55, 10){\line(0, 1){15}}
\put(18, 15){\tiny $r-2$}

\put(5, 10){\line(0, -1){20}}
\put(15, 10){\line(0, -1){20}}
\put(5, -3){...}

\qbezier(20,10)(25, -7)(30, 10)
\qbezier(40,10)(45, -7)(50, 10)
\put(30, 7){\tiny ...}
\put(33, -2){\tiny$k$}
\qbezier(20,-10)(25, 7)(30, -10)

\put(0, -10){\line(1, 0){40}}
\put(0, -25){\line(1, 0){40}}
\put(0, -10){\line(0, -1){15}}
\put(40, -10){\line(0, -1){15}}
\put(10, -20){\tiny $r-2k$}

\put(5, 25){\line(0, 1){10}}
\put(50, 25){\line(0, 1){10}}
\put(20, 28){$\cdots$}

\put(5, -25){\line(0, -1){10}}
\put(35, -25){\line(0, -1){10}}
\put(14, -33){$\cdots$}

\put(50, -35){.}
\end{picture}
\end{aligned}
\end{eqnarray}
\end{lemma}
\vspace{.2cm}
\begin{proof} For $k=0$, the formula is an identity.
The important case is $k=1$, where the formula becomes
\begin{eqnarray}\label{eq:k-1}
\begin{aligned}
\begin{picture}(80, 50)(20, -25)
\put(0, 10){\line(1, 0){60}}
\put(0, -10){\line(1, 0){60}}
\put(0, 10){\line(0, -1){20}}
\put(60, 10){\line(0, -1){20}}
\put(25, -3){$r$}

\put(5, 10){\line(0, 1){15}}
\put(15, 15){$\cdots$}
\put(35, 10){\line(0, 1){15}}
\qbezier(45, 10)(50, 35)(55, 10)

\put(5, -10){\line(0, -1){15}}
\put(15, -20){$\cdots$}
\put(35, -10){\line(0, -1){15}}
\qbezier(45, -10)(50, -35)(55, -10)

\put(65, -3){$=$}
\end{picture}
\begin{picture}(80, 50)(-40, -25)
\put(-65, -5){$4(r-2+\frac{\delta}{2})$}
\put(0, 10){\line(1, 0){50}}
\put(0, -10){\line(1, 0){50}}
\put(0, 10){\line(0, -1){20}}
\put(50, 10){\line(0, -1){20}}
\put(15, -3){$r-2$}

\put(10, 10){\line(0, 1){15}}
\put(18, 15){$\cdots$}
\put(40, 10){\line(0, 1){15}}

\put(10, -10){\line(0, -1){15}}
\put(18, -20){$\cdots$}
\put(40, -10){\line(0, -1){15}}

\put(60, -3){$+$}
\end{picture}
\begin{picture}(80, 60)(-90, -25)

\put(-55, -3){$(r-4)!^{-1}$}

\put(0, 10){\line(1, 0){50}}
\put(0, 25){\line(1, 0){50}}
\put(0, 10){\line(0, 1){15}}
\put(50, 10){\line(0, 1){15}}
\put(15, 15){\tiny $r-2$}

\put(5, 10){\line(0, -1){20}}
\put(25, 10){\line(0, -1){20}}
\put(8, -3){$\cdots$}

\qbezier(35,10)(40, -7)(45, 10)
\qbezier(35,-10)(40, 7)(45, -10)

\put(0, -10){\line(1, 0){50}}
\put(0, -25){\line(1, 0){50}}
\put(0, -10){\line(0, -1){15}}
\put(50, -10){\line(0, -1){15}}
\put(15, -20){\tiny $r-2$}

\put(5, 25){\line(0, 1){10}}
\put(45, 25){\line(0, 1){10}}
\put(18, 28){$\cdots$}

\put(5, -25){\line(0, -1){10}}
\put(45, -25){\line(0, -1){10}}
\put(18, -33){$\cdots$}

\put(55, -35){.}
\end{picture}
\end{aligned}
\end{eqnarray}
To prove it, we first obtain from Lemma \ref{lem:Sigma-1}(1) with $\epsilon=-1$
the following relation.
\[
\begin{picture}(80, 50)(0, -30)
\put(0, 10){\line(1, 0){60}}
\put(0, -10){\line(1, 0){60}}
\put(0, 10){\line(0, -1){20}}
\put(60, 10){\line(0, -1){20}}
\put(25, -3){$r$}

\put(5, 10){\line(0, 1){15}}
\put(15, 15){$\cdots$}
\put(35, 10){\line(0, 1){15}}
\qbezier(45, 10)(50, 35)(55, 10)

\put(5, -10){\line(0, -1){15}}
\put(15, -20){$\cdots$}
\put(35, -10){\line(0, -1){15}}
\qbezier(45, -10)(50, -35)(55, -10)

\put(65, -3){$=$}
\end{picture}
\begin{picture}(80, 60)(0, -30)
\put(0, 10){\line(1, 0){60}}
\put(0, -10){\line(1, 0){60}}
\put(0, 10){\line(0, -1){20}}
\put(60, 10){\line(0, -1){20}}
\put(20, -3){$r-1$}

\put(10, 10){\line(0, 1){15}}
\put(20, 15){$\cdots$}
\put(40, 10){\line(0, 1){15}}

\qbezier(50, -10)(60, -40)(65, 0)
\qbezier(50, 10)(60, 40)(65, 0)

\put(10, -10){\line(0, -1){15}}
\put(20, -20){$\cdots$}
\put(40, -10){\line(0, -1){15}}

\put(75, -3){$+$}
\end{picture}
\begin{picture}(80, 60)(-70, -30)
\put(-60, -3){$(r-2)!^{-1}$}

\put(0, 10){\line(1, 0){40}}
\put(0, 25){\line(1, 0){40}}
\put(0, 10){\line(0, 1){15}}
\put(40, 10){\line(0, 1){15}}
\put(12, 15){\tiny $r-1$}

\put(5, 10){\line(0, -1){20}}
\put(25, 10){\line(0, -1){20}}
\put(8, -3){$\cdots$}

\qbezier(35,10)(45, 0)(50, -15)
\qbezier(35,-10)(45, 0)(50, 15)
\qbezier(35,25)(45, 45)(50, 15)
\qbezier(35,-25)(45, -45)(50, -15)

\put(0, -10){\line(1, 0){40}}
\put(0, -25){\line(1, 0){40}}
\put(0, -10){\line(0, -1){15}}
\put(40, -10){\line(0, -1){15}}
\put(12, -20){\tiny $r-1$}

\put(5, 25){\line(0, 1){10}}
\put(25, 25){\line(0, 1){10}}
\put(10, 28){$\cdots$}

\put(5, -25){\line(0, -1){10}}
\put(25, -25){\line(0, -1){10}}
\put(10, -33){$\cdots$}

\put(55, -35){.}
\end{picture}
\]
Using Lemma \ref{lem:Sigma-1}(2) to the first diagram on the right hand side, and applying
Lemma \ref{lem:Sigma-1}(3) and the corresponding relation under the anti-involution $\ast$ to the
second diagram, we obtain \eqref{eq:k-1}.

The general case can be proven by induction on $k$. From
\eqref{eq:Sigma} at $k$, we obtain
\[
\begin{picture}(80, 50)(50, -30)
\put(0, 10){\line(1, 0){60}}
\put(0, -10){\line(1, 0){60}}
\put(0, 10){\line(0, -1){20}}
\put(60, 10){\line(0, -1){20}}
\put(25, -3){$r$}

\put(5, 10){\line(0, 1){15}}
\put(15, 15){$\cdots$}
\put(35, 10){\line(0, 1){15}}
\qbezier(45, 10)(50, 35)(55, 10)

\put(5, -10){\line(0, -1){15}}
\put(7, -20){...}
\put(20, -10){\line(0, -1){15}}
\qbezier(45, -10)(50, -35)(55, -10)
\qbezier(25, -10)(30, -35)(35, -10)
\put(35, -20){...}
\put(35, -30){\tiny $k+1$}

\put(65, -3){$=$}
\end{picture}
\begin{picture}(80, 50)(-45, -30)
\put(-100, -5){$4k(r+\frac{\delta}{2}-k-1)$}
\put(-10, 10){\line(1, 0){60}}
\put(-10, -10){\line(1, 0){60}}
\put(-10, 10){\line(0, -1){20}}
\put(50, 10){\line(0, -1){20}}
\put(10, -3){$r-2$}

\put(-5, 10){\line(0, 1){15}}
\put(13, 15){$\cdots$}
\put(45, 10){\line(0, 1){15}}

\put(-5, -10){\line(0, -1){15}}
\put(-2, -20){...}
\put(10, -10){\line(0, -1){15}}
\qbezier(35, -10)(40, -35)(45, -10)
\qbezier(15, -10)(20, -35)(25, -10)
\put(25, -20){...}
\put(25, -30){\tiny $k$}

\put(55, -3){$+$}
\end{picture}
\begin{picture}(80, 60)(-105, -30)
\put(-75, -3){$(r-2-2k)!^{-1}$}

\put(0, 10){\line(1, 0){55}}
\put(0, 25){\line(1, 0){55}}
\put(0, 10){\line(0, 1){15}}
\put(55, 10){\line(0, 1){15}}
\put(18, 15){\tiny $r-2$}

\put(5, 10){\line(0, -1){20}}
\put(15, 10){\line(0, -1){20}}
\put(5, -3){...}

\qbezier(20,10)(25, -7)(30, 10)
\qbezier(40,10)(45, -7)(50, 10)
\put(30, 7){\tiny ...}
\put(33, -2){\tiny$k$}
\qbezier(20,-10)(25, 7)(30, -10)

\put(0, -10){\line(1, 0){40}}
\put(0, -25){\line(1, 0){40}}
\put(0, -10){\line(0, -1){15}}
\put(40, -10){\line(0, -1){15}}
\put(10, -20){\tiny $r-2k$}

\put(5, 25){\line(0, 1){10}}
\put(50, 25){\line(0, 1){10}}
\put(20, 28){$\cdots$}

\put(5, -25){\line(0, -1){10}}
\put(15, -25){\line(0, -1){10}}
\put(6, -33){...}
\qbezier(20,-25)(25, -45)(30, -25)
\put(50, -35){.}
\end{picture}
\]
Using \eqref{eq:k-1} in the second term on the right hand side, we arrive at
the $k+1$ case of \eqref{eq:Sigma}. This completes the proof.
\end{proof}

\section{A covariant functor}
Let $K$ be a field.
Let $V=K^m$ be an $m$-dimensional vector space with a non-degenerate bilinear form $( - , -)$,
which is either symmetric or skew symmetric. When the form is skew symmetric,
non-degeneracy requires $m=2n$ to be even.
Let $G$ denote the isometry group of the form, so that
$G=\{g\in\GL(V)\mid (gv, gw)=(v, w),   \forall  v, w\in V\}$.
Then $G$ is the orthogonal group $\Or(V)$ if the form is symmetric,
and the symplectic group $\Sp(V)$ if the form is skew symmetric.

Given a basis $\{b_1,\dots,b_m\}$ for $V$, let $\{{\bar b}_1,\dots, {\bar b}_m\}$
be the dual basis of $V$, identified with $V^*$ via the map $v(\in V)\mapsto \phi_v(\in V^*)$
where $\phi_v(x):=(v,x)$; thus $( {\bar b}_i,  b_j)=\delta_{ij}$.

For any positive integer $t$, the space $V^{\otimes t}$
is a $G$-module in the usual way: $g(v_1\otimes\dots\otimes v_t)
= g v_1\otimes gv_2\otimes\dots\otimes gv_t$.
Moreover the form on $V$ induces a non-degenerate bilinear form $[-,-]$
on $V^{\otimes t}$, given by $[v_1\otimes\dots\otimes v_t, w_1\otimes\dots\otimes w_t]
:=\prod_{i=1}^t (v_i, w_i)$,
which permits the identification of $V^{\otimes t}$ with its dual space
${V^{\otimes t}}^*=\Hom_K(V^{\otimes t}, K)$.

Define $c_0\in V\ot V$ by $c_0=\sum_{i=1}^m b_i\otimes {\bar b}_i$ in $V\otimes V$. Then $c_0$
is canonical in that it is independent of the basis, and is invariant under $G$.
We shall consider various $G$-equivariant maps $\beta:V^{\ot s}\to V^{\ot t}$
for $s,t\in\Z_{\geq 0}$. Among these we have the following.
\begin{eqnarray}\label{P-C-C}
\begin{aligned}
&&P: V\otimes V\longrightarrow V\otimes V, &\quad v\otimes w \mapsto w\otimes v, \\
&&\check{C}: K \longrightarrow V\otimes V,  &\quad 1\mapsto c_0, \\
&&\hat{C}: V\otimes V \longrightarrow K, &\quad v\otimes w\mapsto \langle v, w \rangle.
\end{aligned}
\end{eqnarray}
They have the following properties.
\begin{lemma} \label{lem:PAU}
Let $\epsilon=\epsilon(G)$ be $1$ (resp. $-1$) if $G=\Or(V)$ (resp. $\Sp(V)$).
Denote the identity map on $V$ by $\id$.
\begin{enumerate}
\item The element $c_0$ belongs to
$(V\otimes V)^G$ and satisfies $P(c_0)=\epsilon c_0$.

\item The maps $P$, $\check{C}$ and $\hat{C}$ are all $G$-equivariant,
and
\begin{eqnarray}
&P^2=\id^{\ot 2}, \quad
(P\ot \id)(\id\ot P)(P\ot \id) = (\id\ot P)(P\ot\id)(\id\ot P),
\label{eq:PPP}\\
&P \check{C} = \epsilon \check{C}, \quad
\hat{C} P = \epsilon  \hat{C}, \label{eq:fse-es}\\
&\hat{C}\check{C}=\epsilon \dim V,  \quad (\hat{C}\ot\id)(\id\ot\check{C})=\id=(\id\ot\hat{C})(\check{C}\ot\id), \label{eq:CC}\\
&(\hat{C}\ot\id)\circ (\id\ot P)= (\id\ot\hat{C})\circ (P\ot\id), \label{eq:CPC-I}\\
&(P\ot \id)\circ(\id\ot \check{C})=(\id\ot P)\circ(\check{C}\ot \id). \label{eq:CPC-P}
\end{eqnarray}
\end{enumerate}
\end{lemma}
\begin{proof}
Equation \eqref{eq:PPP} reflects standard properties of permutations,
and the relations \eqref{eq:fse-es} are evident.
We prove the other relations.
Consider for example $\hat{C}\check{C}=\hat{C}(\sum_i b_i\ot\bar{b}_i)=\sum_i (b_i, \bar{b}_i)$. The far right hand side
is $\sum_i \epsilon = \epsilon\, \dim V$. This proves the first relation of \eqref{eq:CC}.
The proofs of the remaining relations are similar, and therfore omitted.
\end{proof}

\begin{definition}
We denote by $\cT_G(V)$ the full subcategory of $G$-modules with objects $V^{\otimes r}$ ($r=0, 1, \dots$),
where $V^{\otimes 0}=K$ by convention. The usual tensor product of $G$-modules and
of $G$-equivariant maps is a bi-functor $\cT_G(V)\times \cT_G(V)\longrightarrow \cT_G(V)$, which will be
called the tensor product of the category. We call $\cT_G(V)$ the {\em category of
tensor representations of $G$}.
\end{definition}
Note that $\Hom_G(V^{\otimes r}, V^{\otimes t})=0$ unless $r+t$ is even. The zero module is not an object of $\cT_G(V)$,
thus the category is only pre-additive but not additive.

\begin{remark}
The category $\cT_G(V)$ is also a strict monoidal category with a symmetric braiding in the sense of \cite{JS}, where the braiding is given by the permutation maps $V^{\otimes r}\otimes V^{\otimes t}\longrightarrow  V^{\otimes t}\otimes V^{\otimes r}$, $v\ot w\mapsto w\ot v$.
\end{remark}

We have the following result.
\begin{theorem}\label{thm:functor}
There is a unique additive covariant functor $F: \cB(\epsilon m) \longrightarrow \cT_G(V)$
of pre-additive categories with the following properties:
\begin{enumerate}
\item[(i)] $F$ sends the object
$r$ to $V^{\otimes r}$ and morphism $D: k \to \ell$ to
$F(D): V^{\otimes k}\longrightarrow V^{\otimes l}$ where $F(D)$
is defined on the generators of Brauer diagrams by
\begin{eqnarray}\label{eq:F-generating}
\begin{aligned}
F\left(
\begin{picture}(30, 20)(0,0)
\put(15, -15){\line(0, 1){35}}
\end{picture}\right)=\id_V,
\quad&
F\left(
\begin{picture}(30, 20)(0,0)
\qbezier(5, -15)(15, 3)(25, 20)
\qbezier(5, 20)(15, 3)(25, -15)
\end{picture}\right) = \epsilon P, \\
F\left(
\begin{picture}(30, 20)(0,0)
\qbezier(5, 20)(15, -50)(25, 20)
\end{picture}\right) = \check{C}, \quad&
F\left(
\begin{picture}(30, 20)(0,0)
\qbezier(5, -15)(15, 50)(25, -15)
\end{picture}\right) = \hat{C};
\end{aligned}
\end{eqnarray}
\item[(ii)] $F$ respects tensor products, so that for any
objects $r, r'$ and morphisms $D, D'$ in $\cB(\epsilon m)$,
\[
F(r\otimes r')=V^{\otimes r}\otimes V^{\otimes r'}=F(r)\ot F(r'), \quad
\text{ and }F(D\otimes D')= F(D)\otimes F(D').
\]
\end{enumerate}
\end{theorem}
\begin{proof}
We want to show that the functor $F$
is uniquely defined, and gives rise
to an additive covariant functor from $\cB(\epsilon m)$ to $\cT_G(V)$.

By Lemma \ref{lem:PAU}, the linear maps in \eqref{eq:F-generating}
are all $G$-module maps, and by Theorem \ref{thm:presentation}(1),
the above requirements define $F$ on all objects of $\cB(\epsilon m)$;
it is clear that $F$ respects tensor products of objects.
As a covariant functor, $F$ preseves
composition of Brauer diagrams, and by (ii) $F$ respects
tensor products of morphisms. It remains only to show that
$F$ is well-defined.

To prove this, we need to show that the images of the generators satisfy
the relations in Theorem \ref{thm:presentation}(\ref{relations}). This is precisely
the content of equations
\eqref{eq:CC}-\eqref{eq:CPC-P} in Lemma \ref{lem:PAU}(2).

Hence for any morphism $D$ in $\cB(\epsilon m)$,
$F(D)$ is indeed a well defined morphism in $\cT_G(V)$.
\end{proof}

\begin{remark}
The functor $F$ is a tensor functor between braided strict monoidal categories.
\end{remark}

\begin{lemma}\label{lem:transfrom}
Let $H_s^t = \Hom_G(V^{\otimes s}, V^{\otimes t})$ for all $s, t\in \N$.
\begin{enumerate}
\item The $K$-linear maps
\[
\begin{aligned}
&F{\mathbb U}_p^q:=(-\otimes \id_V^{\otimes q})(\id_V^{\otimes p}\otimes F(U_q)):
H_{p+q}^r \longrightarrow H_p^{r+q}, \\
&F{\mathbb A}^r_q:=(\id_V^{\otimes r}\otimes F(A_q))(- \otimes \id_V^{\otimes q}):
H_p^{r+q} \longrightarrow H_{p+q}^r
\end{aligned}
\]
are well defined and are mutually inverse isomorphisms.
\item For each pair $k, \ell$ of objects in $\cB(\epsilon m)$, the functor $F$ induces a linear map
\begin{eqnarray}\label{eq:functionF}
\begin{aligned}
{F}_k^\ell: B_k^\ell(\epsilon m)\longrightarrow H_k^\ell=\Hom_G(V^{\otimes k}, V^{\otimes \ell}), \quad
D \mapsto F(D),
\end{aligned}
\end{eqnarray}
and the following diagrams are commutative.
\begin{displaymath}
    \xymatrix{
        B_p^{r+q}(\epsilon m) \ar[r]^{{\mathbb A}^r_q} \ar[d]_{F_p^{r+q}} &  B_{p+q}^r(\epsilon m) \ar[d]^{F_{p+q}^r} \\
         H_p^{r+q}  \ar[r]_{F{\mathbb A}^r_q}  & H_{p+q}^r}
\quad\quad
\xymatrix{
         B_{p+q}^r(\epsilon m) \ar[r]^{{\mathbb U}_p^q} \ar[d]_{F_{p+q}^r} &  B_p^{r+q}(\epsilon m)  \ar[d]^{F_p^{r+q}} \\
         H_{p+q}^r \ar[r]_{F{\mathbb U}_p^q}  & H_p^{r+q}. }
\end{displaymath}
\end{enumerate}
\end{lemma}
\begin{proof}
Part (1) follows by applying the functor $F$ to Corollary \ref{isomorphism}, using
Theorem \ref{thm:functor}.

Now for any $D\in B_p^{r+q}(\epsilon m)$,
${\mathbb A}^r_q(D)=(I_{r+q}\otimes A_q)\circ(D \otimes I_q)$. Since
$F$ preserves both composition and tensor product of Brauer diagrams,
\[
\begin{aligned}
F({\mathbb A}^r_q(D)) &= (\id_V^{\otimes (r+q)}\otimes F(A_q))(F(D) \otimes \id_V^{\otimes q})\\
                    &=F{\mathbb A}^r_q(F(D)).
\end{aligned}
\]
This proves the commutativity of the first diagram in part (2). The commutativity of the other diagram
is proved in the same way.
\end{proof}

We shall require the next lemma, which is surely well known. Nevertheless, we supply a proof by
adapting some computations in \cite{ZGB} to the present context.
\begin{lemma}\label{lem:jtrace}
For any endomorphism $L\in\End_K(V^{\ot r})$ define the Jones trace $J(L)$ by
\be\label{eq:jt}
J(L)=F(A_r)\circ (L\ot \id_V^{\ot r})\circ F(U_r)\in \End_K(K)\simeq K,
\ee
where $A_r$ and $U_r$ are the capping and cupping operations defined above.
Then $\Tr(L, V^{\ot r})=\ep^rJ(L)$.

In particular, if $L=F(D)$, for $D\in B_r^r(\ep m)$, we have
$$
\Tr(F(D),V^{\ot r})=\ep^rA_r\circ (D\ot I_r)\circ U_r:=J(D).
$$
The map $J: B_r^r(\ep m)\to K$ is referred to as the Jones trace on the Brauer
algebra.
\end{lemma}
\begin{proof}
Let $L\in\End_K(V^{\ot r})$. Since \eqref{eq:jt} is linear in $L$, it suffices to
prove it for $L=L_1\ot\dots\ot L_r$, where $L_i\in\End_K(V)$ for each $i$. Now observe that
if we write $\Gamma:\End_K(V^{\ot i})\to \End_K(V^{\ot {(i-1)}})$ for the map
defined by
$$
\Gamma(M)=(\id_V^{\ot (i-1)}\ot F(A))\circ (M\ot \id_V)\circ (\id_V^{\ot (i-1)}\ot F(U)),
$$
then $J(L)=\Gamma^r(L)$. We therefore compute $\Gamma(L)$. We have
$$
\begin{aligned}
\Gamma(L)(v_1\ot\dots&\ot v_{r-1})\\
=&(\id_V^{\ot (r-1)}\ot F(A))\circ (L\ot \id_V)\circ (\id_V^{\ot (r-1)}\ot\check C)
(v_1\ot\dots\ot v_{r-1}\ot 1)\\
=&(\id_V^{\ot (r-1)}\ot F(A))\circ (L\ot \id_V)
(v_1\ot\dots\ot v_{r-1}\ot c_0)\\
=&(\id_V^{\ot (r-1)}\ot F(A))
(L_1v_1\ot\dots\ot L_{r-1}v_{r-1}\ot\sum_iL_rb_i\ot \bar b_i)\\
=&(\id_V^{\ot (r-1)}\ot \hat C)
(L_1v_1\ot\dots\ot L_{r-1}v_{r-1}\ot\sum_iL_rb_i\ot \bar b_i)\\
=&\sum_i(L_rb_i, \bar b_i)(L_1v_1\ot\dots\ot L_{r-1}v_{r-1})\\
=&\ep \Tr(L_r,V)(L_1v_1\ot\dots\ot L_{r-1}v_{r-1}).
\end{aligned}
$$
It follows that $\Gamma(L_1\ot\dots\ot L_r)=\ep\Tr(L_r,V)L_1\ot\dots\ot L_{r-1}$,
and hence by induction that $J(L)=\Gamma^r(L)=\ep^r\Tr(L,V^{\ot r})$. The result follows.
\end{proof}
\section{Theory of invariants of the orthogonal and symplectic groups.}
Henceforth we assume that $K$ is a field of characteristic zero.

\subsection{The fundamental theorems of invariant theory}

Let $G$ be either the orthogonal group $\Or(V)$ or the symplectic group $\Sp(V)$.
For any  $t\in\N$, the space $V^{\otimes t}$
is a $G$-module, and hence so is its dual space ${V^{\otimes t}}^* =\Hom_{K}(V^{\otimes t}, K)$.
The space of invariants $({V^{\otimes t}}^*)^G=\Hom_G(V^{\otimes t}, K)$
consists of linear functions on $V^{\otimes t}$ which are constant on $G$-orbits.
One formulation of the first fundamental theorem of classical invariant theory for
the orthogonal and symplectic groups \cite{W, GW} is as follows.

\begin{theorem}\label{thm:fft}
The space $({{V^{\otimes t}}^*})^G$ is zero if $t$ is odd.
If $t=2r$ is even, any element of $({{V^{\otimes t}}^*})^G$ is
a linear combination of maps of the form $\gamma_\alpha$ ($\alpha\in\Sym_{2r}$),
where
\begin{eqnarray}\label{eq:lnr-fctn}
\gamma_\alpha: v_1\otimes\dots\otimes v_{2r}\mapsto \prod_{i=1}^r( v_{\alpha(2i-1)}, v_{\alpha(2i)}),
\end{eqnarray}
\end{theorem}

Now $\Sym_{2r}$ evidently acts transitively on the set of $\gamma_\alpha$ through its action on
$V^{\ot r}$ by place permutations: for $\pi\in\Sym_{2r}$,
$\pi.\gamma_\alpha:=\gamma_\alpha\circ\pi\inv=\gamma_{\pi\alpha}$.
Moreover the centraliser $H$ in $\Sym_{2r}$ of
the involution $(12)(34)\dots(2r-1,2r)$, which is isomorphic to $\Sym_r\ltimes(\Z/2\Z)^r$,
clearly takes $\gamma_1$ to $\pm \gamma_1$. Hence if $\CT_r:=\Sym_{2r}/(\Sym_r\ltimes(\Z/2\Z)^r)$
is a left transversal of $\Sym_r\ltimes(\Z/2\Z)^r$ in $\Sym_{2r}$, it follows that
each function $\gamma_\alpha$ is equal to $\pm\gamma_\beta$, with $\beta\in\CT_r$,
and hence that

\begin{corollary}\label{cor:fft}
With notation as in Theorem \ref{thm:fft}, and writing $\CT_r$ for
the transversal above, $({{V^{\otimes t}}^*})^G$
is spanned by $\{\gamma_\alpha\mid\alpha\in\CT_r\}$.
\end{corollary}

\begin{remark}\label{rem:diagbij}
Note that since $\CT_r:=\Sym_{2r}/(\Sym_r\ltimes(\Z/2\Z)^r)$ is evidently identified with
the set of all pairings of the elements of $\{1,2,\dots,2r\}$, $\CT_r$ is in bijection with
the diagrams in $B_r^r(\ep m)$.
\end{remark}


For any subset $S\subseteq [1,t]$,
let $\Sym(S)$ be the symmetric group of $S$, regarded as the subgroup of $\Sym_t$
which fixes all elements in $[1, t]\setminus S$.
%
The next lemma provides some linear relations among the $\gamma_\alpha$.

\begin{lemma}\label{lem:gammainker} Let $S$ be any subset of $[1,t]$ with $|S|=m+1$.
Then for any $\gamma=\gamma_\alpha$ as in \eqref{eq:lnr-fctn}, we have
$\sum_{\pi\in\Sym(S)}(-1)^{|\pi|}\pi\gamma=0$. In particular, for $\alpha\in\CT_r$,
\be\label{eq:basicrel}
\sum_{\pi\in\Sym(S)}(-1)^{|\pi|}\gamma_{\alpha\pi}=0.
\ee
\end{lemma}
\begin{proof}
For any $S\subset [1, t]$ of cardinality $m+1$ and $\gamma\in ({V^{\ot t}}^*)^G$, we have
\be\label{eq:zero}
\begin{aligned}
&\sum_{\pi\in\Sym(S)}(-1)^{|\pi|}\pi\gamma(v_1\otimes\dots\otimes v_t) \\
&=\sum_{\pi\in\Sym(S)}(-1)^{|\pi|}\gamma(\pi\inv(v_1\otimes\dots\otimes v_t))\\
&=\gamma(\sum_{\pi\in\Sym(S)}(-1)^{|\pi|}\pi\inv(v_1\otimes\dots\otimes v_t))\\
&=0,\\
\end{aligned}
\ee
since $\Sym(S)$ acts on $m+1$ positions, and therefore the alternating sum has a factor
which is an element of $\Lambda^{m+1}(V)$, which is zero since $m=\dim(V)$.
\end{proof}

\begin{remark}\label{rem:pertinent}

(i) When the form is symmetric, the inner sum in the third line of Equation \eqref{eq:zero}
may be zero for the trivial reason that an involution in $S$ might fix $\gamma$.
Thus some of the relations above are trivial in the orthogonal case.

(ii) Although $\alpha\pi$ may not be in $\CT_r$ above, it is always the case that
$\gamma_{\alpha\pi}=\pm\gamma_\beta$ for some $\beta\in\CT$. Thus the Lemma does
provide linear relations among the $\gamma_\alpha$ for $\alpha\in\CT_r$.
\end{remark}

The second fundamental theorem for the orthogonal and symplectic groups \cite{W}
may be stated as follows \cite{GW}.
\begin{theorem}\label{thm:sft}
Write $m=\dim( V)$ and let $d=m$ if $G=\Or(V)$, and $d=\frac{m}{2}$ if $G=\Sp(V)$.
%
%
If $r\le d$, the linear functions $\{\gamma_\alpha\mid\alpha\in\CT_r\}$
of Corollary \ref{cor:fft} form a basis
of the space of $G$-invariants on $V^{\otimes 2r}$. If $r> d$,
any linear relation among the functionals $\gamma_\alpha$ is a linear consequence
of the relations in Lemma \ref{lem:gammainker}.
\end{theorem}

\subsection{Categorical generalisations of the fundamental theorems}

We now return to the category $\cB(\epsilon m)$ of Brauer diagrams with parameter $\epsilon m$ (where
$\epsilon = \epsilon(G)$) and the covariant functor $F: \cB(\epsilon m)\longrightarrow \cT_G(V)$.
Recall that the group algebra $K\Sym_r$ is embedded in the Brauer algebra $B_r(\epsilon m)$
of degree $r$. In particular, $\Sigma_\epsilon(r)$ belongs to $B_r^r(\epsilon m)$.
Let $\phi_r=F(\Sigma_\epsilon(r))\in\End_G(V^{\otimes r})=H_r^r$. Then for any
$r$ vectors $v_i$ in $V$,
\[
\phi_r(v_1\otimes v_2\otimes \dots \otimes v_r)
= \sum_{\sigma\in \Sym_r}(-1)^{|\sigma|}v_{\sigma(1)}
\otimes v_{\sigma(2)}\otimes \dots \otimes v_{\sigma(r)}.
\]
In particular, if $r=m+1$, then $\phi_r=0$ as an element in $H_{m+1}^{m+1}$.

\begin{definition}
Denote by $\langle \Sigma_\epsilon(m+1)\rangle$ the subspace of $\oplus_{k,\ell}B_k^\ell(\epsilon m)$
spanned by the morphisms
in $\cB(\epsilon m)$ obtained from $\Sigma_\epsilon(m+1)$ by composition and tensor product.
Set $\langle \Sigma_\epsilon(m+1)\rangle_k^\ell =
\langle \Sigma_\epsilon(m+1)\rangle\cap B_k^\ell(\epsilon m)$.
\end{definition}

The first and second fundamental theorems of classical invariant theory for
the orthogonal and symplectic groups
can be respectively interpreted as parts (1) and (2) of the following theorem.
\begin{theorem} \label{thm:fft-sft} Assume that $K$ has characteristic $0$ and
write $d=m$ if $G=\Or(V)$, and $d=\frac{m}{2}$ if $G=\Sp(V)$, where $m=\dim( V)$.
\begin{enumerate}
\item  The functor $F: \cB(\epsilon m)\longrightarrow \cT_G(V)$ is full.
That is, $F$ is surjective on $\Hom$ spaces.
\item
The map ${F}_k^\ell$ is injective if $k+\ell\le 2d$, and
$\Ker{F}_k^\ell=\langle \Sigma_\epsilon(m+1)\rangle_k^\ell$
if $k+\ell> 2d$.
\end{enumerate}
\end{theorem}
\begin{proof}
It follows from Lemma \ref{lem:transfrom} that we have a canonical
isomorphism $B_k^\ell\simeq B_{k+\ell}^0$, and the
study of $F_k^\ell$ is equivalent to that of $F_{k+\ell}^0$. Hence without loss of
generality, we may assume that $\ell=0$. When $\ell=0$,
the theorem is true trivially when $k$ is odd.
Thus we only need to consider the case $\ell=0$ and $k=2r$.

(1). By Corollary \ref{cor:fft}, every element of $H_{2r}^0$ is a linear
combination of functionals $\gamma_\alpha$ for $\alpha\in\CT_r$. As remarked in
Remark \ref{rem:diagbij}, the elements of $\CT_r$ are in canonical bijection
with pairings of the set $[1,2r]$, i.e. the partitioning of $[1,2r]$ into
a disjoint union of pairs. Let $D$ be the diagram corresponding to $\alpha\in\CT_r$.
Then $F(D)=\gamma_\alpha$.
Thus $F_{2r}^0$ is surjective, and so is also $F_k^\ell$ for all $k$ and $\ell$.
This proves part (1) of the theorem.

(2). Note that every $(2r, 0)$ Brauer diagram is mapped by $F$ to a $\gamma_\alpha$ of
the form \eqref{eq:lnr-fctn}. Thus if $r\le d$, then $\Ker{F}_{2r}^0=0$
by Theorem \ref{thm:sft}, the second fundamental theorem.

Now consider the case $r>d$. By Theorem \ref{thm:sft} it suffices
to show that every relation of the form \eqref{eq:basicrel} arises by applying
$F_{2r}^0$ to an element of $\langle \Sigma_\epsilon(m+1)\rangle_{2r}^0$.
Fix $\alpha\in\CT_r$, and let $D\in B_{2r}^0$ be the diagram such that $F(D)=\gamma_\alpha$.
This is the diagram corresponding to $\alpha\in\CT_r$ by Remark \ref{rem:diagbij}.

Write $\gamma_\alpha(S)=\sum_{\pi\in\Sym(S)}(-1)^{|\pi|}\gamma_{\alpha\pi}$ for the
left side of \eqref{eq:basicrel}.

If $\sigma\in\Sym_{2r}$ satisfies $\{\sigma([1,m+1])\}=S$,
then $\Sym(S)=\sigma \Sym({[1, m+1]}) \sigma\inv$. Now regard $\Sym_{2r}$ as
embedded in $B_{2r}^{2r}(\epsilon m)$, and define the element
$$
D_S:= \sum_{\pi\in \Sym({[1, m+1]})} (-\epsilon)^{|\pi|}D\circ\sigma\circ \pi\circ \sigma\inv
$$
in $B_{2r}^0(\epsilon m)$.
Then $\gamma_\alpha(S)=F(D_S)$, and since we have
\[
\begin{aligned}
D_S &=D\circ\sigma\circ\Sigma_\epsilon(m+1)\circ\sigma\inv
&\in \langle \Sigma_\epsilon(m+1) \rangle_{2r}^0,
\end{aligned}
\]
it follows that $\gamma_\alpha(S)\in F(\langle \Sigma_\epsilon(m+1) \rangle_{2r}^0)$.

By Theorem \ref{thm:sft}, all relations among invariant functionals on $V^{\otimes 2r}$
are linear consequences of the relations $\gamma_\alpha(S)=0$.  Using the bijection
between diagrams and $\CT_r$, it follows that $\Ker{F}_{2r}^0$ is spanned
by elements of the form $D_S$.

Conversely, it is evident that
$\langle \Sigma_\epsilon(m+1) \rangle_{2r}^0\subset \Ker{F}_{2r}^0$ since $\phi_{m+1}=F(\Sigma_\epsilon(m+1))=0$.
This proves part (2) for $r>d$, completing the proof of the Theorem.
\end{proof}

\begin{remark}
Theorem \ref{thm:fft-sft}(1) with $k=2r,\ell=0$ yields the linear version of FFT,
while the endomorphism algebra formulation arises from the case $k=\ell=r$. The equivalence of
the two versions is an obvious consequence of Lemma \ref{lem:transfrom}.
\end{remark}

\begin{corollary}\label{cor:small-degree}
If $k+\ell\le 2d$, then $\langle\Sigma_\epsilon(m+1)\rangle_k^\ell=0$.
\end{corollary}
\begin{proof}
Since $\phi_{m+1}=0$, $\langle \Sigma_\epsilon(m+1)\rangle_k^\ell$ is contained in $\Ker{F}_k^\ell$. But
$\Ker{F}_k^\ell=0$ for $k+\ell\le 2d$, and the lemma follows.
\end{proof}

\section{Structure of the endomorphism algebra: the symplectic case}

Recall from Section \ref{sect:Brauer-algebra} that $B_r^r(\epsilon m)$ is the Brauer algebra of degree $r$.
Thus
$\Ker{F_r^r}$ is a two-sided ideal of $B_r^r(\epsilon m)$, and $B_r^r(\epsilon m)/Ker{F}_r^r$
is canonically isomorphic to the endomorphism algebra $\End_G(V^{\ot r})$ by Theorem \ref{thm:fft-sft}(2).
In order to understand the algebraic structure of $\End_G(V^{\ot r})$, we need to understand that
of $Ker{F}_r^r$, and this is what we shall do in this section and the next section.

Here we take $G=\Sp(V)$ with $\dim V = 2n$ and $\epsilon=-1$. Denote $\Sigma_{-1}(r)$ by $\Sigma(r)$.

\subsection{Generators of the kernel}

For any $s<r$, there is a natural embedding $B_s^s(-2n)\hookrightarrow B_r^r(-2n)$,
$b\mapsto b\ot I_{r-s}$,
of the Brauer algebra of degree $s$ in that of degree $r$ as associative algebras. Thus we may
regard $B_s^s(-2n)$ as the subalgebra of $B_r^r(-2n)$ consisting of elements of the form
$b\ot I_{r-s}$.

Let $D(p, q)$ denote the element of the Brauer algebra $B_k^k(-2n)$ of degree $k=2n+1-p+q$ shown in
Figure \ref{Dpq}.

\begin{figure}[h]
\begin{center}
\begin{picture}(100, 60)(0,0)
\put(-50, 30){$D(p, q)=$}
\put(5, 40){\line(0, 1){20}}
\put(6, 50){...}
\put(18, 40){\line(0, 1){20}}

\put(30, 58){\tiny${p-q}$}
\qbezier(22, 40)(36, 70)(52, 40)
\put(35, 52){.}
\put(35, 50){.}
\put(35, 48){.}
\qbezier(30, 40)(36, 54)(44, 40)

\qbezier(58, 40)(100, 90)(108, 0)
\qbezier(72, 40)(95, 70)(97, 0)
\put(97, 7){...}
\put(98, -5){\tiny$q$}

\put(0, 20){\line(1, 0){80}}
\put(0, 20){\line(0, 1){20}}
\put(80, 20){\line(0, 1){20}}
\put(0, 40){\line(1, 0){80}}
\put(30, 28){\tiny${2n+1}$}

\put(5, 20){\line(0, -1){20}}
\put(20, 10){...}
\put(50, 20){\line(0, -1){20}}

\qbezier(58, 20)(95, -30)(108, 60)
\qbezier(72, 20)(95, -10)(97, 60)
\put(97, 53){...}
\put(98, 60){\tiny$p$}
\end{picture}
\end{center}
\caption{}
\label{Dpq}
\end{figure}

\begin{proposition}\label{lem:sp-generat}
Assume that $r>n$. As a two-sided ideal of the Brauer algebra $B_r^r(-2n)$,
$\Ker{F}_r^r$  is generated by
$D(p, q)$ and $\ast D(p, q)$ with $p+q\le r$ and $p\le n$.
\end{proposition}
\begin{proof}
Let $A$ be a single $(2r, 0)$ Brauer diagram with $r>n$. Then $F(A)$ is some functional $\gamma$
on $V^{\ot 2r}$  defined by \eqref{eq:lnr-fctn}.  For any $\pi\in\Sym_{2r}\subset B_{2r}^{2r}(-2n)$,
$A\circ \pi$ is defined. Note that $A$ has only one row of vertices at the bottom,
which will be labelled $1, 2, \dots, 2r$ from left to right.
Choose a subset $S$ of $[1, 2r]$
of cardinality $2n+1$ as in Lemma \ref{lem:gammainker}, and consider
$\Sym_S\subset \Sym_{2r}\subset B_{2r}^{2r}(-2n)$. Define
\begin{eqnarray}\label{eq:AS}
A_S=\sum_{\pi\in\Sym_S} A\circ\pi.
\end{eqnarray}
Then by Theorem \ref{thm:sft}, and equivalently
Theorem \ref{thm:fft-sft}(2),
$\Ker{F}_{2r}^0$ is spanned by $A_S$ for all $A$ and $S$. Given $A_S$, we define
\[
A_S^\natural=A_S\circ(I_r\otimes U_r)\in B_r^r(-2n).
\]
Then $\Ker{F}_{r}^r$ is spanned by $A_S^\natural$ for all $A$ and $S$ by Lemma \ref{lem:transfrom}(2).

We can considerably simplify the description of $\Ker{F}_{2r}^0$ and $\Ker{F}_{r}^r$.
There exist elements $\sigma=(\sigma_1, \sigma_2)$ in the parabolic subgroup $\Sym_r\times\Sym_r$ of $\Sym_{2r}$, which map $S$ to $S'=\{i+1, i+2, \dots, i+2n+1\}\subset[1, 2r]$ for  some $i\le 2r-2n-1$. Let $\sigma_2^{-\tau}=\ast(\sigma_2^{-1})$, where $\ast$ is the anti-involution of $B_r^r(-2n)$. Then
\begin{eqnarray}
&\sigma_2^{-\tau}\circ A_S^\natural\circ \sigma_1^{-1}=(A_S\circ \sigma^{-1})^\natural,\nonumber\\
&A_S\circ \sigma^{-1}=\sum_{\pi\in\Sym_{S'}}(A\circ\sigma^{-1})\circ\pi. \label{eq:symmS}
\end{eqnarray}

By appropriately choosing $\sigma$, we can ensure that $A\circ\sigma^{-1}$ is of the
form shown in Figure \ref{Asigma}.
\begin{figure}[h]
\begin{center}
\begin{picture}(290, 60)(0,0)
\qbezier(-3, 0)(140, 106)(283, 0)
\qbezier(15, 0)(140, 85)(265, 0)
\put(140, 50){.}
\put(140, 48){.}
\put(140, 46){.}

\qbezier(25, 0)(35, 15)(35, 0)
\qbezier(50, 0)(60, 35)(65, 0)
\put(40, 3){...}\put(42, -4){\tiny$t$}

\qbezier(215, 0)(220, 35)(230, 0)
\qbezier(245, 0)(245, 15)(255, 0)
\put(232, 3){...}\put(234, -4){\tiny$t'$}

\qbezier(70, 0)(100, 65)(120, 0) \put(118, -2){$\bullet$}
\qbezier(85, 0)(110, 65)(155, 0)\put(153, -2){$\bullet$}
\put(75, 3){...}

\qbezier(145, 0)(155, 70)(210, 0)
\qbezier(135, 0)(155, 60)(195, 0)
\put(133, -2){$\bullet$}\put(143, -2){$\bullet$}
\put(125, 3){...}
\put(196, 3){...}

\qbezier(105, 0)(135, 50)(170, 0)
\put(103, -2){$\bullet$}\put(168, -2){$\bullet$}
\put(90, 3){...}
\put(176, 3){...}
\end{picture}
\end{center}
\caption{}
\label{Asigma}
\end{figure}
The vertices labeled by $\bullet$ are those in $S'$, which all appear in the middle, and the other vertices
all appear at the left end and right end. Here $t$ denotes the number of edges in $A\circ \sigma^{-1}$
with both vertices in $\{1, 2, \dots, i\}$,
and $t'$ that of the edges with both vertices in $\{i+2n+2, i+2n+3, \dots, 2r\}$.
Note that after such a $\sigma$ is chosen, $\pi\in \text{Sym}_{S'}$ acting on $A\circ \sigma^{-1}$
permutes only vertices labeled by $\bullet$.  Thus every term on the right hand side of \eqref{eq:symmS}
is of the form Figure \ref{Asigma} with the same $t$ and $t'$.

Now $(A_S\circ \sigma^{-1})^\natural$ can be expressed as
$D_1\otimes D_2$, where $D_1\in B_{r_1}^{r_1}(-2n)$ for $r_1$ maximal, $D_2\in B_k^k(-2n)$
with $k>n$ satisfying $r_1+k=r$.
There are several possibilities for $D_2$ depending on
$i$, $t$ and $t'$. Assume $i+2n+1> r$.  If
$t=t'$, then $D_2$ is as shown in Figure \ref{Dpq-1}.
\begin{figure}[h]
\begin{center}
\begin{picture}(100, 60)(-20,0)

\put(0, 20){\line(1, 0){80}}
\put(0, 20){\line(0, 1){20}}
\put(80, 20){\line(0, 1){20}}
\put(0, 40){\line(1, 0){80}}
\put(30, 28){\tiny${2n+1}$}

\qbezier(5, 40)(-15, 60)(-18, 0)
\qbezier(18, 40)(-20, 80)(-32,0)
\put(-28, 10){...}
\put(-28, 0){\tiny$q$}

\put(30, 58){\tiny${p-q}$}
\qbezier(22, 40)(36, 70)(52, 40)
\put(35, 52){.}
\put(35, 50){.}
\put(35, 48){.}
\qbezier(30, 40)(36, 54)(44, 40)

\put(58, 40){\line(0, 1){20}}
\put(73, 40){\line(0, 1){20}}
\put(61, 50){...}

\put(5, 20){\line(0, -1){20}}
\put(20, 10){...}
\put(50, 20){\line(0, -1){20}}

\qbezier(55, 20)(93, -30)(103, 60)
\qbezier(70, 20)(90, -10)(91, 60)
\put(92, 53){...}
\put(95, 60){\tiny$p$}
\end{picture}
\end{center}
\caption{}
\label{Dpq-1}
\end{figure}
If $t<t'$, then $D_2=E\circ(I_s\otimes D_3)$ for some $s$, where $D_3$ is as shown in Figure \ref{Dpq-1},
and $E$ is the product of some $e_i$'s composed with a permutation in
$\Sym_{2n+1+q-p}$ ($D_3$ and $E$ may not be unique). Analogously, $D_2=(D_3\otimes I_{s})\circ E$ if $t>t'$.
Assume that $i+2n+1\le r$. Then $D_2=E\circ(I_{s_1}\otimes \Sigma(2n+1)\otimes I_{s_2})$ for some
$E$ in $B_k^k(-2n)$, and fixed nonnegative integers $s_1$ and $s_2$ satisfying $s_1+s_2+2n+1=k$.

Therefore, $\Ker{F}_r^r$  is generated as a two sided ideal of $B_r^r(-2n)$ by elements
of the form of Figure \ref{Dpq-1} with $2n+1+q-p\le r$.
If $p>n$, we apply the anti-involution $\ast$ of $B_k^k(-2n)$ to the element of Figure \ref{Dpq-1} to obtain
the element shown in Figure \ref{Dpq-2},
\begin{figure}[h]
\begin{center}
\begin{picture}(100, 60)(-20,0)

\put(0, 20){\line(1, 0){80}}
\put(0, 20){\line(0, 1){20}}
\put(80, 20){\line(0, 1){20}}
\put(0, 40){\line(1, 0){80}}
\put(30, 28){\tiny${2n+1}$}

\put(5, 40){\line(0, 1){20}}
\put(20, 40){\line(0, 1){20}}
\put(10, 50){...}
\put(10, 55){\tiny$q$}

\put(35, 58){\tiny${p-q}$}
\qbezier(27, 40)(41, 70)(57, 40)
\put(40, 52){.}
\put(40, 50){.}
\put(40, 48){.}
\qbezier(35, 40)(41, 54)(49, 40)

\qbezier(62, 40)(90, 80)(102, 0)
\qbezier(75, 40)(85, 60)(90, 0)
\put(90, 10){...}

\put(30, 20){\line(0, -1){20}}
\put(75, 20){\line(0, -1){20}}
\put(48, 10){...}
\put(48, 0){\tiny$p$}

\qbezier(5, 20)(-5, -5)(-10, 60)
\qbezier(20, 20)(-15, -25)(-25, 60)
\put(-22, 50){...}

\end{picture}
\end{center}
\caption{}
\label{Dpq-2}
\end{figure}
which we denote by $D$. Recall the element $X_{s, t}$ of Figure \ref{X},
which belongs to $\Sym_{s+t}$, where $\Sym_{s+t}$ is regarded as embedded in $B_{s+t}^{s+t}(-2n)$.
Then $X_{2n+1-p, q}\circ D\circ X_{2n+1-2p+q, p}$ is
of the form shown in Figure \ref{Dpq-1}, but with $p$ replaced by $2n+1-p\le n$.

Therefore, we only need to consider Figure \ref{Dpq-1} with $p\le n$ and its $\ast$ image.
Post-composing $X_{2n+1-p, q}$ to Figure \ref{Dpq-1} turns the latter into the form shown in Figure \ref{Dpq}.
Since $X_{2n+1-p, q}$ is invertible in $B_r^r(-2n)$,
$\Ker{F}_r^r$ as a two-sided ideal of $B_r^r(-2n)$ is generated by elements of
$D(p, q)$ and $\ast D(p, q)$ with $2n+1+q-p\le r$ and $p\le n$.
\end{proof}

\subsection{The element $\Phi$}

For each $k$ such that $0\le k\le \left[\frac{n+1}{2}\right]$,
define the element $E(k)= \prod_{j=1}^k e_{n+2-2j}$ of
$B_{n+1}^{n+1}(-2n)$, where $E(0)$ is the identity by convention.
Then define
\[
\Xi_k= \Sigma(n+1) E(k) \Sigma(n+1),
\]
which may be represented pictorially as

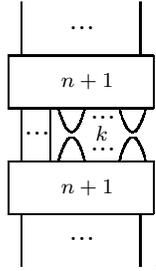
\begin{center}

\begin{picture}(100, 100)(0,-40)
\put(5, 40){\line(0, 1){20}}
\put(23, 50){...}
\put(50, 40){\line(0, 1){20}}

\put(0, 20){\line(1, 0){55}}
\put(0, 20){\line(0, 1){20}}
\put(55, 20){\line(0, 1){20}}
\put(0, 40){\line(1, 0){55}}
\put(20, 28){\tiny${n+1}$}

\put(5, 20){\line(0, -1){20}}
\put(6, 10){...}
\put(16, 20){\line(0, -1){20}}

\qbezier(19, 20)(24, 2)(29, 20)
\put(31, 16){...}
\qbezier(42, 20)(47, 2)(52, 20)
\put(33, 8){\tiny$k$}
\qbezier(19, 0)(24, 18)(29, 0)
\put(31, 4){...}
\qbezier(42, 0)(47, 18)(52, 0)

\put(0, 0){\line(1, 0){55}}
\put(0, 0){\line(0, -1){20}}
\put(55, 0){\line(0, -1){20}}
\put(0, -20){\line(1, 0){55}}
\put(20, -12){\tiny${n+1}$}
\put(5, -20){\line(0, -1){20}}
\put(23, -30){...}
\put(50, -20){\line(0, -1){20}}
\put(55, -40){.}
\end{picture}
\end{center}
Now define the following element of $B_{n+1}^{n+1}(-2n)$.
\begin{eqnarray}\label{Phi}
\begin{aligned}
\Phi=\sum_{k=0}^{\left[\frac{n+1}{2}\right]} a_k \Xi_k \quad \text{with} \quad
a_k =\frac{1}{(2^k k! )^2 (n+1-2k)!}.
\end{aligned}
\end{eqnarray}

\begin{lemma}\label{lem:Phi-Z}
The element $\Phi$ is a linear combination of Brauer diagrams with integral coefficients,
thus is defined over the ring $\Z$ of integers.
\end{lemma}
\begin{proof}
We claim that each $\frac{\Xi_k}{(2^k k! )^2 (n+1-2k)!}$
is an integral sum of Brauer diagrams despite the appearance of the denominator.

This is obvious when $k=0$ since $\frac{\Xi_0}{(n+1)!}=\Sigma(n+1)$.

For $k>0$, we let $\tau:=\Sigma(n+1)\circ(I_{n+1-2k}\otimes U^{\otimes k})$ and
$\beta:=(I_{n+1-2k}\otimes A^{\otimes k})\circ\Sigma(n+1)$.
Then $\Xi_k=\tau\circ\beta$. It is important to observe that both
$\tau$ and $\beta$ are invariant
under the interchange of the end points of each
$A$ or $U$ and under permutations of the $A$ factors or $U$ factors. Thus they are
$2^k k!$ multiples of $\Z$-linear combinations of Brauer diagrams; that is,
$\frac{\tau}{2^k k!}$ and $\frac{\beta}{2^k k!}$ are $\Z$-linear combinations of
Brauer diagrams.

It is obvious from the symmetry of $\Sigma(n+1)$ that
any permutation of the bottom $n+1-2k$ vertices in $\tau$ does not change
$\tau$; similarly permutations of the top $n+1-2k$ vertices in $\beta$
do not change $\beta$. Thus $\frac{\tau}{2^k k!}\circ\frac{\beta}{2^k k!}$ is a
$(n+1-2k)!$ multiple of a $\Z$-linear combination of Brauer diagrams.

This proves the claim, and hence the lemma.
\end{proof}

We have the following result.
\begin{lemma}\label{lem:Phi}
The element $\Phi$ has the following properties:
\begin{enumerate}
\item \label{lem:Phi-1} $e_i\Phi=\Phi e_i=0$ for all $e_i\in B_{n+1}^{n+1}(-2n)$;
\item \label{lem:Phi-2} $\Phi^2 = (n+1)!\Phi$;
\item \label{lem:Phi-3} $\ast\Phi=\Phi$;
\item \label{lem:Phi-4} $\Phi\in \Ker{F}_{n+1}^{n+1}$.
\end{enumerate}
\end{lemma}

\begin{proof}
Part (\ref{lem:Phi-3}) follows from the fact that
$\ast\Xi_k=\Xi_k$ for all $k$. Part (\ref{lem:Phi-2}) immediately follows from (\ref{lem:Phi-1}).

Since $*(e_i\circ\Phi)=\Phi\circ e_{n+1-i}$, we only need to show that $e_i\circ\Phi=0$
for all $i$ in order to prove part (\ref{lem:Phi-1}).
In view of the symmetrising property of $\Sigma(n+1)$,
it suffices to show that $e_n\circ\Phi=0$. Consider
$(I_{n-1}\otimes A_1)\circ \Xi_k$, which can be shown to be equal to
\begin{eqnarray}\label{eq:ASigma}
\begin{aligned}
\begin{picture}(100, 100)(20,-40)
\put(-30, 8){$-4k^2$}
\put(5, 40){\line(0, 1){20}}
\put(23, 50){...}
\put(50, 40){\line(0, 1){20}}

\put(0, 20){\line(1, 0){55}}
\put(0, 20){\line(0, 1){20}}
\put(55, 20){\line(0, 1){20}}
\put(0, 40){\line(1, 0){55}}
\put(20, 28){\tiny${n-1}$}

\put(5, 20){\line(0, -1){20}}
\put(6, 10){...}
\put(16, 20){\line(0, -1){20}}

\qbezier(19, 20)(24, 2)(29, 20)
\put(31, 16){...}
\qbezier(42, 20)(47, 2)(52, 20)
\put(30, 8){\tiny$k$-1}
\qbezier(19, 0)(24, 18)(29, 0)
\put(31, 4){...}
\qbezier(42, 0)(47, 18)(52, 0)

\qbezier(57, 0)(62, 18)(67, 0)

\put(0, 0){\line(1, 0){70}}
\put(0, 0){\line(0, -1){20}}
\put(70, 0){\line(0, -1){20}}
\put(0, -20){\line(1, 0){70}}
\put(20, -12){\tiny${n+1}$}
\put(5, -20){\line(0, -1){20}}
\put(28, -30){$\dots$}
\put(65, -20){\line(0, -1){20}}

\put(80, 8){$+$}
\end{picture}
\begin{picture}(100, 100)(-90,-40)
\put(-110, 8){$(n+1-2k)(n-2k)$}
\put(5, 40){\line(0, 1){20}}
\put(23, 50){...}
\put(50, 40){\line(0, 1){20}}

\put(0, 20){\line(1, 0){55}}
\put(0, 20){\line(0, 1){20}}
\put(55, 20){\line(0, 1){20}}
\put(0, 40){\line(1, 0){55}}
\put(20, 28){\tiny${n-1}$}

\put(5, 20){\line(0, -1){20}}
\put(6, 10){...}
\put(16, 20){\line(0, -1){20}}

\qbezier(19, 20)(24, 2)(29, 20)
\put(31, 16){...}
\qbezier(42, 20)(47, 2)(52, 20)
\put(33, 8){\tiny$k$}
\qbezier(19, 0)(24, 18)(29, 0)
\put(31, 4){...}
\qbezier(42, 0)(47, 18)(52, 0)

\qbezier(57, 0)(62, 18)(67, 0)

\put(0, 0){\line(1, 0){70}}
\put(0, 0){\line(0, -1){20}}
\put(70, 0){\line(0, -1){20}}
\put(0, -20){\line(1, 0){70}}
\put(20, -12){\tiny${n+1}$}
\put(5, -20){\line(0, -1){20}}
\put(28, -30){$\dots$}
\put(65, -20){\line(0, -1){20}}
\end{picture}
\end{aligned}
\end{eqnarray}
by using Lemma \ref{lem:Sigma} with $\delta=-2n$. Note that each Brauer diagram summand of the
first term has $n+1-2k$ through strings, while the summands in the
second term have $n-1-2k$ through strings.
Using \eqref{eq:ASigma} one shows by simple calculation that
\[
\sum a_k (I_{n-1}\otimes A_1)\circ \Xi_k=0.
\]
Hence $(I_{n-1}\otimes A_1)\circ\Phi=0$, which implies statement (1).

To prove part (\ref{lem:Phi-4}), we note that the trace of $\frac{F(\Phi)}{(n+1)!}$ is equal to
the dimension of the subspace $F(\Phi)(V^{\otimes(n+1)})$, since
$\frac{F(\Phi)}{(n+1)!}$ is an idempotent by part (\ref{lem:Phi-2}).
In order to evaluate $tr\left(\frac{F(\Phi)}{(n+1)!}\right)$, we first consider
$tr\left(\frac{F(\Xi_k)}{(n+1)!}\right)$,
which is given by
\[
\begin{picture}(145, 60)(0, -30)
\put(0, 0) {$(-1)^{n+1}$}
\put(50, 10){\line(1, 0){70}}
\put(50, -10){\line(1, 0){70}}
\put(50, 10){\line(0, -1){20}}
\put(120, 10){\line(0, -1){20}}
\put(70, -3){$n+1$}

\qbezier(55,10)(60,40)(65, 10)
\qbezier(85,10)(90,40)(95, 10)
\qbezier(105, 10)(135, 40)(140, 0)
\qbezier(115, 10)(130, 30)(130, 0)
\put(70, 15){...}
\put(73, 20){\tiny$k$}

\qbezier(55,-10)(60,-40)(65, -10)
\qbezier(85,-10)(90,-40)(95, -10)
\qbezier(105, -10)(135, -40)(140, 0)
\qbezier(115, -10)(130, -30)(130, 0)
\put(70, -18){...}
\put(73, -25){\tiny$k$}

\put(132, 0){\tiny{...}}

\put(145, 0){$=$}
\end{picture}
\begin{picture}(150, 60)(-45, -30)
\put(-30, 0) {$(-1)^{n+1}\frac{(2n-2k)!}{(n-1)!}$}
\put(50, 10){\line(1, 0){50}}
\put(50, -10){\line(1, 0){50}}
\put(50, 10){\line(0, -1){20}}
\put(100, 10){\line(0, -1){20}}
\put(70, -3){$2k$}

\qbezier(55,10)(60,40)(65, 10)
\qbezier(85,10)(90,40)(95, 10)
\put(70, 15){...}
\put(73, 20){\tiny$k$}

\qbezier(55,-10)(60,-40)(65, -10)
\qbezier(85,-10)(90,-40)(95, -10)
\put(70, -18){...}
\put(73, -25){\tiny$k$}
\put(105, -10){,}
\end{picture}
\]
where the last step uses Lemma \ref{lem:Sigma-1}(2) with $\epsilon=-1$.
Using \eqref{eq:k-1}, one can show that
\[
\begin{picture}(150, 60)(50, -30)
\put(50, 10){\line(1, 0){50}}
\put(50, -10){\line(1, 0){50}}
\put(50, 10){\line(0, -1){20}}
\put(100, 10){\line(0, -1){20}}
\put(70, -3){$2k$}

\qbezier(55,10)(60,40)(65, 10)
\qbezier(85,10)(90,40)(95, 10)
\put(70, 15){...}
\put(73, 20){\tiny$k$}

\qbezier(55,-10)(60,-40)(65, -10)
\qbezier(85,-10)(90,-40)(95, -10)
\put(70, -18){...}
\put(73, -25){\tiny$k$}
\put(105, 0){$=$}
\put(115, 0){$(-1)^k 2^{2k} \frac{n! k!}{(n-k)!}$}
\put(190, -7){.}
\end{picture}
\]
Putting these formulae together, we arrive at
\[
\begin{aligned}
tr\left(\frac{F(\Phi)}{(n+1)!}\right) &= \frac{n!}{(n-1)!}\sum_{k=0}^{\left[\frac{n+1}{2}\right]}
a_k (-1)^k 2^{2 k} \frac{k! (2n-2k)!}{(n-k)!}\\
&=\sum_{k=0}^{\left[\frac{n+1}{2}\right]}
(-1)^k \begin{pmatrix}n \\ k\end{pmatrix}
\begin{pmatrix}2n-2k \\ n-1\end{pmatrix}.
\end{aligned}
\]
There is a binomial coefficient identity stating that the far right hand side is equal to zero.
Hence $F(\Phi)$ is the zero map on $V^{\otimes(n+1)}$.
\end{proof}

The corollary below follows from Lemma \ref{lemma:Ep-1}
and the fact that $\pi \Sigma(n+1)\pi'= \Sigma(n+1)$  for all $\pi, \pi'\in\Sym_{n+1}$.

\begin{corollary}\label{cor:uniqueness-Phi}
The element $\Phi/(n+1)!$ is the central
idempotent in $B_{n+1}^{n+1}(-2n)$ which corresponds to the trivial representation $\rho_1$
of $B_{n+1}^{n+1}(-2n)$, defined by $\rho_1(s_i)=1$ and $\rho_1(e_i)=0$ for all $i$.
It generates a $1$-dimensional two-sided ideal of $B_{n+1}^{n+1}(-2n)$.
\end{corollary}

\begin{remark}
Another formula for $\Phi/(n+1)!$ was given in terms of
Jucys-Murphy elements in \cite{IMO}.
\end{remark}

\subsection{The main theorem}

Recall the natural embedding of the Brauer algebra of degree $s$ in that of degree $t$ for any $t>s$.
\begin{definition}\label{def:ideal}
For each $r>n$, let $\langle \Phi\rangle_r$ be the two-sided ideal in the Brauer algebra
$B_r^r(-2n)$ generated by $\Phi$.
\end{definition}

\begin{remark}
A priori, elements such as $(I_{r-q}\ot A_q\ot I_q)(z\ot X_{q, q})(I_{r-q}\ot U_q\ot I_q)$
are not included in $\langle \Phi\rangle_r$ even if $z\in \langle \Phi\rangle_r$.
\end{remark}

We have the following result.
\begin{lemma}\label{lem:Sigma-Phi}
The element $\Sigma(2n+1)$ belongs to $\langle \Phi\rangle_{2n+1}$.
\end{lemma}
\begin{proof}
Consider $B_r^r(-2n)$ for $r>n$. Let $E_r^r(k)= \prod_{j=1}^k e_{r-2j+1}$, and define
\[
\begin{aligned}
\Upsilon(r)_k&=\Sigma(r) E_r^r(k) \Sigma(r), \quad k\ge 1,\\
\Upsilon(r)_{\ge k}&=\text{linear span of $\langle \Phi\rangle_r\cup\{\Upsilon(r)_i\mid i\ge k\}$}.
\end{aligned}
\]
We first want to show that
\begin{eqnarray}\label{eq:step}
\Sigma(r)\in \Upsilon(r)_{\ge\left[\frac{r+1-n}{2}\right]}.
\end{eqnarray}
From the formula for $\Phi$, we obtain
\[
\begin{aligned}
r!(n+1)! \Sigma(r)& =\Sigma(r)\left(\left(\Phi - \sum_{k=1}^{\left[\frac{n+1}{2}\right]}a_k
\Xi_k\right)\otimes I_{r-n-1}\right)\Sigma(r).
\end{aligned}
\]
Thus $\Sigma(r)\in \Upsilon(r)_{\ge 1}$.

Note that for any $z\in \Upsilon(r-2k)_{\ge 1}$,
$\Sigma(r)(z\otimes I_{2k}) E_r^r(k) \Sigma(r)$ belongs to $\Upsilon(r)_{\ge k+1}$.
We can always re-write $\Upsilon(r)_k$ as
\[
\Upsilon(r)_k=\frac{1}{(r-2k)!}\Sigma(r)(\Sigma(r-2k)\otimes I_{2k}) E_r^r(k) \Sigma(r).
\]
If $r-2k>n$, then $\Sigma(r-2k)\in \Upsilon(r-2k)_{\ge 1}$. This implies that
$\Upsilon(r)_k\in \Upsilon(r)_{\ge k+1}$ if $r-2k>n$. Hence
$
\Upsilon(r)_{\ge 1}=\Upsilon(r)_{\ge 2} = \dots = \Upsilon(r)_{\ge\left[\frac{r+1-n}{2}\right]},
$
 and \eqref{eq:step} is proved.

Now consider $\Sigma(2n+1)$.
It follows from \eqref{eq:step} that $\Sigma(2n+1)^2$ can be expressed as a linear combination of
elements in $\langle \Phi\rangle_{2n+1}$ and also elements of the form
\[
\Sigma(2n+1) E_{2n+1}^{2n+1}(i) \Sigma(2n+1) E_{2n+1}^{2n+1}(j) \Sigma(2n+1),
\quad i, j\ge 1+\left[\frac{n}{2}\right].
\]
Using the symmetrising property of $\Sigma(2n+1)$, we can write this element
as $\Sigma(2n+1)(I_{2n+1-2i}\otimes U_i) \Psi_{i j}(I_{2n+1-2j}\otimes A_j) \Sigma(2n+1)$ with
\[
\Psi_{i j}=
(I_{2n+1-2i}\otimes A_i)\Sigma(2n+1)(I_{2n+1-2j}\otimes U_j).
\]
By Corollary \ref{cor:small-degree}, $\Psi_{i j}=0$ for all  $i, j\ge 1+\left[\frac{n}{2}\right]$. Hence
$\Sigma(2n+1)^2$ belongs to $\langle \Phi\rangle_{2n+1}$, and so does also $\Sigma(2n+1)$.
\end{proof}

The following is one of the main results of this paper.
\begin{theorem}\label{thm:sp-main}
The algebra homomorphism  ${F}_r^r: B_r^r(-2n)\longrightarrow \Hom_{\Sp(V)}(V^{\otimes r}, V^{\otimes r})$
is injective if $r\le n$. If $r\ge n+1$,
then $\Ker{F}_r^r$ is the two-sided ideal of the Brauer algebra $B_r^r(-2n)$ which is generated by
the element $\Phi$ defined by \eqref{Phi}.
\end{theorem}

\begin{proof} Only the second statement requires proof. Thus we assume that $r\ge n+1$.
Consider first the case $r=n+1$. Then there is only one $D(p, q)$ with $p=n$ and $q=0$ (see Figure \ref{Dpq}).
Using $\Sigma(n+1)=\Phi - \sum_{k=1}^{\left[\frac{n+1}{2}\right]}a_k \Xi_k$, we have
\[
D(n, 0)=\frac{D(n, 0)\Phi}{(n+1)!} - \sum_{k=1}^{\left[\frac{n+1}{2}\right]} a_k\frac{D(n, 0)\Xi_k}{(n+1)!}.
\]
Note that

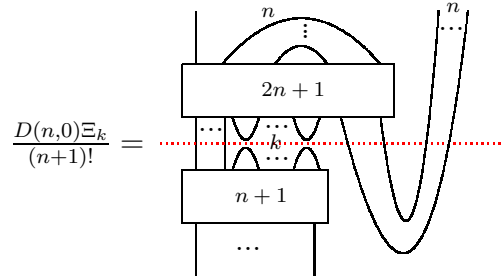
\begin{center}
\begin{picture}(100, 80)(0,-30)
\put(-65, 7){$\frac{D(n, 0)\Xi_k}{(n+1)!}=$}
\put(5, 40){\line(0, 1){20}}

\put(30, 58){\tiny$n$}
\qbezier(15, 40)(45, 75)(75, 40)
\qbezier(35, 40)(45, 54)(55, 40)
\put(45, 54){.}
\put(45, 52){.}
\put(45, 50){.}

\put(0, 20){\line(1, 0){80}}
\put(0, 20){\line(0, 1){20}}
\put(80, 20){\line(0, 1){20}}
\put(0, 40){\line(1, 0){80}}
\put(30, 28){\tiny${2n+1}$}

\put(5, 20){\line(0, -1){20}}
\put(6, 15){...}
\put(16, 20){\line(0, -1){20}}

{\color{red}\qbezier[60](-8, 10)(56, 10)(120, 10)}

\qbezier(19, 20)(24, 3)(29, 20)
\put(31, 16){...}
\qbezier(42, 20)(47, 3)(52, 20)
\put(33, 8){\tiny$k$}
\qbezier(19, 0)(24, 17)(29, 0)
\put(31, 4){...}
\qbezier(42, 0)(47, 17)(52, 0)

\qbezier(60, 20)(90, -100)(108, 60)
\qbezier(75, 20)(87, -75)(97, 60)
\put(97, 53){...}
\put(100, 60){\tiny$n$}

\put(0, 0){\line(1, 0){55}}
\put(0, 0){\line(0, -1){20}}
\put(55, 0){\line(0, -1){20}}
\put(0, -20){\line(1, 0){55}}
\put(20, -12){\tiny${n+1}$}
\put(5, -20){\line(0, -1){20}}
\put(20, -30){...}
\put(50, -20){\line(0, -1){20}}
\end{picture}
\end{center}
where the dotted-line indicates that the diagram is the composition of
the two diagrams above and below the line. The diagram above the dotted line is the tensor product
of an element in $\langle \Sigma(2n+1)\rangle_{n+1-2k}^1$ with $I_n$.  Since
$\langle \Sigma(2n+1)\rangle_{n+1-2k}^1=0$ for all $k\ge 1$
by Corollary \ref{cor:small-degree}, we have $\frac{D(n, 0)\Xi_k}{(n+1)!}=0$.
This proves $D(n, 0)\in \langle \Phi\rangle_{n+1}$.

Now we use induction on $r$ to show that the theorem holds for $r>n+1$.
If $p=0$, the diagram corresponds to $\Sigma(2n+1)$, which
belongs to $\langle \Phi\rangle_{2n+1}$ by Lemma \ref{lem:Sigma-Phi}.
Assume $n\ge p\ge 1$, and let $r=2n+1-p+q$. Consider $D(p, q)\circ\Sigma(2n+1-p)$
by using the the formula
\[
\Sigma(2n+1-p)=\left(\left(\Phi - \sum_{k=1}^{\left[\frac{n+1}{2}\right]}a_k \Xi_k\right)
\otimes I_{n-p}\right)\frac{\Sigma(2n+1-p)}{(n+1)!}.
\]
We obtain an expression for $D(p, q)$ of the form
\begin{eqnarray}\label{eq:D-Phi}
 D(p, q) = \sum_{k\ge 1} c_k D(p, q; k) + D^0,
\end{eqnarray}
where  $c_k$ are scalars, $D^0\in \langle \Phi\rangle_r$, and

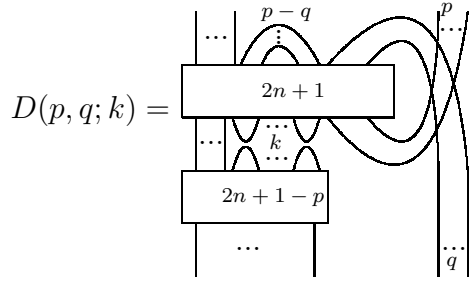
\begin{center}
\begin{picture}(100, 80)(0,0)
\put(-65, 20){$D(p, q; k) =$}
\put(5, 40){\line(0, 1){20}}
\put(8, 50){...}
\put(20, 40){\line(0, 1){20}}

\put(30, 58){\tiny${p-q}$}
\qbezier(22, 40)(36, 70)(52, 40)
\put(35, 52){.}
\put(35, 50){.}
\put(35, 48){.}
\qbezier(30, 40)(36, 54)(44, 40)

\qbezier(55, 40)(100, 90)(108, 0)
\qbezier(70, 40)(95, 70)(97, 0)

\put(0, 20){\line(1, 0){80}}
\put(0, 20){\line(0, 1){20}}
\put(80, 20){\line(0, 1){20}}
\put(0, 40){\line(1, 0){80}}
\put(30, 28){\tiny${2n+1}$}

\put(5, 20){\line(0, -1){20}}
\put(6, 10){...}
\put(16, 20){\line(0, -1){20}}

\qbezier(19, 20)(24, 2)(29, 20)
\put(31, 16){...}
\qbezier(42, 20)(47, 2)(52, 20)
\put(33, 8){\tiny$k$}
\qbezier(19, 0)(24, 18)(29, 0)
\put(31, 4){...}
\qbezier(42, 0)(47, 18)(52, 0)

\qbezier(55, 20)(95, -30)(108, 60)
\qbezier(70, 20)(95, -10)(97, 60)
\put(97, 53){...}
\put(98, 60){\tiny$p$}

\put(0, 0){\line(1, 0){55}}
\put(0, 0){\line(0, -1){20}}
\put(55, 0){\line(0, -1){20}}
\put(0, -20){\line(1, 0){55}}
\put(15, -12){\tiny${2n+1-p}$}
\put(5, -20){\line(0, -1){20}}
\put(20, -30){...}
\put(50, -20){\line(0, -1){20}}

\put(97, 0){\line(0, -1){40}}
\put(108, 0){\line(0, -1){40}}

\put(98, -30){...}
\put(100, -36){\tiny$q$}
\end{picture}
\end{center}

\vspace{2cm}

The diagram $D(p, q; k)$ is the composition of

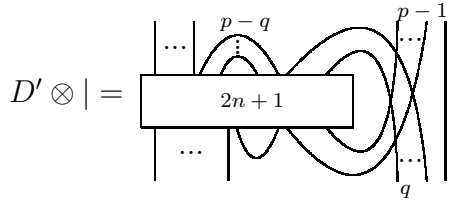
\begin{center}

\begin{picture}(100, 80)(0,0)
\put(-50, 30){$D'\otimes |=$}
\put(5, 40){\line(0, 1){20}}
\put(8, 50){...}
\put(20, 40){\line(0, 1){20}}

\put(30, 58){\tiny${p-q}$}
\qbezier(22, 40)(36, 70)(52, 40)
\put(35, 52){.}
\put(35, 50){.}
\put(35, 48){.}
\qbezier(30, 40)(36, 54)(44, 40)

\qbezier(55, 40)(100, 90)(108, 0)
\qbezier(70, 40)(95, 70)(97, 0)
\put(97, 7){...}
\put(98, -5){\tiny$q$}

\put(0, 20){\line(1, 0){80}}
\put(0, 20){\line(0, 1){20}}
\put(80, 20){\line(0, 1){20}}
\put(0, 40){\line(1, 0){80}}
\put(30, 28){\tiny${2n+1}$}

\put(5, 20){\line(0, -1){20}}
\put(14, 10){...}
\put(33, 20){\line(0, -1){20}}

\qbezier(36, 20)(43, -3)(52, 20)

\qbezier(55, 20)(95, -30)(108, 60)
\qbezier(70, 20)(95, -10)(97, 60)
\put(97, 53){...}
\put(97, 62){\tiny${p-1}$}
\put(115, 60){\line(0, -1){60}}
\end{picture}
\end{center}

\bigskip

\noindent
with the following element of $B_r(-2n)$

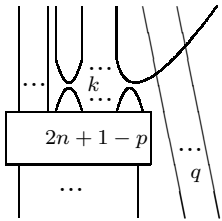
\begin{center}
\begin{picture}(100, 80)(-5, -35)
\qbezier(19, 20)(24, 2)(29, 20)
\put(31, 16){...}
\put(31, 8){\tiny$k$}
\qbezier(19, 0)(24, 18)(29, 0)
\put(31, 4){...}
\qbezier(42, 0)(47, 18)(52, 0)

\put(5, 20){\line(0, -1){20}}
\put(6, 10){...}
\put(16, 20){\line(0, -1){20}}

\put(5, 20){\line(0, 1){20}}
\put(16, 20){\line(0, 1){20}}

\put(19, 20){\line(0, 1){20}}
\put(29, 20){\line(0, 1){20}}
\put(42, 20){\line(0, 1){20}}

\qbezier(42, 20)(47, -5)(80, 40)

\put(0, 0){\line(1, 0){55}}
\put(0, 0){\line(0, -1){20}}
\put(55, 0){\line(0, -1){20}}
\put(0, -20){\line(1, 0){55}}
\put(15, -12){\tiny${2n+1-p}$}
\put(5, -20){\line(0, -1){20}}
\put(20, -30){...}
\put(50, -20){\line(0, -1){20}}

\put(65, -15){...}
\put(70, -25){\tiny$q$}

\put(52, 40){\line(1, -5){16}}
\put(65, 40){\line(1, -5){16}}
\end{picture}
\end{center}

\noindent
Note that $D'$ belongs to $\ker{{F}_{r-1}}$.
Thus $D'\in \langle \Phi\rangle_{r-1}$ by the
induction hypothesis and it follows that
$D(p, q; k)\in \langle \Phi\rangle_r$.
This completes the proof.
\end{proof}

\begin{remark}\label{rem:uniqueness-Phi}
Any element which generates the kernel $\Ker{F}_{n+1}^{n+1}=\langle \Phi\rangle_{n+1}$
must be a nontrivial scalar multiple of $\Phi$. It was proved in \cite{HX} that for all $r\ge n+1$,
$\Ker{F}_r^r$ is generated by a single generator belonging to $\Ker{F}_{n+1}^{n+1}$.
Therefore, our $\Phi$ provides an explicit formula for this generator (up to a scalar multiple).
\end{remark}

\section{Structure of the endomorphism algebra: the orthogonal case}
We now study the algebraic structure of $\Ker{F_r^r}$ in the case of the orthogonal group.
Throughout this section, we take $G=\Or(V)$ with $\dim V = m$ and $\epsilon=1$.

\subsection{Generators of the kernel}

For $p=0, 1, \dots, m+1$, let $E_{m+1-p}$ denote the element of the
Brauer algebra $B_{m+1}^{m+1}(m)$ of degree $m+1$ shown in Figure \ref{Ep}.
\begin{figure}[h]
\begin{center}
\begin{picture}(100, 60)(0,0)
\put(-35, 27){$E_{m+1-p}=$}

\put(25, 40){\line(0, 1){20}}
\put(33, 50){...}
\put(50, 40){\line(0, 1){20}}

\qbezier(58, 40)(100, 90)(108, 0)
\qbezier(72, 40)(95, 70)(97, 0)
\put(97, 7){...}
\put(98, -5){\tiny$p$}

\put(20, 20){\line(1, 0){60}}
\put(20, 20){\line(0, 1){20}}
\put(80, 20){\line(0, 1){20}}
\put(20, 40){\line(1, 0){60}}
\put(40, 28){\tiny${m+1}$}

\put(25, 20){\line(0, -1){20}}
\put(33, 10){...}
\put(50, 20){\line(0, -1){20}}

\qbezier(58, 20)(95, -30)(108, 60)
\qbezier(72, 20)(95, -10)(97, 60)
\put(97, 53){...}
\put(98, 60){\tiny$p$}
\end{picture}
\end{center}
\caption{}
\label{Ep}
\end{figure}

\begin{lemma}\label{lem:E-Z}
For all $0\le k\le m+1$, the elements $E_k$ are linear combinations of
Brauer diagrams over $\Z$.
\end{lemma}
This is evident from the definition of these elements.
They also have the following properties.
\begin{lemma}\label{lemma:Ep-1}
\begin{enumerate}
\item $*E_p = E_{m+1-p}$ for all $p$.
\item $F_p E_p = E_p F_p=p!(m+1-p)! E_p$.
\item $e_i E_p = E_p e_i=0$ for all $i\le m$.
\end{enumerate}
\end{lemma}
\begin{proof} Both (1) and (2) follow easily from the pictorial representation of
$E_p$ given in Figure \ref{Ep}. If $i\ne p$, then $e_i F_p=F_p e_i =0$. Thus (3)
holds for all $i\ne p$. The $i=p$ case of (3) follows from the fact that
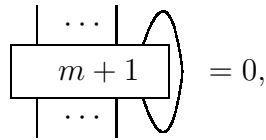
\[
\begin{picture}(80, 60)(0, -30)
\put(0, 10){\line(1, 0){60}}
\put(0, -10){\line(1, 0){60}}
\put(0, 10){\line(0, -1){20}}
\put(60, 10){\line(0, -1){20}}
\put(18, -3){$m+1$}

\put(10, 10){\line(0, 1){15}}
\put(20, 15){$\cdots$}
\put(40, 10){\line(0, 1){15}}

\qbezier(50, -10)(60, -40)(65, 0)
\qbezier(50, 10)(60, 40)(65, 0)

\put(10, -10){\line(0, -1){15}}
\put(20, -20){$\cdots$}
\put(40, -10){\line(0, -1){15}}

\put(75, -3){$=0$,}
\end{picture}
\]
which is implied by Lemma \ref{lem:Sigma-1}(2) when $r=m+1$ and $\epsilon=1$.
\end{proof}

The arguments used in the proof of \cite[Corollary 5.13]{LZ4} lead to
\begin{corollary}[\cite{LZ4}]\label{cor:ann}
Let $D$ be any diagram in $B_{m+1}^{m+1}(n)$ which has fewer
than $m+1$ through strings. Then $D E_i = E_i D = 0$ for all $i$.
\end{corollary}

Note that $E_0=E_{m+1}=\Sigma_{+1}(m+1)$.
\begin{proposition}\label{lem:o-generat}
Assume $r>m$. As a two-sided ideal of the Brauer algebra $B_r^r(m)$, $\Ker{F}_r^r$
is generated by $E_p$ for all $0\le p\le m+1$.
\end{proposition}
\begin{proof}
The proof of Proposition \ref{lem:sp-generat} can easily be modified
to prove the assertion above.  The two required modifications are that for
any $(2r, 0)$ tangle diagram $A$ with
associated invariant functional $\gamma=F(A)$,
(i) the definition \eqref{eq:AS} of $A_S$ needs to be changed to
\[
A_S= \sum_{\pi\in\Sym_S} (-1)^{|\pi|}A\circ\pi;
\]
(ii) we only need to consider subsets $S$ of $[1, 2r]$ which will not lead to
the trivial vanishing of $A_S$ discussed in Remark \ref{rem:pertinent} (i).
With these modifications, the arguments following \eqref{eq:AS}
may be repeated verbatim, leading to the conclusion that
$\Ker{F}_r^r$ is generated as a two-sided ideal of
$B_r^r(-2n)$  by elements of the form Figure \ref{Ep-1}.
\begin{figure}[h]
\begin{center}
\begin{picture}(100, 60)(-20,0)

\put(20, 20){\line(1, 0){60}}
\put(20, 20){\line(0, 1){20}}
\put(80, 20){\line(0, 1){20}}
\put(20, 40){\line(1, 0){60}}
\put(42, 28){\tiny${m+1}$}

\qbezier(30, 40)(10, 70)(10, 0)
\qbezier(45, 40)(5, 90)(-5,0)
\put(0, 10){...}
\put(0, 0){\tiny$p$}

\put(50, 40){\line(0, 1){20}}
\put(70, 40){\line(0, 1){20}}
\put(55, 50){...}

\put(30, 20){\line(0, -1){20}}
\put(35, 10){...}
\put(50, 20){\line(0, -1){20}}

\qbezier(55, 20)(95, -30)(105, 60)
\qbezier(70, 20)(90, -10)(90, 60)
\put(92, 53){...}
\put(95, 60){\tiny$p$}
\end{picture}
\end{center}
\caption{}
\label{Ep-1}
\end{figure}

Post-multiplying the diagram in Figure \ref{Ep-1} by the invertible element $X_{m+1-p, p}$ , we obtain
Figure \ref{Ep} up to a sign. This completes the proof.
\end{proof}

\begin{remark}
Figure \ref{Ep-1} is the $p=q$ analogue of Figure \ref{Dpq-1}.
In the present case,  diagrams of the form Figure \ref{Dpq-1} with $p>q$ vanish identically,
since $\Sigma_{+1}(m+1)$ is the total antisymmetriser in $\Sym_{m+1}$.
\end{remark}

\subsection{Formulae for the $E_i$}

If $k,l$ are integers such that $1\leq k<l$, write
$A(k,l):=\Sigma_{+1}(\Sym_{\{k,k+1,\dots,l\}})$ for the total antisymmetriser
in $\Sym_{\{k,k+1,\dots,l\}}$.
By convention,  $A(k,l)=1$ if $k\geq l$.
Represent $A(1, t)A(t+1, t+s)$ in $B_{t+s}^{t+s}(m)$ pictorially by
\[
\begin{picture}(120, 80)(0, -40)
\put(5, 10){\line(1, 0){45}}
\put(5, -10){\line(1, 0){45}}
\put(5, 10){\line(0, -1){20}}
\put(50, 10){\line(0, -1){20}}
\put(25, -3){$t$}

\put(10, 10){\line(0, 1){25}}
\put(45, 10){\line(0, 1){25}}
\put(23, 25){...}

\put(10, -10){\line(0, -1){25}}
\put(45, -10){\line(0, -1){25}}
\put(23, -30){...}

\put(60, 10){\line(1, 0){45}}
\put(60, -10){\line(1, 0){45}}
\put(60, 10){\line(0, -1){20}}
\put(105, 10){\line(0, -1){20}}
\put(80, -3){$s$}

\put(100, -10){\line(0, -1){25}}
\put(80, -30){...}
\put(65, -10){\line(0, -1){25}}
\put(100, 10){\line(0, 1){25}}
\put(65, 10){\line(0, 1){25}}
\put(80, 25){...}
\put(110, -5){.}
\end{picture}
\]
The lemma below is the graphical reformulation of
some of the computations in the proofs of \cite[Corollary 5.2]{LZ4} and
\cite[Theorem 5.10]{LZ4}.

\begin{lemma}\label{lem:AfU}
For all $k=0, 1, \dots, i$
\begin{equation}\label{eq:AfU}
\begin{aligned}
\begin{picture}(120, 80)(5, -40)
\put(5, 10){\line(1, 0){45}}
\put(5, -10){\line(1, 0){45}}
\put(5, 10){\line(0, -1){20}}
\put(50, 10){\line(0, -1){20}}
\put(23, -3){\tiny$i$}

\qbezier(45, 10)(55, 30) (65, 10)
\put(10, 10){\line(0, 1){25}}
\put(35, 10){\line(0, 1){25}}
\put(18, 25){...}

\put(10, -10){\line(0, -1){25}}
\put(25, -10){\line(0, -1){25}}
\put(13, -30){...}

\qbezier(30, -10)(55, -50) (80, -10)
\qbezier(45, -10)(55, -30) (65, -10)
\put(55, -29){.}
\put(55, -23){.}
\put(55, -26){.}
\put(55, -40){\tiny$k$}

\put(60, 10){\line(1, 0){45}}
\put(60, -10){\line(1, 0){45}}
\put(60, 10){\line(0, -1){20}}
\put(105, 10){\line(0, -1){20}}
\put(67, -3){\tiny${m+1-i}$}

\put(100, -10){\line(0, -1){25}}
\put(88, -30){...}
\put(85, -10){\line(0, -1){25}}
\put(100, 10){\line(0, 1){25}}
\put(75, 10){\line(0, 1){25}}
\put(83, 25){...}
\put(110, -5){$=$}
\end{picture}
\begin{picture}(120, 80)(-5, -40)
\put(-6, -5){$k^2$}
\put(5, 10){\line(1, 0){45}}
\put(5, -10){\line(1, 0){45}}
\put(5, 10){\line(0, -1){20}}
\put(50, 10){\line(0, -1){20}}
\put(20, -3){\tiny${i-1}$}

\put(10, 10){\line(0, 1){25}}
\put(45, 10){\line(0, 1){25}}
\put(23, 25){...}

\put(10, -10){\line(0, -1){25}}
\put(25, -10){\line(0, -1){25}}
\put(13, -30){...}

\qbezier(30, -10)(55, -50) (80, -10)
\qbezier(45, -10)(55, -30) (65, -10)
\put(55, -29){.}
\put(55, -23){.}
\put(55, -26){.}
\put(53, -40){\tiny${k-1}$}

\put(60, 10){\line(1, 0){45}}
\put(60, -10){\line(1, 0){45}}
\put(60, 10){\line(0, -1){20}}
\put(105, 10){\line(0, -1){20}}
\put(72, -3){\tiny${m-i}$}

\put(100, -10){\line(0, -1){25}}
\put(88, -30){...}
\put(85, -10){\line(0, -1){25}}
\put(100, 10){\line(0, 1){25}}
\put(65, 10){\line(0, 1){25}}
\put(80, 25){...}
\put(110, -5){$+$}
\end{picture}
\begin{picture}(120, 80)(-20, -30)
\put(-15, 5){$\zeta_{i, k}$}
\put(0, -20){\line(1, 0){40}}  \put(55, -20){\line(1, 0){40}}
\put(0, -5){\line(1, 0){40}}   \put(55, -5){\line(1, 0){40}}
\put(0, 20){\line(1, 0){45}}   \put(50, 20){\line(1, 0){45}}
\put(0, 35){\line(1, 0){45}}   \put(50, 35){\line(1, 0){45}}
\put(45, 20){\line(0, 1){15}}  \put(95, 20){\line(0, 1){15}}
\put(0, 20){\line(0, 1){15}}   \put(50, 20){\line(0, 1){15}}
\put(0, -20){\line(0, 1){15}}  \put(55, -20){\line(0, 1){15}}
\put(40, -20){\line(0, 1){15}} \put(95, -20){\line(0, 1){15}}

\put(15, 25){\tiny${i-1}$}    \put(70, 25){\tiny${m-i}$}
\put(10, -15){\tiny${i-k}$} \put(70, -15){\tiny$j$}

\put(5, 35){\line(0, 1){10}}    \put(55, 35){\line(0, 1){10}}
\put(18, 40){...}               \put(68, 40){...}
\put(40, 35){\line(0, 1){10}}   \put(90, 35){\line(0, 1){10}}

\put(5, 20){\line(0, -1){25}}   \put(20, 20){\line(0, -1){25}}
\put(75, 20){\line(0, -1){25}}  \put(90, 20){\line(0, -1){25}}
\put(8, 5){...}                 \put(78, 5){...}

\put(5, -20){\line(0, -1){10}}  \put(90, -20){\line(0, -1){10}}
\put(35, -20){\line(0, -1){10}} \put(60, -20){\line(0, -1){10}}
\put(18, -25){...}                \put(67, -25){...}

\qbezier(25, 20)(48, -10) (70, 20)
\qbezier(40, 20)(48, 10) (55, 20)
\put(48, 12){.}
\put(48, 9){.}
\put(48, 6){.}
\put(52, 9){\tiny$k$}
\qbezier(25, -5)(48, 5) (70,-5)

\put(105, -5){,}
\end{picture}
\end{aligned}
\end{equation}
where $j={m+1-i-k}$ and $\zeta_{i, k}= \frac{1}{(i-k-1)!(m-i-k)!}$.
\end{lemma}
\begin{proof} When $k=0$, \eqref{eq:AfU} is an identity.

We use Lemma \ref{lem:Sigma-1}(1) twice to obtain
\begin{eqnarray}\label{eq:AfU-1}
\begin{aligned}
\begin{picture}(130, 80)(-10, -40)
\put(5, 10){\line(1, 0){45}}
\put(5, -10){\line(1, 0){45}}
\put(5, 10){\line(0, -1){20}}
\put(50, 10){\line(0, -1){20}}
\put(25, -3){$t$}

\put(10, 10){\line(0, 1){25}}
\put(35, 10){\line(0, 1){25}}
\put(18, 25){...}

\put(10, -10){\line(0, -1){25}}
\put(35, -10){\line(0, -1){25}}
\put(18, -30){...}

\put(100, -10){\line(0, -1){25}}
\put(83, -30){...}
\put(75, -10){\line(0, -1){25}}
\put(100, 10){\line(0, 1){25}}
\put(75, 10){\line(0, 1){25}}
\put(83, 25){...}

\put(60, 10){\line(1, 0){45}}
\put(60, -10){\line(1, 0){45}}
\put(60, 10){\line(0, -1){20}}
\put(105, 10){\line(0, -1){20}}
\put(80, -3){$s$}

\qbezier(45, 10)(55, 45)(65, 10)
\qbezier(45, -10)(55, -45)(65, -10)
\put(110, -5){$=$}
\end{picture}
\begin{picture}(130, 80)(-15, -40)
\put(-15, -5){$\psi_{t, s}$}
\put(5, 10){\line(1, 0){45}}
\put(5, -10){\line(1, 0){45}}
\put(5, 10){\line(0, -1){20}}
\put(50, 10){\line(0, -1){20}}
\put(20, -3){$t-1$}

\put(10, 10){\line(0, 1){25}}
\put(45, 10){\line(0, 1){25}}
\put(23, 25){...}

\put(10, -10){\line(0, -1){25}}
\put(45, -10){\line(0, -1){25}}
\put(23, -30){...}

\put(60, 10){\line(1, 0){45}}
\put(60, -10){\line(1, 0){45}}
\put(60, 10){\line(0, -1){20}}
\put(105, 10){\line(0, -1){20}}
\put(72, -3){$s-1$}

\put(100, -10){\line(0, -1){25}}
\put(80, -30){...}
\put(65, -10){\line(0, -1){25}}
\put(100, 10){\line(0, 1){25}}
\put(65, 10){\line(0, 1){25}}
\put(80, 25){...}
\put(110, -5){$+$}
\end{picture}
\begin{picture}(120, 80)(-25, -30)
\put(-15, 5){$\phi_{t, s}$}
\put(0, -20){\line(1, 0){45}}  \put(50, -20){\line(1, 0){45}}
\put(0, -5){\line(1, 0){45}}   \put(50, -5){\line(1, 0){45}}
\put(0, 20){\line(1, 0){45}}   \put(50, 20){\line(1, 0){45}}
\put(0, 35){\line(1, 0){45}}   \put(50, 35){\line(1, 0){45}}
\put(45, 20){\line(0, 1){15}}  \put(95, 20){\line(0, 1){15}}
\put(0, 20){\line(0, 1){15}}   \put(50, 20){\line(0, 1){15}}
\put(0, -20){\line(0, 1){15}}  \put(50, -20){\line(0, 1){15}}
\put(45, -20){\line(0, 1){15}} \put(95, -20){\line(0, 1){15}}

\put(10, 25){$t-1$}       \put(60, 25){$s-1$}
\put(10, -15){$t-1$}      \put(60, -15){$s-1$}

\put(5, 35){\line(0, 1){10}}    \put(55, 35){\line(0, 1){10}}
\put(18, 40){...}               \put(68, 40){...}
\put(40, 35){\line(0, 1){10}}   \put(90, 35){\line(0, 1){10}}

\put(5, -20){\line(0, -1){10}}  \put(90, -20){\line(0, -1){10}}

\put(5, 20){\line(0, -1){25}}   \put(20, 20){\line(0, -1){25}}
\put(75, 20){\line(0, -1){25}}  \put(90, 20){\line(0, -1){25}}
\put(8, 5){...}                 \put(78, 5){...}

\put(40, -20){\line(0, -1){10}} \put(55, -20){\line(0, -1){10}}
\put(18, -28){...}                \put(67, -28){...}

\qbezier(30, 20)(48, -5) (65, 20)
\qbezier(30, -5)(48, 10) (65,-5)

\put(105, -5){,}
\end{picture}
\end{aligned}
\end{eqnarray}
where
\[
\psi_{t, s}=m+2-t-s, \quad \phi_{t, s}=\frac{1}{(t-2)!(s-2)!}.
\]
The case $k=1$ of \eqref{eq:AfU} can be obtained by setting $t=i$ and $s=m+1-i$.

Now use induction on $k$. Post-composing $I_{i-k-1}\otimes U\otimes I_{m-i-k}$
to \eqref{eq:AfU} we obtain
\begin{equation*}
\begin{aligned}
\begin{picture}(120, 80)(5, -40)
\put(5, 10){\line(1, 0){45}}
\put(5, -10){\line(1, 0){45}}
\put(5, 10){\line(0, -1){20}}
\put(50, 10){\line(0, -1){20}}
\put(23, -3){\tiny$i$}

\qbezier(45, 10)(55, 30) (65, 10)
\put(10, 10){\line(0, 1){25}}
\put(35, 10){\line(0, 1){25}}
\put(18, 25){...}

\put(10, -10){\line(0, -1){25}}
\put(25, -10){\line(0, -1){25}}
\put(13, -30){...}

\qbezier(30, -10)(55, -50) (80, -10)
\qbezier(45, -10)(55, -30) (65, -10)
\put(55, -29){.}
\put(55, -23){.}
\put(55, -26){.}
\put(50, -40){\tiny${k+1}$}

\put(60, 10){\line(1, 0){45}}
\put(60, -10){\line(1, 0){45}}
\put(60, 10){\line(0, -1){20}}
\put(105, 10){\line(0, -1){20}}
\put(67, -3){\tiny${m+1-i}$}

\put(100, -10){\line(0, -1){25}}
\put(88, -30){...}
\put(85, -10){\line(0, -1){25}}
\put(100, 10){\line(0, 1){25}}
\put(75, 10){\line(0, 1){25}}
\put(83, 25){...}
\put(110, -5){$=$}
\end{picture}
\begin{picture}(120, 80)(-5, -40)
\put(-6, -5){$k^2$}
\put(5, 10){\line(1, 0){45}}
\put(5, -10){\line(1, 0){45}}
\put(5, 10){\line(0, -1){20}}
\put(50, 10){\line(0, -1){20}}
\put(20, -3){\tiny${i-1}$}

\put(10, 10){\line(0, 1){25}}
\put(45, 10){\line(0, 1){25}}
\put(23, 25){...}

\put(10, -10){\line(0, -1){25}}
\put(25, -10){\line(0, -1){25}}
\put(13, -30){...}

\qbezier(30, -10)(55, -50) (80, -10)
\qbezier(45, -10)(55, -30) (65, -10)
\put(55, -29){.}
\put(55, -23){.}
\put(55, -26){.}
\put(55, -40){\tiny$k$}

\put(60, 10){\line(1, 0){45}}
\put(60, -10){\line(1, 0){45}}
\put(60, 10){\line(0, -1){20}}
\put(105, 10){\line(0, -1){20}}
\put(72, -3){\tiny${m-i}$}

\put(100, -10){\line(0, -1){25}}
\put(88, -30){...}
\put(85, -10){\line(0, -1){25}}
\put(100, 10){\line(0, 1){25}}
\put(65, 10){\line(0, 1){25}}
\put(80, 25){...}
\put(110, -5){$+$}
\end{picture}
\begin{picture}(120, 80)(-20, -30)
\put(-15, 5){$\zeta_{i, k}$}
\put(0, -20){\line(1, 0){40}}  \put(55, -20){\line(1, 0){40}}
\put(0, -5){\line(1, 0){40}}   \put(55, -5){\line(1, 0){40}}
\put(0, 20){\line(1, 0){45}}   \put(50, 20){\line(1, 0){45}}
\put(0, 35){\line(1, 0){45}}   \put(50, 35){\line(1, 0){45}}
\put(45, 20){\line(0, 1){15}}  \put(95, 20){\line(0, 1){15}}
\put(0, 20){\line(0, 1){15}}   \put(50, 20){\line(0, 1){15}}
\put(0, -20){\line(0, 1){15}}  \put(55, -20){\line(0, 1){15}}
\put(40, -20){\line(0, 1){15}} \put(95, -20){\line(0, 1){15}}

\put(15, 25){\tiny${i-1}$}    \put(70, 25){\tiny${m-i}$}
\put(10, -15){\tiny${i-k}$} \put(70, -15){\tiny$j$}

\put(5, 35){\line(0, 1){10}}    \put(55, 35){\line(0, 1){10}}
\put(18, 40){...}               \put(68, 40){...}
\put(40, 35){\line(0, 1){10}}   \put(90, 35){\line(0, 1){10}}

\put(5, 20){\line(0, -1){25}}   \put(20, 20){\line(0, -1){25}}
\put(75, 20){\line(0, -1){25}}  \put(90, 20){\line(0, -1){25}}
\put(8, 5){...}                 \put(78, 5){...}

\put(5, -20){\line(0, -1){10}}  \put(90, -20){\line(0, -1){10}}
\put(20, -20){\line(0, -1){10}} \put(75, -20){\line(0, -1){10}}
\put(10, -25){...}                \put(78, -25){...}

\qbezier(25, 20)(48, -10) (70, 20)
\qbezier(40, 20)(48, 10) (55, 20)
\put(48, 12){.}
\put(48, 9){.}
\put(48, 6){.}
\put(52, 9){\tiny$k$}
\qbezier(25, -5)(48, 5) (70,-5)
\qbezier(25, -20)(48, -35) (70,-20)

\put(105, -5){.}
\end{picture}
\end{aligned}
\end{equation*}
By using \eqref{eq:AfU-1} in the bottom half of the second diagram on the right hand side, we
obtain \eqref{eq:AfU} for $k+1$, completing the proof.
\end{proof}

Following \cite[\S 4.2]{LZ4}, we introduce the elements of $B_{m+1}^{m+1}(m)$ below.
For $p=0,1,\dots,m+1$, let
\[
F_p:=A(1,p)A(p+1,m+1),
\]
where $F_0$ is interpreted as $A(1,m+1)$.
For $j=0,1,2,\dots, i$, define $e_i(j)=e_{i,i+1}e_{i-1,i+2}\dots e_{i-j+1,i+j}$.
Note that $e_i(0)=1$ by convention.
We have the following formulae for the $E_i$.
\begin{lemma}\label{lem:Ep}
For $i=0, 1, \dots, m+1$, let $min_i=min(i, m+1-i)$. Then
\begin{eqnarray}\label{eq:E-formula}
E_i=\sum_{j=0}^{min_i}(-1)^j c_i(j) \Xi_i(j) \quad\text{with}\quad \Xi_i(j)=F_i e_i(j) F_i,
\end{eqnarray}
where $c_i(j)=\left((i-j)!(m+1-i-j)!(j!)^2 \right)^{-1}$.
\end{lemma}

\begin{remark}
For $0\le i\le \left[\frac{m+1}{2}\right]$, the lemma states that the $E_i$ are the elements defined
in \cite[Definition 4.2]{LZ4} with the same notation.
\end{remark}

\begin{proof}
We have $*\Xi_i(j)=\Xi_{m+1-i}(j)$. For $i\le\left[\frac{m}{2}\right]$,
\[
*\left(\sum_{j=0}^{min_i}(-1)^j c_i(j)\Xi_i(j)\right) = \sum_{j=0}^{min_i}(-1)^j c_{m+1-i}(j)\Xi_{m+1-i}(j),
\]
since $c_i(j)=c_{m+1-i}(j)$. Therefore, equation \eqref{eq:E-formula} will hold for
all $i$ by Lemma \ref{lemma:Ep-1}(1), if we can show that it
holds for $0\le i\le \left[\frac{m}{2}\right]$.
This will be done in two steps.

(i).
We first show that for each $i\le \left[\frac{m}{2}\right]$, there exist scalars $x_i(j)$ such that
\[
\begin{aligned}
E_i=\sum_{j=0}^i x_i(j)\Xi_i(j).
\end{aligned}
\]

The case $i=0$ is obvious as we have $E_0=A(1, m+1)$. Thus we only need to consider the case with $i\ge 1$.

Let us label the vertices of $E_i$ (see Figure \ref{Ep}) in the bottom row by $1, 2, \dots, m+1$
from left to right,
and those in the top row by $1', 2', \dots, (m+1)'$ from left to right.
Let $L=\{1, 2, \dots, i\}$, $R=\{i+1, i+2, \dots, m+1\}$,
$L'=\{1', 2', \dots, i'\}$ and $R'=\{(i+1)', (i+2)', \dots, (m+1)'\}$.
Since $A(1, m+1)$ has through strings only,
a Brauer diagram in $E_i$ can only have the following types of edges
(an edge is represented by its pair of vertices)
\[
\begin{aligned}
(a, t)  \in L\times R, \quad
(a', t') \in L'\times R',\\
(a', b) \in L'\times L, \quad
(s', t)\in R'\times R,
\end{aligned}
\]
and the numbers of edges in $L\times R$ and in $L'\times R'$
must be equal.
Thus it follows Lemma \ref{lemma:Ep-1}(2) and
the antisymmetrising property of $A(1, i)$ and $A(i+1, m+1)$ that
$E_i$ is a linear combination of $\Xi_i(j)$.

(ii). To determine the scalar $x_i(0)$, we observe that
the terms in $A(1, m+1)$ which do not contain $s_i$ make up
$F_i=A(1, i)A(i+1, m+1)$. Note that
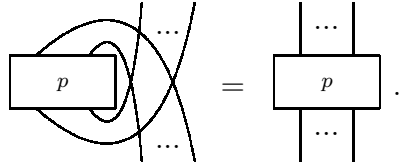
\[
\begin{picture}(80, 60)(0,0)

\put(0, 20){\line(1, 0){40}}
\put(0, 40){\line(1, 0){40}}
\put(0, 20){\line(0, 1){20}}
\put(40, 20){\line(0, 1){20}}
\put(18, 28){\tiny$p$}

\qbezier(10, 40)(55, 80)(70, 0)
\qbezier(30, 40)(45, 60)(50, 0)
\put(55, 5){...}

\qbezier(10, 20)(55, -20)(70, 60)
\qbezier(30, 20)(45, -0)(50, 60)
\put(55, 48){...}

\put(80, 25){$=$}
\end{picture}
\begin{picture}(50, 60)(-20,0)

\put(0, 20){\line(1, 0){40}}
\put(0, 40){\line(1, 0){40}}
\put(0, 20){\line(0, 1){20}}
\put(40, 20){\line(0, 1){20}}
\put(18, 28){\tiny$p$}

\put(10, 40){\line(0, 1){20}}
\put(15, 50){...}
\put(30, 40){\line(0, 1){20}}
\put(10, 20){\line(0, -1){20}}
\put(15, 10){...}
\put(30, 20){\line(0, -1){20}}

\put(45, 25){.}
\end{picture}
\]
Thus $x_i(0)\Xi_i(0)=A(1, i)A(i+1, m+1)$, and hence $x_i(0)= (i!(m+1-i)!)^{-1}=c_i(0)$.

Now we determine the $x_i(k)$ for all $k>0$. By Lemma \ref{lemma:Ep-1}(3), $e_i E_i=0$.
Using \eqref{eq:AfU} in this relation,
we obtain
\[
(k+1)^2 x_i(k+1) + (i-k)! (m+1-i-k)! \zeta_{i, k} x_i(k)=0, \quad  0\le k\le i.
\]
The recurrent relation with $x_i(0)=c_i(0)$ yields $x_i(k)=(-1)^j c_i(k)$.
\end{proof}

The following result is an easy consequence of Lemma \ref{lem:Ep}.
Recall the elements $X_{s, t}\in\Sym_{s+t}$ shown in Figure \ref{X}.
\begin{corollary}\label{cor:XEX}
For all $i=0, 1, \dots, m+1$, we have $X_{i, m+1-i} E_i X_{m+1-i, i}= E_{m+1-i}$.
\end{corollary}
\begin{proof}
It is easy to show pictorially that $X_{i, M+1-i}\Xi_i(j) X_{m+1-i, i}= \Xi_{m+1-i}(j)$
for all $j\le i$. Since $c_i(j)=c_{m+1-i}(j)$, this proves the claim of the corollary.
\end{proof}

\subsection{The main theorem}

The following theorem is Theorem 4.3 in \cite{LZ4}, which is the main result of that paper.
\begin{theorem}[\cite{LZ4}] \label{thm:o-main}
The algebra map  ${F}_r^r: B_r^r(m)\longrightarrow \Hom_{\Or(V)}(V^{\otimes r}, V^{\otimes r})$
is injective if $r\le m$. If
$r>m$, the two-sided ideal $\Ker{F}_r^r$ of the Brauer algebra $B_r^r(m)$ is generated by
the element $E=E_\ell$ with $\ell=\left[\frac{m+1}{2}\right]$.
\end{theorem}
\begin{proof} Only the second part of the theorem needs explanation.
By Proposition \ref{lem:o-generat} and Corollary \ref{cor:XEX}, the elements
$E_i$ with $i=0, 1, \dots \ell=\left[\frac{m+1}{2}\right]$ generates $\Ker{F}_r^r$.
Using some general properties of the
symmetric group and Corollary \ref{cor:ann},  we showed in \cite[\S 7]{LZ4} that
$E_{i-1}$ is contained in the ideal generated by $E_i$ for each $i=1, \dots,\ell$.
The theorem follows.
\end{proof}

\section{The case of positive characteristic}\label{s:ss}

The following statement is an immediate consequence of \cite[Theorem 2.3]{RS}.
\begin{lemma}\label{lem:ss}
Let $n,r\in \Z_{>0}$. The following are equivalent for the Brauer algebras over $\Z$.
\begin{enumerate}
\item The Brauer algebra $B_r(n)$ is semisimple.
\item The Brauer algebra $B_r(-2n)$ is semisimple.
\item $r\leq n+1$.
\end{enumerate}
\end{lemma}

It follows from this that $n+1$ is the largest value of $r$ such that $B_r(n)$ and $B_r(-2n)$
are semisimple. The idempotents we have found are thus each in the `last' Brauer algebra
which is semisimple. This is in complete analogy with the situation in the Temperley-Lieb algebra
when $q$ is a root of unity, where the radical of the Jones trace function is the idempotent
corresponding to the trivial representation of the `last' semisimple Temperley-Lieb algebra
(see \cite[Cor. 3.7, Remark 3.8]{GL96}).

Note that our basic setup in this paper remains the same over the ring $\Z$ of integers.
Since we will deal with the orthogonal and symplectic groups
simultaneously in this section, we write the functor $F$ as $F_\epsilon: \cB(\epsilon m)\longrightarrow \cT_G(V)$,
and $F_k^l$ as $F_{\epsilon,k}^l$ for easy reference.
Recall that $m=\dim V$ and $\epsilon=-1$ if $G=Sp(V)$ and
$\epsilon=1$ if $G=O(V)$. We also set
$d=m/2$ if $\epsilon=-1$, and $d=m$ if $\epsilon=1$.

By Lemmas \ref{lem:Phi-Z} and \ref{lem:E-Z},
the element $\Phi$ defined by equation \eqref{Phi} and the elements
$E_k$ $(0\le k\le \left[\frac{m+1}{2}\right])$ of
Lemma \ref{lemma:Ep-1} are linear combinations of Brauer diagrams over $\Z$.

\begin{lemma}
We have $\Phi\in \Ker F_{-1, r}^r$ and $E_k\in \Ker F_{1, r}^r$ (for all $k$)
over any field $K$.
\end{lemma}
\begin{proof}
For the elements $E_k$, the claim immediately follows from their definition and
of Theorem \ref{thm:ft-Z} (2). It was also proved in \cite{LZ4}.

Next note that by Lemma \ref{lem:Phi-Z}, $\Phi$ is defined over $\Z$.
It follows from Lemma \ref{lem:ss} that $B_r(-2n)$ is semi-simple over $K$, and $\Ker F_{-1, r}^r$ is the
2-sided ideal of $B_r(-2n)$ corresponding to the one-dimensional simple module. The element $\Phi$ is
a central quasi-idempotent contained in this 2-sided ideal.
\end{proof}

The following result is a generalisation of Theorem \ref{thm:fft-sft} to fields of positive characteristic.
\begin{theorem}\label{thm:ft-Z}
Over any field $K$ of characteristic $\cp\geq m+2$,
\begin{enumerate}
\item  the functor $F_\epsilon: \cB(\epsilon m)\longrightarrow \cT_G(V)$ is full;
\item
the map $F_{\epsilon,k}^\ell$ is injective if $k+\ell\le 2d$,
and $\Ker{F}_{\epsilon,k}^\ell=\langle \Sigma_\epsilon(m+1)\rangle_k^\ell$
if $k+\ell> 2d$.
\end{enumerate}
\end{theorem}
\begin{proof}
In the orthogonal case, this was proved in \cite[Theorem 9.4]{LZ4} as an application of
\cite[Prop. 21]{Ri}. Although the symplectic case is surely in the literature, we have been unable to find it,
and therefore provide the following sketch of the argument, which may be found in \cite{ALZ}.
Note that it provides a proof of the
second fundamental theorem in positive characteristic for the symplectic groups.

Let $R=\Z[((m+1)!)\inv]$. Then we may consider the symplectic Lie algebra $\mf G_R$ over $R$,
and the corresponding $R$-forms $V_R$ and $B_R=(B_r(-m))_R$. Note that by Lemma \ref{lem:Phi-Z}
we may regard $\Phi$ as an element of $B_R$. It is shown in \cite{ALZ}, that
if $M$ is a tilting module for $\mf G_R$ and $K$ is a field with $\phi:R\to K$ a ring homomorphism,
then $\End_{\mf G_K}(M\ot_R K)\simeq \End_{\mf G_R}(M)\ot_R K$. It also follows from
{\it loc.~cit.} that $V_R\ot V_R^*$ is a tilting module. This implies (cf. \cite[Cor. 3.4]{ALZ})
that $\dim_K (B_R/\langle\Phi\rangle)\ot_R K=\dim_\C(B_r(-m)/\langle\Phi\rangle)=
\dim\End_{\mf G_K}(V_K^{\ot r})$, and the result follows.
\end{proof}

\begin{scholium}\label{cor:poschar} Let $K$ be a field of characteristic $\cp\ge m +2$.
Then the kernel of the algebra homomorphism $F_{\epsilon, r}^r: B_r(\epsilon m)
\longrightarrow \End_G(V^{\otimes r})$ as a two-sided ideal in the Brauer algebra
is generated by $\Phi$ in the case of the symplectic group $(i.e., \epsilon=-1)$,
and by $E=E_\ell$ with $\ell=\left[\frac{m+1}{2}\right]$ in the case of the
orthogonal group $(i.e., \epsilon=1)$.
\end{scholium}

\begin{remark}
Recent results of Hu and Xiao show that Scholium \ref{cor:poschar}
is valid for all fields $K$ such that $\cp>2$.
\end{remark}

\section{Quantum analogues.}

\subsection{Background}
Let $\Uq^+$ (resp. $\Uq^-$) be the quantised enveloping algebra in the sense of
\cite[\S 6]{LZ1} of the Lie algebra $\fo_m(\C)$ (see \cite[8.1.2]{LZ1} for the definition)
(resp. $\fsp_m(\C)$), over the field $\CK=\C(q)$, where in the latter case we require
that $m=2n$ is even. Write $\CA_q$ for the subring of $\C(q)$ consisting of
rational functions with no pole at $q=1$. Denote by $V_q=\CK^m$ the quantum analogue of the
natural representation of $\Uq$. The study of the endomorphism algebras
$\End_\Uq(V_q^{\ot r}$ is closely analogous to the classical case we have been
considering, which may be thought of as the limit as $q\to 1$ of the quantum case,
in a way we shall shortly make precise.

In particular, there are homomorphisms from certain specialisations of the
Birman-Murakami-Wenzl algebra $\BMW_r(q)$ to $\End_\Uq(V_q^{\ot r})$, and the classical case
is essentially the limit of the quantum case in the sense that $\lim_{q\to 1}\BMW_r(q)
=B_r$, the Brauer algebra. Let us recall the details (see \cite[\S 4]{LZ2}).
Let $y,z$ be indeterminates over
$\C$ and write $\CA=\C[y^{\pm 1},z]$.
The BMW algebra $BMW_r(y,z)$ over $\CA$ is the associative
$\CA$-algebra with generators $g_1^{\pm 1},\dots,g_{r-1}^{\pm 1}$
and $e_1,\dots,e_{r-1}$, subject to the following relations:

The braid relations for the $g_i$:
\begin{equation}\label{braidgi}
\begin{aligned}
g_ig_j&=g_jg_i\text{ if }|i-j|\geq 2\\
g_ig_{i+1}g_i&=g_{i+1}g_ig_{i+1} \text{ for }1\leq i\leq r-1;\\
\end{aligned}
\end{equation}
The Kauffman skein relations:
\begin{equation}\label{kauffman}
g_i-g_i\inv=z(1-e_i)\text { for all }i;
\end{equation}
The de-looping relations:
\begin{equation}\label{delooping}
\begin{aligned}
&g_ie_i=e_ig_i=ye_i;\\\
&e_ig_{i-1}^{\pm 1}e_i=y^{\mp 1}e_i;\\
&e_ig_{i+1}^{\pm 1}e_i=y^{\mp 1}e_i.\\
\end{aligned}
\end{equation}

The next four relations are easy consequences of the previous three.
\begin{eqnarray}
&&e_ie_{i\pm 1}e_i=e_i; \label{bmwtl}\\
&&(g_i-y)(g_i^2-zg_i-1)=0; \label{cubic}\\
&&ze_i^2=(z+y\inv -y)e_i,\quad  
\label{esquared}\\
&&-yze_i=g_i^2-zg_i-1. \label{equadg}
\end{eqnarray}

It is easy to show that $BMW_r(y,z)$ may be defined using the
relations (\ref{braidgi}), (\ref{delooping}), (\ref{cubic}) and
(\ref{equadg}) instead of (\ref{braidgi}), (\ref{kauffman}) and
(\ref{delooping}), i.e. that (\ref{kauffman}) is a consequence of
(\ref{cubic}) and (\ref{equadg}).

\subsection{Specialisations and integral forms}

Now in both the orthogonal and symplectic cases, $V_q$ is the simple
$\Uq$-module correponding to the highest weight $\ve_1$ using
the standard notation for the weights as in \cite{Bour}, and we have
the following decomposition of $V_q^{\ot 2}$:
\be\label{eq:square}
V_q\ot V_q=L_{2\ve_1}\oplus L_{\ve_1+\ve_2}\oplus L_0,
\ee
where $L_\lambda$ is the simple module corresponding to the dominant weight
$\lambda$, and $L_0$ is the trivial module. The eigenvalues of the $R$-matrix
$\check R$ on these respective components are as follows (see \cite[(6.12)]{LZ1}):
$$
\begin{aligned}
\Uq(\fo_m)&:q; -q\inv; q^{1-m}\\
\Uq(\fsp_m)&:q; -q\inv; -q^{-1-m}\\
\end{aligned}
$$

Now define two $\C$-algebra homomorphisms $\psi^{\pm}:\CA\to\CA_q$ as follows.
$\psi^+(y)=q^{1-m}$, $\psi^+(z)=q-q\inv$, $\psi^-(y)=-q^{-1-m}$, $\psi^-(z)=q-q\inv$.
We then obtain two $\CA_q$-algebras $\BMW_r^{\pm}(q):=\CA_q\ot_{\psi^{\pm}}\BMW_r(y,z)$,
and we write $\BMW_r^{\pm}(\CK):=\CK\ot_\iota\BMW_r^{\pm}(q)$, where $\iota$ is
the inclusion of $\CA_q$ into $\CK$.

It follows from \eqref{esquared} that in these two specialisations, we have
$e_i^2=\delta^{\pm}(q)e_i$, where $\delta^+(q)=[m-1]_q+1$ and
$\delta^-(q)=-([m+1]_q-1)$. Here we use the standard notation for $q$-numbers:
for any integer $t$, $[q]_t=\frac{q^t-q^{-t}}{q-q\inv}$.

It is a consequence of \cite[Theorem 7.5]{LZ1} that we have surjective homomorphisms
\be\label{eq:bmwfft}
\BMW_r^{\pm}(\CK)\overset{\eta_q}{\lr} \End_{\Uq^{\pm}}(V_q^{\ot r}).
\ee

To relate the above statement to the classical ($q=1$) case,
it was shown in \cite[\S 8.2]{LZ1} that $\Uq$ and the modules $V_q^{\ot r}$
have $\CA_q$-forms $\Uq(\CA_q)$, $V_q^{\ot r}(\CA_q)$ such that
$\Uq(\CA_q)$ acts on $V_q^{\ot r}(\CA_q)$, and the projections to the components
in \eqref{eq:square} are defined over $\CA_q$, so that the decomposition \eqref{eq:square}
is compatible with the $\CA_q$ forms. We may therefore take $\lim_{q\to 1}:=\C\ot_{\psi_1}-$
of all $\CA_q$-modules in \eqref{eq:bmwfft}, where $\psi_1:\CA_q\to \C$ takes $q$ to $1$.
It is well known that $\lim_{q\to 1}(\Uq^\ep)=\fsp_m(\C)$
if $\ep=-1$, and $\fo_m(\C)$ if $\ep=+1$, and that $\lim_{q\to 1}(\BMW_r^\ep(q))=B_r(\ep m)$.
In the proof of the next result we shall make
extensive use of the cellular structure of $\BMW_r^\ep(q)$ and its relationship to the
cellular structure of $B_r(\ep m)$, as described in \cite[Proposition 7.1]{LZ2}.

We therefore recall the following facts from {\it loc.~cit.}.
\begin{lemma}\label{lem:cell}(cf. \cite[Proposition 7.1]{LZ2})
\begin{enumerate}
\item For each $r$, the algebras $\BMW_r^\ep(q))$ and $B_r(\ep m)$ have a cellular structure
with the same cell datum $(\Lambda,M,C)$.
\item The structure constants of $B_r(\ep m)$ are obtained from those of
$\BMW_r^\ep(q))$ by putting $q=1$.
\item For each $\lambda\in\Lambda$, denote the cell module of $\BMW_r^\ep(q))$
by $W_q(\lambda)$ and that of $B_r(\ep m)$ by $W(\lambda)$. Then $W(\lambda)=
\lim_{q\to 1}W_q(\lambda)(=\C\ot_{\psi_1}W_q(\lambda)$, the Gram matrix
of the canonical form on $W(\lambda)$ is obtained from that of $W_q(\lambda)$
by setting $q=1$, as is the matrix of
of $\lim_{q\to 1}b \in B_r(\ep m)$ from that of $b$.
\end{enumerate}
\end{lemma}
The main result of this section is the following.

\begin{theorem}\label{thm:quantumlift}
(i) With notation as above, suppose $\Phi$ is an idempotent in $B_r(\ep m)$
such that the ideal $\langle \Phi\rangle$ is equal to
$\Ker(\eta:B_r(\ep m)\lr \End_G(V^{\ot r}))$. Suppose that $\Phi_q\in\BMW_r^\ep(q)$
is such that
\begin{enumerate}
\item $\Phi_q^2=f(q)\Phi_q$ where $f(q)\in\CA_q$.
\item $\lim_{q\to 1} \Phi_q=c\Phi$, where $c\neq 0$.
\end{enumerate}

Then $\Phi_q$ generates $\Ker (\eta_q:\BMW_r^\ep(q)\lr\End_\Uq(V_q^{\ot r}))$.

(ii) In the symplectic case, $\BMW_{d+1}^-(\CK)$ is semisimple, and
the kernel of $\eta_q$ is generated by the idempotent
corresponding to the trivial representation of $\BMW_{d+1}^-(\CK)$, where $m=2d$.

(iii) In the orthogonal case, there is an idempotent in $\BMW_{m+1}^+(q)$ which generates
$\Ker(\eta_q)$.
\end{theorem}

\begin{proof}
It is clear from Lemma \ref{lem:cell} that $\rank_{\CA_q}\langle\Phi_q\rangle
\geq\dim_\C\langle\Phi\rangle$ (this follows also from the fact that
$\lim_{q\to 1}\left(\BMW_r^\ep(q)\Phi_q  \BMW_r^\ep(q)\right)=B_r(\ep m)\Phi B_r(\ep m)$),
and hence that
$\dim_\CK(\BMW_r^\ep(\CK)/\langle\Phi_q\rangle)\leq\dim_\C( B_r(\ep m)/\langle\Phi\rangle)$.
It follows that if we knew that $\Phi_q\in\Ker(\eta_q)$, then
$$
\begin{aligned}
\dim_\C( B_r(\ep m)/\langle\Phi\rangle)
&\geq\dim_\CK(\BMW_r^\ep(\CK)/\langle\Phi_q\rangle)\\
&\geq\dim_\CK(\BMW_r^\ep(\CK)/\Ker(\eta_q))
=\dim_\C( B_r(\ep m)/\langle\Phi\rangle),\\
\end{aligned}
$$
whence (i) follows. Hence we turn to the proof that $\Phi_q\in\Ker(\eta_q)$.

Let $M_q=V_q^{\ot r}$, and $M=V^{\ot r}=\lim_{q\to 1}M_q$. We wish to show that
$\Phi_qM_q=0$. Now $\lim_{q\to 1}\Phi_qM_q=c\Phi M=0$. It follows that
$\Phi_qM_q\subseteq (q-1)M_q$. We shall show that $\Phi_qM_q\subseteq (q-1)^iM_q$
for each integer $i$, which will show that $\Phi_qM_q=0$.

Assume that $\Phi_qM_q\subseteq (q-1)^iM_q$; then operating by $\Phi_q$,
we obtain $\Phi_q^2M_q=f(q)\Phi_qM_q\subseteq (q-1)^{i+1}M_q$. But
$f(q)$ is not divisible by $q-1$, since $\lim_{q\to 1}\Phi_q^2=c^2\Phi
=f(1)\Phi\neq 0$. Hence $\Phi_qM_q\subseteq (q-1)^{i+1}M_q$, and it follows
by induction that $\Phi_qM_q\subseteq (q-1)^iM_q$ for all $i$, completing the proof
of (i).

(ii) We are in now the symplectic case, and by Theorem \ref{thm:sp-main},
the idempotent $\Phi\in B_{d+1}(- m)$ which corresponds to the trivial representation
generates $\Ker (\eta)$. Since the Gram matrix $G(W(\lambda))$
of the cell module $W(\lambda)$ of $B_{d+1}(- m)$ is obtained from the
Gram matrix $G(W_q(\lambda))$ of the corresponding cell module of $\BMW_{d+1}^-(q)$
by taking $\lim_{q\to 1}$, it follows that since the former is non-singular for
each $\lambda$, so is the latter. Hence $\BMW_{d+1}^-(q)$ is semisimple.
Hence there is a central idempotent $\tilde\Phi_q\in\BMW_{d+1}^-(\CK)$ which corresponds
to the trivial representation. This is characterised by the property that
$e_i\tilde\Phi_q=\tilde\Phi_qe_i=0$ and $g_i\tilde\Phi_q=\tilde\Phi_qg_i=q\tilde\Phi_q$
for all $i$. Now there is an element $f(q)\in\CA_q$ such that $f(q)\tilde\Phi_q\in
\BMW_{d+1}^-(q)$ and $f(1)\neq 0$. Write $\Phi_q=f(q)\tilde\Phi_q$.
Using an argument by descent similar to that used above,
it is easily shown that $\lim_{q\to 1}\Phi_q\neq 0$, i.e. $\Phi_q\not\in(q-1)\BMW_{d+1}^-(q)$.

If we write $\sigma_i\in B_r(-m)$ for the transposition
$(i,i+1)$, then with a slight abuse of notation, we have $\lim_{q\to 1}(g_i)=\sigma_i$
and $\lim_{q\to 1}e_i=e_i$. Taking limits, the relations above show that
$\Phi_1:=\lim_{q\to 1}\Phi_q$
is central in $B_{d+1}(-m)$ and satisfies $e_i\Phi_1=\Phi_1e_i=0$ and
$\sigma_i\Phi_1=\Phi_1\sigma_i=\Phi_1$ for all $i$. It follows that $\Phi_1=c\Phi$,
for some non-zero scalar $c$, and hence by (i), that $\Phi_q$ generates $\Ker(\eta_q)$.

(iii) In the orthogonal case, it follows from Theorem \ref{thm:o-main} that $\Ker(\eta)$
is generated by an idempotent element $\Phi\in B_{m+1}(m)$, which may be taken to be
a scalar multiple of $E_\ell$. Now $B_{m+1}(m)$ is semisimple, and hence there are
primitive central idempotents $I_1,\dots,I_s\in B_{m+1}(m)$ such that $I_1+\dots+I_s=1$.
Hence $\Phi=\Phi I_1+\dots+\Phi I_s$. Suppose without loss of generality that
$\Phi I_j\neq 0$ if $j\leq t$, and $\Phi I_j= 0$ if $j> t$. Then the ideal
genrated by $\Phi$ is equal to that which is generated by $\Psi:=I_1+\dots +I_t$.
For clearly $\langle\Phi\rangle\subseteq\langle I_1+\dots + I_t\rangle$, but conversely
if $\Phi I_j\neq 0$, the two sided ideal generated by $\Phi I_j$ includes the simple ideal
generated by $I_j$, and hence $I_j$ itself. So
$\langle I_1+\dots+I_t\rangle\subseteq\langle I_1,\dots,I_t\rangle\subseteq\langle\Phi\rangle$.

We shall show that there is an element $\Psi_q\in\BMW_{m+1}^+(q)$
with properties analogous to those of $\Phi_q$ in (i), but for $\Psi$.
First observe that by the same agrgument
as in (ii) (using Lemma \ref{lem:cell}) the algebra $\BMW_{m+1}^+(q)$, and hence
$\BMW_{m+1}^+(\CK)$, whose cell modules have the same Gram matrices, are semisimple.
It follows that there are unique primitive central idempotents $\Phi_{1,q},\dots,\Phi_{t,q}
\in\BMW_{m+1}^+(\CK)$ which correspond to the same cells as $\Phi_1,\dots,\Phi_t$ respectively
(recall that $\Lambda$ parametrises the cells of both $\BMW_{m+1}^+(\CK)$ and $B_{m+1}(m)$,
and hence also their minimal two-sided ideals). For each $i$, there is an element $f_i(q)\in\CA_q$
such that $f_i(q)\Psi_{i,q}\in\BMW_{m+1}^+(q)$. Using the same argument as in (ii),
$f_j(q)$ may be chosen so that $\lim_{q\to 1}(f_j(q)\Psi_{j,q})=f_j(1)I_j\neq 0$.
Then ${\Psi}_q:=f_1(q)f_2(q)\dots f_t(q)(\Phi_{1,q}+\dots+\Phi_{t,q})\in \BMW_{m+1}^+(q)$,
and satisfies:(i) $\Psi_q^2=F(q)\Psi_q$, where $F(q)=f_1(q)\dots f_t(q)$ and
(ii) $\lim_{q\to 1}(\Psi_q)=F(1)\Psi$. It now follows from (i) that $\Psi_q$ generates
$\Ker(\eta_q)$.
\end{proof}

We remark finally that Hu and Xiao \cite{HX} have also contributed to the subject of this section.

\subsection{Further comments}\label{sect:quantum}

The invariant theory of quantum groups \cite{D, L} in a broad sense has been extensively studied.
One aspect of it is the quantum group theoretical construction \cite{R, RT, ZGB}
(see \cite{T2} for a review)
of the Jones polynomial of knots \cite{J} and its cousins.
It was in this context that the braided monoidal category
structure of the category of quantum group representations rose to prominence.

In the quantum case, the right replacement of the category of Brauer diagrams
is the category of (nondirected) ribbon graphs \cite{RT, T2},
also known as the category of framed tangles.
The Reshetikhin-Turaev functor \cite{RT} gives rise to a full tensor functor from this
category to the category of tensor representations of the symplectic quantum  group,
or the orthogonal quantum group defined in \cite{LZ1}.
This is the quantum analogue of Theorem \ref{thm:fft-sft}(1).

The FFT of invariant theory for quantum groups is best understood
in terms of endomorphism algebras (see e.g.,  \cite{DPS, LZ1}).
However, in order to establish a quantum analogue of FFT in the polynomial formulation,
one has to go beyond commutative algebra and consider quantum group actions on
noncommutative algebras. This was developed in \cite{LZZ}.

\end{document}